\documentclass[leqno]{imsart}
\usepackage{amssymb,amsfonts,amsmath,amsthm}
\usepackage{enumitem}
\usepackage[margin=1in]{geometry}
\usepackage[colorlinks]{hyperref}
\usepackage{ulem}

\DeclareFontFamily{U}{mathx}{\hyphenchar\font45}
\DeclareFontShape{U}{mathx}{m}{n}{
      <5> <6> <7> <8> <9> <10>
      <10.95> <12> <14.4> <17.28> <20.74> <24.88>
      mathx10
      }{}
\DeclareSymbolFont{mathx}{U}{mathx}{m}{n}
\DeclareFontSubstitution{U}{mathx}{m}{n}
\DeclareMathAccent{\widecheck}{0}{mathx}{"71}

\numberwithin{equation}{section}
\newtheorem{theorem}{Theorem}
\makeatletter \@addtoreset{theorem}{section}\makeatother

\newtheorem*{theorem*}{Theorem}
\newtheorem{proposition}[theorem]{Proposition}
\newtheorem{prop}[theorem]{Proposition}

\newtheorem{corollary}[theorem]{Corollary}
\newtheorem{remark}[theorem]{Remark}
\newtheorem*{remark*}{Remark}
\newtheorem{lemma}[theorem]{Lemma}
\newtheorem*{lemma*}{Lemma}
\newtheorem*{claim*}{Claim}

\newtheorem{query*}[theorem]{Query*}

\newcommand{\MF}{{\mathcal{M}_F}}
\newcommand{\Rd}{ {\mathbb{R}}^d}
\newcommand{\Rtwo}{ {\mathbb{R}}^2}
\newcommand{\BMr}{ {B^N_r}}
\newcommand{\BMrx}{ {B^N_r(x)}}

\newcommand{\BMrz}{ {B^N_r(z)}}
\newcommand{\BNr}{ {B^N_r}}

\newcommand{\1}{\mbox {\bf 1}}

\newcommand{\nn}{{\mathbb N}}

\newcommand{\cala}{{\mathcal A}}

\newcommand{\calf}{{\mathcal F}}

\newcommand{\cale}{{\mathcal E}}
\newcommand{\cali}{{\mathcal I}}

\newcommand{\veps}{\varepsilon}
\newcommand{\vep}{\varepsilon}
\newcommand{\beq}{\begin{eqnarray*}}
\newcommand{\feq}{\end{eqnarray*}}
\newcommand{\beqn}{\begin{eqnarray}}
\newcommand{\feqn}{\end{eqnarray}}

\makeatletter \@addtoreset{theorem}{section}\makeatother

\newcommand{\ds}{\displaystyle}

\newcommand{\Z}{{\mathbb Z}}
\newcommand{\N}{{\mathbb N}}

\newcommand{\R}{{\mathbb R}}

\newcommand{\cF}{{\mathcal F}}
\newcommand{\cM}{{\mathcal M}}

\newcommand{\txi}{\tilde \xi}
\newcommand{\tgamma}{{\tilde \gamma}}

\begin{document}
\begin{frontmatter}

\title{Rescaling  the spatial lambda\\
  Fleming-Viot process and convergence to super-Brownian motion}
\runtitle{Convergence of SLFV to SBM}

\author{\fnms{J.\ Theodore}
  \snm{Cox}\thanksref{t1}\corref{}\ead[label=e1]{jtcox@syr.edu}} 
\address{Department of Mathematics\\ Syracuse University\\
  \printead{e1}}
\author{\fnms{Edwin A.} \snm{Perkins}\thanksref{t2}
\ead[label=e3]{perkins@math.ubc.ca}}
\address{Department of Mathematics\\
The University of British Columbia\\
\printead{e3}}
\thankstext{t1}{Supported in part by an NSERC Discovery Grant}
\thankstext{t2}{Supported in part by an NSERC Discovery Grant}

\runauthor{Cox and Perkins}

\begin{abstract}
We show that a space-time rescaling of the spatial
Lamba-Fleming-Viot process of Barton and Etheridge converges to
super-Brownian motion. This can be viewed as an extension of a result of
Chetwynd-Diggle and Etheridge \cite{CDE}. In that work the scaled impact
factors (which govern the event based dynamics) vanish in the limit,
here we drop that requirement. The analysis is particularly interesting
in the biologically relevant two-dimensional case.
\end{abstract}




\end{frontmatter}

\section{Introduction}
Our purpose in this paper is to extend a result in
  \cite{CDE} which shows that certain suitably rescaled
  spatial Lambda-Fleming-Viot (SLFV) processes converge weakly
  to super-Brownian motion (SBM). Our extension is analogous to that of allowing 
  nearest neighbour interactions in interacting particle models, as opposed to taking
  long range limits, and is particularly delicate in the critical two-dimensional case. SBM 
is a well known measure-valued diffusion, introduced in
  \cite{W68} and \cite{D75}, for which there is an
extensive research literature (e.g., for reviews see \cite{D93}, \cite{E00} and 
  \cite{P02}). SLFV processes were introduced
more recently, in \cite{E08}, to serve as models for
the evolution of allele frequencies in populations
  distributed across spatial continua. An analytic
  construction was given in \cite{BEV}, along with a
  discussion of the biological significance of the model.  A
  more probabilistic construction was given in \cite{VW},
  one which gives a very useful connection between SLFV
  processes and their duals.  Following \cite{CDE}, we
  consider here a neutral two-type version of the
  general SLFV model, taking ``space'' to be
  $\Rd$. Informally, our process (constructed below) is
  a Markov process 
  $(\mu_t)_{t\ge 0}$ where for each $x\in\Rd$, $\mu_t(x)$ is
  a probability distribution on the type space $\{0,1\}$, with the
  interpretation that $\int_B\mu_t(x)(\{i\})dx$ represents the
  proportion of the population of type $i$ in a
  region $B\subset\Rd$ at time $t$. 
 We will consider an extension of the fixed radius case from \cite{CDE} (Theorem 2.6 of that reference) and not the 
interesting variable radius case, also discussed there in Theorem 2.7, in which stable branching arises in the limit.

SBM arises as the limit under Brownian space-time rescaling of a range of critical spatially interacting models in mathematical physics and biology {\it above the critical dimension} including critical oriented percolation \cite{HS03}, critical lattice trees \cite{HHP17}, the critical contact process \cite{HS10}, and the voter model \cite{CDP}; it is believed to be the scaling of critical ordinary percolation in the same regime.  The only scaling limit of the above which
has been verified {\it at the critical dimension} is the voter model \cite{CDP} where the critical dimension is two. In this case the simple nature of the dual process, a coalescing random walk, allows one to carry out the required explicit calculations. Now our challenge is to use the related but more complex dual of the Barton-Etheridge model to carry through the analysis.  It is understood here that we are not taking ``long-range" limits (e.g. as was done for the contact process in \cite{DP}) which will weaken the interaction and make the analysis considerably easier. In our setting this means not letting the impact factor (described below) approach zero in the rescaling.
  
  We start by recalling the definition of the fixed
  radius SLFV process
  given in 
\cite{CDE}.  Let $r>0$ be the 
``interaction radius'', let $\rho\in[0,1]$ be the ``impact factor,''
and let $\Pi$ be a Poisson point process 
on $\Rd\otimes (0,\infty)$ with intensity $dx\otimes dt$. We suppose the distribution of types in the population changes
over time according to ``reproduction events'' determined by $\Pi$.
Given $\mu_{t-}$, if $(x,t)\in\Pi$,
choose an independent point $z$ uniformly
at random from the Euclidean ball $$B_r(x)=\{y:|y-x|\le
r\},$$ and (independently) a type $\alpha$ according to 
the distribution $\mu_{t-}(z)$, and then set
\[
\mu_t(y) = (1-\rho)\mu_{t-}(y) + \rho\delta_{\alpha} \quad
\forall y\in B_r(x).
\]
We keep $\mu_t(y)=\mu_{t-}(y)$ for $y\notin B_r(x)$. 
Writing $\mu_t(x)$ in the form
$w_t(x)\delta_1+(1-w_t(x))\delta_0$, we can reformulate the
above dynamics more conveniently in terms of $w_t$ as
follows. Starting from a
Borel $w_0:\Rd\to [0,1]$ with compact support, for
$(x,t)\in\Pi$, 
choose an independent parental location $z$ uniformly at random
from $B_r(x)$, independent of everything,  and then: 
\begin{equation}\label{e:dynamics}
\begin{aligned}
\text{(i) }&\text{ with probability $w_{t-}(z)$ put
  $w_{t}(y)=(1-\rho)w_{t-}(y) + \rho$ for all  $y\in B_r(x)$,}\\
\text{(ii) }&\text{ with probability $1-w_{t-}(z)$ put
  $w_{t}(y)=(1-\rho)w_{t-}(y)$ for all $y\in B_r(x)$,}\\
\text{(iii) }& \text{ for all $y\notin B_r(x)$ keep $w_t(y)=w_{t-}(y)$.}
\end{aligned} 
\end{equation}
As noted in Section~3 of \cite{CDE}, this description gives
a well-defined $w_t:\Rd\to[0,1]$ which has compact support
at all times. (See \cite{VW} for more details on the
  construction.)  It will be useful to regard $w_t$ as the
measure $w_t(x)dx$, and for bounded Borel
$\phi:\Rd\to[0,\infty)$, write
\begin{equation}\label{wphidef}
w_t(\phi) = \int_{\Rd}\phi(x)\,w_t(x)dx .
  \end{equation}

  Closely associated with the process $w_t$ is a dual
    process of coalescing ``lineages''. If we sample a
    finite number of spatial locations $\{x_i\}$ at time $T$, it
    is easy to see that the values $w_T(x_i)$ can be
    determined from $w_0$ by using $\Pi$ to trace the
    lineages backward in time. Since $\Pi$ run backwards is still a Poisson
    process, we may define a version of the
    lineages process starting at backwards time 0 from a finite number
    of locations $\{x_i\}$ as follows.  If $(x,t)\in \Pi$, mark
    each lineage in $B_r(x)$ independently with probability
    $\rho$, and choose a point $z$ uniformly at random from
    $B_r(x)$. If at least one of the lineages in $B_r(x)$ is
    marked, all marked lineages in $B_r(x)$ coalesce and the
    resulting lineage is moved to $z$. If no lineage is
    marked, no lineage moves. Lineages outside of $B_r(x)$
    are not affected. In this paper it will suffice to
    consider only the one and two-lineage systems, so we
    will ignore the higher lineage systems which are more
    complex to analyze. We now give a more precise
    description of these Markov jump
    processes, using the language of ``particles''
    instead of lineages.

Let $|\Gamma|$ be the 
Lebesgue measure of  
$\Gamma\subset\Rd$.
Let $U,U^1, U^1$ be independent random variables
uniformly distributed on $B_r=B_r(0)$, and let $\bar U$ 
have the law of $U^1+U^2$, i.e., $\bar U$ has density
\begin{equation}\label{e:barUdens}
P(\bar U\in dz) = \frac{|B_r(0)\cap B_r(z)|}{|B_r(0)|^2} dz:=h_{\bar U}(z)dz.
\end{equation}
We let $\bar\sigma^21_{d\times d}$ denote the covariance matrix of $\bar U$, so that if $x=(x_1,\dots,x_d)$, then
\begin{equation}\label{barsigdef}
\bar\sigma^2=\frac{2}{|B_r|}\int_{B_r} (x_1)^2\,dx.
\end{equation}
We will use this notation throughout, along with $\eta_t$ for
the single particle dual and 
$\xi_t=(\xi^1_t,\xi^2_t)$ for the two particle dual. 

(a) {\it The single-particle dual $\eta_t$.} If we start
with a single particle at $x$, it is easy to see that  
$\eta_t$ is the random walk on $\Rd$ starting at $x$ which
makes jumps at rate $\rho|B_r|$ with jump distribution given
in \eqref{e:barUdens}. We write $P_{\{x\}}$ for the underlying law of $\eta$.

(b) {\it The two-particle dual $(\xi^{1}_t,\xi^{2}_t)$.}
If we start with two particles, one at $x_1$ and
the other at $x_2\ne x_1$, 
$(\xi^{1}_t,\xi^{2}_t)$ is the
Markov jump process starting  at $(x_1,x_2)$,  and with law $P_{\{x_1,x_2\}}$, which makes  
transitions
\begin{equation}\label{e:d2rates}
(y_1,y_2) \to
\begin{cases} 
(y+ \bar U, y+\bar U) &\text{ at rate }\rho|B_r| \text{ if }y_1=y_2=y\\
(y_1+\bar U,y_2)&\text{ at rate }
\rho(|B_r|-\rho|B_r(y_1)\cap B_r(y_2)|) \text{ if }y_1\ne y_2\\
(y_1,y_2+\bar U)&\text{ at rate }
\rho(|B_r|-\rho|B_r(y_1)\cap B_r(y_2)|) \text{ if }y_1\ne y_2\\
(U+U_{y_1,y_2},U+U_{y_1,y_2}) &\text{ at rate }
\rho^2|B_r(y_1)\cap B_r(y_2)| \text{ if }y_1\ne y_2,
\end{cases}
\end{equation}
where $U_{y_1,y_2}$ is an independent random variable, 
uniformly distributed over $B_r(y_1)\cap B_r(y_2)$. 
For $y_1\ne y_2$, the total jump rate
at $(y_1,y_2)$, $y_1\ne y_2$, is $
2\rho|B_r|-\rho^2|B_r(y_1)\cap B_r(y_2)|$.
To see the above rates consider, for example, the second transition from $(y_1,y_2)$ to $(y_1+\bar U,y_2)$ for $y_1\neq y_2$ where $(y_1,y_2)$ is the current site of our two-particle dual. 
The next jump in the first coordinate can only occur at a point $(x,t)\in\Pi$ with $x\in B_r(y_1)$ so let $(x,t)$ be the next such point.  
At $(x,t)$ such a jump (affecting the first coordinate but not the second) can occur in one of two ways:  if $x$ lands in $B_r(y_1)\setminus (B_r(y_1)\cap B_r(y_2))$ and the particle $\xi^1$ at $y_1$ is marked, or if $x$ lands in $B_r(y_1)\cap B_r(y_2)$ and the particle at $y_1$ is marked and the particle at $y_2$ is not.  The total rate in $t$ is obtained by integrating out $x$ and so is
\[\rho(|B_r(y_1)|-|B_r(y_1)\cap B_r(y_2)|)+\rho(1-\rho)|B_r(y_1)\cap B_r(y_2)|=\rho|B_r|-\rho^2|B_r(y_1)\cap B_r(y_2)|.\]
In either of the above scenarios the particle at $y_1$ will jump to $z$, a uniformly selected site in $B_r(x)$.  Given $y_1$, $x$ will be uniformly distributed on $B_r(y_1)$ and so $x-y_1$ will be uniform on $B_r$.
Clearly given $(y_1,x)$, $z-x$ is uniformly distributed over $B_r$ and so $(x-y_1,z-x)$ is a pair of independent uniforms on $B_r$.  Therefore the jump in $\xi^1$ at time $t$ is $z-y_1=(z-x)+(x-y_1)$ and so has law $\bar U$ as claimed. The other transitions are similar to analyze.

The coalescence time  for the two-particle dual starting at
$(x_1,x_2)$ is
\begin{equation}\label{e:tauc}
\tau = 
\inf\{t\ge 0:\xi^{1}_t=\xi^{2}_t\}.
\end{equation}
Although $(\xi^{1}_t,\xi^{2}_t)$ is 
Markov, the individual coordinates $\xi^{1}_t,
\xi^{2}_t$ are not.
However, when 
$B_r(\xi^{1}_t)\cap B_r(\xi^{2}_t)  =\emptyset$, both
coordinates move independently according to the single particle dynamics,
while for $t> \tau$, the coalesced coordinates move together
according to the single particle dynamics. It
is also clear from \eqref{e:d2rates} that  
the two-particle dual is translation
invariant, that is,
\begin{equation}\label{TI}P_{\{x_1+x,x_2+x\}}((\xi^1,\xi^2)\in\cdot)
=P_{\{x_1,x_2\}}((x+\xi^1,x+\xi^2)\in\cdot)\quad\forall x,x_1,x_2\in\R^d.
\end{equation}

The two special cases of the general duality
equation in Proposition~2.5 of \cite{CDE} that we need are the following.
For all $t\ge0$, $\psi_1\in C(\Rd)\cap 
L^1(\Rd)$ and $\psi_2\in C(\Rd\times\Rd)\cap
L^1(\Rd\times\Rd)$, 
\begin{align}\label{e:dual1}
&E_{w_0} [w_t(\psi_1)]  = \int_{\R^d} \psi_1(x) E_{\{x\}}[
w_0(\eta_t)]dx,  \text{ and }\\\notag
&E_{w_0}\Big[\int_{ \R^d\times \R^d}
\psi_2(x_1,x_2)w_t(x_1)w_t(x_2)\, dx_1dx_2\Big]\\
&\qquad\qquad = \int_{ \R^d\times \R^d} \psi_2(x_1,x_2)E_{\{x_1,x_2\}}\big[
w_0(\xi^{1}_t)1_{\{\tau\le t\}} + 
w_0(\xi^{1}_t)w_0(\xi^{2}_t)1_{\{\tau>t\}}\big]\,
dx_1dx_2 .
\label{e:dual2}
\end{align}
By standard approximation arguments, these equations then
  hold for all Borel $\psi_1,\psi_2$ which are either
  nonnegative or integrable (on one side or the other). In particular,  letting $\1$ denote the constant function 1 on
$\Rd$, we have
\begin{equation}\label{e:mass}
E_{w_0}[w_t(\1)] =  \int_{\Rd}E_{\{x\}}[w_0(\eta_t)]dx = w_0(\1).
\end{equation}

Before stating Theorem 2.6 in \cite{CDE} we introduce
super-Brownian motion using the martingale problem
formulation.  If $(X_t)_{t\ge 0}$ is a stochastic process,
$(\cF^X_t)_{t\ge 0}$ will denote the right-continuous filtration
generated by $X$.  Let $\MF(\Rd)$ denote the space of finite
Borel measures on $\Rd$ endowed with the topology of weak
convergence, and for $\mu\in\MF(\Rd)$ let
$\mu(\phi)=\int_{\Rd}\phi d\mu$. The space of bounded continuous
functions on $\R^d$ is denoted by $C_b(\R^d)$, and
$C_0^3(\R^d)$ is the space of continuous functions on $\R^d$
which vanish at infinity and have bounded continuous
partials of order $3$ and less.  Then (see, e.g.,
Theorem~A.1 of \cite{CDP} for uniqueness, and Theorem~II.5.1 and Remark~II.5.5 of \cite{P02} for existence) Super-Brownian motion with
diffusion coefficient $\sigma^2$ and branching rate $b$,
denoted
SBM($X_0,\sigma^2,b$), 
is the
unique $\MF(\Rd)$-valued Markov
process $(X_t)_{t\ge 0}$ with continuous paths and 
initial state $X_0$,   such that 
such that for
every $\phi\in C^3_0(\Rd)$, 
\begin{equation}
M_t(\phi) = X_t(\phi) -X_0(\phi) - \int_0^t X_s \Big(\frac{\sigma^2}2 
\Delta \phi\Big) ds
\end{equation}
is a local $(\cF^X_t)$-martingale
  with predictable quadratic variation process
\begin{equation}\label{SBMsqfn}
\langle M(\phi)\rangle_t = b \int_0^t X_s(\phi^2) ds. 
\end{equation}

Theorem~2.6 in \cite{CDE} considers a
sequence $X^N_t=K w^N_t$ of scaled versions 
of $w_t$, defined with sequences $K=K_N$, $M=M_N$, and
$J=J_N$. Namely, given $w^N_0$, $w^N_t$ is constructed in the same way as
$w_t$, but with the following modifications. If
$(x,t)\in\Pi$ then there is a reproduction
event at $(\frac xM,\frac tN)$, with 
impact factor $\frac \rho J$ and reproduction region
$B_{\frac rM}(\frac xM)$. Thus, for $w^N_t$, time is sped up by 
$N$, space is shrunk by $M_N$, and the impact factor is
reduced by $J_N\ge 1$. Denote the rescaled Poisson point process with intensity
$M^dN\,dx\otimes dt$ by
\[\Pi^N=\Big\{\Big(\frac{x}{M},\frac{t}{N}\Big):(x,t)\in\Pi\Big\}.\]
If $n_t^N(A)=\#\{(s,x)\in \Pi^N\cap([0,t]\times A)\}$, let $(\cF^N_t)$ be the right-continuous
filtration generated by $\{n_s^N(A):s\le t, A\text{ a Borel
  set in }\R^d\}$.
  
 Let
$D([0,\infty),\cM_F(\R^d))$ denote the space of cadlag 
$\cM_F(\R^d)$-valued paths equipped with the Skorokhod (J1) topology.
\setcounter{theorem}{-1}
\begin{theorem}[Theorem~2.6 in \cite{CDE}]\label{t:CDE}
Suppose that for a compact set $D_0\subset\Rd$,
$\text{supp}(w^N_0)\subset D_0$ for all $N$, 
and as elements of $\MF(\Rd)$, 
$X^N_0\to X_0\in\cM_F(\R^d)$ as $N\to\infty$.
In addition, suppose there are constants
$C_1,C_2\in(0,\infty)$ such that, as $N\to\infty$,
\begin{enumerate}[label=(\arabic*)]
\item $M\to\infty$,
\item $\frac{N}{JM^2} \to C_1$,
\item $\frac{KN}{J^2M^d}\to C_2$,
\item
$\ds 
\begin{cases}
\frac MJ \to 0 &\text{if }d=1,\\
\frac{\log M}J \to 0 &\text{if }d=2,\\
\frac{1}{J} \to 0 &\text{if }d\ge 3.
\end{cases}$
\end{enumerate}
Then the sequence $(X^N)_{N\ge 1} $ converges weakly in $D([0,\infty),\cM_F(\R^d))$ to SBM($X_0,\sigma^2,b$)
with
\[
\sigma^2=2C_1\rho \int_{B_r}(x_1)^2 dx
\text{ and }
b =  C_2\rho^2|B_r|^2.
\]
\end{theorem}
(The constant $C(d)$ in Definition~4.1 in \cite{CDE} should
be $C(d) = \int_{B_1}(x_1)^2dx$.)  As noted in \cite{CDE},
this result is similar in spirit to Theorem 1.1 in
\cite{CDP}, which proves convergence to SBM for certain
sparse ``long range'' kernel voter models. Due to conditions (1) and (4)
above, $J_N\to\infty$ and hence the impact factors
$\rho/J_N\to 0$. It is this fact  which makes these SLFV
processes analogous to the ``long range'' voter models in
\cite{CDP}. Keeping $J_N$ bounded, so that the impact
factors $\rho=\rho/J_N$ do not vanish in the limit, would
correspond to the ``fixed'' kernel voter models in
Theorem~1.2 in \cite{CDP}.  In biological terms this
corresponds to keeping the ``neighbourhood size'' finite
in the scaling limit, while letting $J_N\to\infty$
effectively allows this parameter to become infinite; see
the discussion in Section~2 of \cite{AFPS} and especially
Definition 2.2 there. In that work they showed in this fixed
neighbourhood size setting (Theorem 2.7 of \cite{AFPS})
that, with an appropriate selection term, the dual particle
process converges to a branching Brownian motion in the
scaling limit.  The purpose of this paper is to prove that
in this setting, with no selection, there is also a
forwards limit theorem giving convergence to SBM.

For $d=1$, in \cite{AVU18} the Wright-Fisher SPDE was obtained as an appropriate scaling limit of SLFV but
with the impact factors scaled to approach zero like $N^{-1/3}$ (see \cite{MT} for the 
corresponding scaling limit for the voter model). In this
setting the strong recurrence of one-dimensional random walk
leads to heavy clustering which means scaling limits with
bounded neighbourhood size lead to segregation of types; the
corresponding scaling limit for the voter model is the
Arratia flow \cite{A81}.

{\bf Throughout this work we will  assume $d\ge 2$ and 
\begin{equation}\label{Nrange}
N\in[3,\infty).
\end{equation}}
For our scaled SLFV processes $w^N_t$ we then make the
choices
\begin{equation}\label{e:JMK}
J =1, \quad
M=\sqrt N, \quad
K = \begin{cases} 
N^{\frac d2 -1}&\text{for }d\ge 3,\\
\log N &\text{for }d=2.
\end{cases}
\end{equation}
If we set $J\equiv 1$, and take $C_1=C_2=1$ for
simplicity, the conditions (1)-(3) in  
Theorem~\ref{t:CDE} suggest the choices for $M$ and $K$ above
except for the logarithmic correction to $K$ for 
$d=2$. Without this
correction, one can show that the limiting process  in Theorem~{\ref{t:main}} would
be nonrandom heat flow acting on $X_0$, as is the case for  the voter model \cite{PS83}.

In order to state our limit theorem for scaled SLFV
processes assuming \eqref{e:JMK}, we must first identify
certain constants $\gamma_e^{(d)}$ that appear in the
limiting SBM branching rate. These constants are determined
by the asymptotic tail behavior of the coalescence times
$\tau$ for the two-particle dual process.
Introduce 
\begin{equation*} \ds 
\gamma_e(t)=\gamma_e^{(d)}(t)=
\frac{1}{|B_r|^2} \int_{B_r}
\int_{B_r}    P_{\{x_1,x_2\}}(\tau>t) dx_1dx_2.
\end{equation*}

\begin{proposition}\label{t:noncoal}
There are constants
  $\gamma_e= \gamma_e^{(d)}>0$ such that as $t\to\infty$, 
  \begin{equation}\label{e:noncoal3d}
\lim_{t\to\infty}\gamma_e^{(d)}(t)=\gamma_e \text{ if }d\ge 3
  \end{equation}
  and
  \begin{equation}\label{e:noncoal2d}
\lim_{t\to\infty}(\log t)\,\gamma_e^{(d)}(t)=\gamma_e \text{ if }d=2.
  \end{equation}
\end{proposition}

\bigskip Recall that when outside $B_{2r}$,
$\xi_t^1-\xi_t^2$ behaves like a rate $2\rho|B_r|$ random
walk with jump distribution given in \eqref{e:barUdens}, and
$\tau = \inf\{t\ge0: \xi_t^1-\xi_t^2=0\}$.
Therefore, if $d\ge 3$, the difference will escape to
infinity with positive probability by transience, and so the
limit in \eqref{e:noncoal3d}, which exists by monotonicity,
will have a non-zero limit.  For $d=2$ the situation is more
delicate.  One can
predict the $1/\log t$ behaviour of $\gamma_e(t)$ from the
corresponding non-return probabilities for irreducible
symmetric random walk on $\Z^2$ with diagonal covariance matrix (see, e.g., Lemma A.3(ii)
of \cite{CDP}), but the slowing rates when the difference 
$\xi^1-\xi^2$ is in $B_{2r}$ complicates things. The limit
\eqref{e:noncoal2d} can be derived from
Lemma~4.10 in \cite{AFPS}. The analysis there is based on a
construction using successive ``inner'' and ``outer''
excursions of $\xi^1-\xi^2$ from certain balls
before coalescence occurs. Our argument represents 
the difference process as a time change of a
rate $2\rho |B_r|$ random walk with step distribution
$h_{\bar U}$, and makes use of a reflection coupling. We feel the proof is
of independent interest and so 
have included it in an 
Appendix.  One advantage of the excursion approach in \cite{AFPS} is that it should also allow inclusion of 
a random ``interaction radius", that is 
\begin{align}\label{varradcase}&\text{the driving Poisson point process $\Pi$ is now on $\R^d\times [0,\infty)\times [0,r_{\text{max}}]$ with intensity}\\
\nonumber&\text{ $dx\otimes dt\otimes \mu(dr)$ for some finite measure $\mu$ on the compact interval $[0,r_{\text{max}}]$.}
\end{align}
  However, as is discussed below, our time-change representation of the dual difference process in the fixed radius case will also play an important role in the analysis of the martingale square function which is the key ingredient in the proof of our main convergence result, Theorem~\ref{t:main} below.

  With the choice of renormalization constants in
  \eqref{e:JMK} we now give a different description of the
  rescaled SLFV processes $X^N$, which will clarify the
    comparison with Theorem~\ref{t:main} below with the
    fixed kernel result in \cite{CDP}. Assume
  $X_0\in \cM_F(\R^d)$ and the compactly supported initial
  conditions $w_0^{(N)}:\R^d\to[0,1]$ satisfy
\begin{equation}\label{ICscaling}
K_Nw^{(N)}_0(\sqrt N x)dx\to X_0\text{ in }\cM_F(\R^d).
\end{equation}
For each $N$, let $w^{(N)}$ be the (original, unscaled)
SLFV process with initial condition $w^{(N)}_0$ and fixed
interaction radius $r$ and impact factor $\rho$, and define
the rescaled SLFV process by
\begin{equation}\label{wNdef}
w^N_t(x)=w^{(N)}_{Nt}(\sqrt N x).
\end{equation}
  This process has the same law as
  $w^N$ defined right before Theorem~\ref{t:CDE}. (For example,
  if $(x,s)$ is the first point in $\Pi$ affecting
  $w^{(N)}$, then $w^{(N)}_s\ne w^{(N)}_{s-}$ is only
  possible inside $B_r(x)$. Letting $Nt=s$, this means that
  $w^N_t(y)\ne w^N_{t-}(y)$ is possible only for
  $y\sqrt N\in B_r(x)$, or $y\in B_{r/\sqrt N}(x)$. Thus the
  interaction radius for $w^N$ is $r/\sqrt N$.) Finally our
  approximating empirical measures are given by
\begin{equation}\label{XNdef}
X^N_t(dx)=K_Nw^N_t(x)dx:=X^N_t(x)dx\in\cM_F(\R^d),
\end{equation}
so that \eqref{ICscaling} just asserts that $X_0^N\to X_0$.
A simple change of variables shows that in terms of our original SLFV processes, $w^{(N)}$, we have for any bounded Borel $\phi$ on $\R^d$,
\begin{equation}\label{XNdef2}
X^N_t(\phi)=\begin{cases}\frac{1}{N}\int_{\Rd}
w_{Nt}^{(N)}(y)\phi(yN^{-1/2})\,dy&\text{ if }d\ge 3,\\ 
\frac{\log N}{N}\int_{\Rd} w_{Nt}^{(N)}(y)\phi(yN^{-1/2})\,dy&\text{ if }d=2.\end{cases}
\end{equation}

Here is our main result for the scaled SLFV process.  For a
measure or function $H$, we let $\text{supp}(H)$ denote its
closed support.  Recall the definition of 
$\bar\sigma^2$ from \eqref{barsigdef}.
\begin{theorem}\label{t:main} Suppose that $d\ge2$, 
$\text{supp}(w^N_0)$ is compact for each $N$,  and
$\lim_{N\to\infty}X^N_0=X_0$ in $\cM_F(\R^d)$.
As $N\to\infty$, $X^N$ converges
  weakly in $D([0,\infty),\cM_F(\R^d))$ to SBM($X_0,\sigma^2,b$), where
\begin{equation}
\sigma^2=\rho |B_r|\bar\sigma^2
\text{ and }
b = \rho^2|B_r|^2\gamma_e.
\end{equation}
\end{theorem}

As is usual for SBM limit theorems, the scaling condition on
the initial conditions in \eqref{ICscaling} corresponds to a
regime where type 1's are scarce. For example,
  if, say, $X_0$ assigns no mass to the boundary of $[-1,1]^d$, then taking
$\phi=1(\Vert x\Vert_\infty\le 1)$ we have $X^N_0(\phi)\to
X_0([-1,1]^d)$ which implies 
\[\frac{\int_{[-\sqrt N,\sqrt N] ^d}w_0^{(N)}(x)dx}{(2\sqrt N)^{d}}\sim \frac{X_0([-1,1]^d)}{2^{d/2}K_N}\to 0\text{ as }N\to\infty.\]
It is important to note that in our scaling regime with
$J=1$ the original $w^{(N)}$ we are working with is an
ordinary SLVF process with fixed interaction range $r$ and
impact factor $\rho$, but with an initial condition in which
type $1$'s are scarce.

\eqref{XNdef2} should be compared to the corresponding rescaled empirical measures  in \cite{CDP} associated with a sequence of voter models $\xi^{(N)}_t(x), x\in\Z^d$ whose rescaled initial states again converge to a limiting $X_0\in\cM_F(\R^d)$: if $N'=N$ for $d\ge 3$ and $N'=N/\log N$ if $d=2$, let
\begin{equation*}
\hat X^N_t=\frac{1}{N'}\sum_{x\in \Z^d}\xi^{(N)}_{Nt}(x)\delta_{xN^{-1/2}}.
\end{equation*}
In that reference it is shown that $\hat X^{N}$ converges
weakly in $D([0,\infty),\cM_F(\R^d))$ to an appropriate SBM,
whose branching rate  is determined by the  asymptotics of the
escape probability (from $0$) for a continuous time random
walk starting at a uniformly chosen neighbour of $0$ in the
integer lattice  through a two-particle dual calculation.
This suggests the same should hold (as it does) for the SLFV but now with
the asymptotics of the non-coalescing probability of our two
particle dual playing the role of the random walk escape
probability. 

The proof follows a familiar outline, based in part on methods in \cite{CDP}.
For appropriate test functions
$\phi$ the semimartingale
decomposition from \cite{CDE} states that
\[
X^N_t(\phi) = X^N_0(\phi) + 
D^N_t(\phi) + M^N_t(\phi),
\]
where $M^N_t(\phi)$ is a local martingale, and both the drift
$D^N_t(\phi)$ and predictable 
quadratic variation process $\langle M^N(\phi)\rangle_t$ are
given explicitly. Tightness of $\{X^N\}$ is then established in Section~\ref{sec:mainresult}, where
Theorem~\ref{t:main} is also proved by showing that any weak limit satisfies the martingale problem
for SBM($X_0,\sigma^2,b$). The term $D^N(\phi)$ is easy to
handle; it is the asymptotic behavior of the quadratic variation process $\langle
M^N(\phi)\rangle$ which requires some work. Proposition~\ref{p:2did} is the key result here and is proved in Section~\ref{sec:sqfunction}.  Its proof uses Proposition~\ref{t:noncoal} but the issues go well beyond this result.  

The behavior of the quadratic variation
process is the main difference in the proofs of Theorem~\ref{t:main} and  its counterpart in \cite{CDE}, Theorem~\ref{t:CDE}. 
Lemma~4.3 in \cite{CDE} shows that
a key term in the variation process is negligible in the limit
$N\to\infty$. This fact is a consequence of the assumption
$J\to\infty$. In our case, with $J\equiv 1$, this term is non-negligible, and 
in fact determines the limiting
SBM branching rate. Its analysis is the main objective of Section~\ref{sec:sqfunction}. The analysis for $d\ge 3$ is straightforward; it is the $2$-dimensional case (the most relevant from a biological perspective) that is the most interesting.  In this setting the proof requires an extension of the arguments in \cite{CDP} and \cite{CP08} used to analyze the voter model and stochastic Lotka-Volterra models, respectively.    In \cite{CDP} four-particle duals were used to analyze the voter model, while in \cite{CP08} refinements allowed one to use only three-particle duals to analyze the more general stochastic Lotka-Volterra models considered there. Here, because of the non-Markovian property of individual
coordinates in the dual, similar calculations seem out of
reach and we must establish Proposition~\ref{p:2did} using only one- and two-particle duals. The fact that Proposition~\ref{p:2did} controls the square functions uniformly in time means it also allows one to establish tightness without any higher moments. The required properties of the two-particle dual are established in Section~\ref{sec:2dual}. Lemma~\ref{lem:difftimechange} represents the difference of the coordinates of the dual as the time change of a continuous time random walk and this result is then used to obtain several probability estimates on the two-particle dual. These results (notably Lemmas~\ref{l:Ibnd} to \ref{lem:xibound}) then play a central role in Section~\ref{sec:sqfunction}. The time-change is particularly useful when controlling
the two-particle dual when the particles are close together and the dual motions slow down.

 It would be interesting to see if it is possible to extend Theorem\ref{t:main} to the variable but bounded radius case
 discussed above.
\bigskip
\paragraph{Constants.} In proofs, $C$ will denote a positive constant
whose value may change from line to line. We will use $C_T$
and $C_\phi$ for constants depending on $T>0$ or functions $\phi$
in a similar way.
In some cases constants will be numbered and dependence on 
various quantities indicated explicitly. Finally, most constants
will have an implicit dependence on the impact radius $r$,
this dependence will be pointed out in some cases for
clarity.

\section{Martingale characterization}
Let $B^N_r(x)=B_{r/\sqrt N}(x)$, $\BMr =\BMr(0)$.  For $d\ge 2$
and $\phi\in C_b(\Rd)$,
\begin{align}\label{e:dNs}
d^N_s(\phi) &= 
\frac{\rho N^{1+d}}{|B_r|}\int_{\Rd} \int_{\BMrx} \Big\{
X^N_s(z)\int_{\BMrx}\phi(y)dy -
\int_{\BMrx}\phi(y)X^N_s(y)dy\Big\} dz dx,\\
\notag
m^N_s(\phi) & = 
\frac{\rho^2 N^{1+d}}{|B_r|} \int_{\Rd}
\int_{\BMrx}
\left\{\Big(1-\frac{X^N_s(z)}{K}\Big)
\Big(\int_{\BMrx}\phi(y)X^N_s(y)dy\Big)^2 
\right. \\
&\quad \left. +\frac{X^N_s(z)}{K} \Big( K
  \int_{\BMrx}\phi(y)dy-\int_{\BMrx}\phi(y)X^N_s(y)dy
\Big)^2\right\}dz dx.\label{e:mNs}
\end{align}

\medskip
\begin{lemma}\label{increments} Let $\phi\in C_b(\R^d)$ and $\Delta
X^N_s(\phi) = |X^N_s(\phi) - 
X^N_{s-}(\phi)|$. Then with probability one, 
\begin{equation}\label{e:incs}
|\Delta X^N_s(\phi)| \le \Vert\phi\Vert_\infty\begin{cases}
\rho|B_r|/N &\text{if }d\ge 3\\
\rho|B_r|\log N / N&\text{if }d=2
\end{cases} \quad \text{for all }s\ge 0.
\end{equation}
\end{lemma}
\begin{proof} By the dynamics \eqref{e:dynamics}, 
for $(x,s)\in \Pi^N$ (we may assume there is at most one such $x$), $w^N_s(y)=w^N_{s-}(y)$ for all
$y\notin \BMr(x)$, and for $y\in \BMr(x)$,
\[
w^N_s(y)-w^N_{s-}(y) =\begin{cases}
-\rho w^N_{s-}(y) + \rho, &\text{or }\\
-\rho w^N_{s-}(y).
\end{cases}
\]
Thus, $|w^N_s(y)-w^N_{s-}(y)|\le \rho1_{\BMr(x)}(y)$, and so,
\[
\int_{\Rd} |w^N_s(y) - w^{N}_{s-}(y)|dy \le 
\sup_{x\in\R^d} \int_{\BMr(x)}\rho \, dy
\le \rho |\BMr|.
\]
Finally,
\[
|\Delta   X^N_s(\phi) |\le 
K\int_{\Rd}|\phi(y)| |w^N_s(y) -
w^N_{s-}(y)| dy
\le \Vert\phi\Vert_\infty\rho|B_r| K N^{-d/2},
\]
which is \eqref{e:incs}. 
\end{proof}

The martingale characterization below is provided by
Lemma~3.1 of \cite{CDE}. The filtration below is implicit in their argument.  Although $\phi=1$ is not included in that result it is easy to handle it by a localization argument using the stopping times $T_n=\inf\{t\ge0:\int_{\{|x|\ge n\}}w_t(x)\,dx>0\}$.   Recall $\cF^N_t$ is defined prior to Theorem~\ref{t:CDE}.
\begin{proposition}\label{p:sm-decomp}
 Let $\phi\in C^3_0(\Rd)$ or $\phi=\1$.  Then $X^N_t(\phi)$ has the
semimartingale decomposition:
\begin{equation}\label{e:semimart}
X^N_t(\phi) = X^N_0(\phi) + 
D^N_t(\phi) + M^N_t(\phi),
\end{equation}
where 
\begin{equation}\label{e:DNs}
D^N_s(\phi) = \int_0^t d^N_s(\phi)ds
\end{equation}
and $M^N_t(\phi)$ is a local $(\cF^N_t)$-martingale
 with predictable
quadratic variation
\begin{equation}\label{e:MNs}
\langle M^N(\phi)\rangle_t = 
\int_0^t  m^N_s(\phi)ds .
\end{equation}
\end{proposition}

Implicit in the above is the fact that the local martingale $M^N_t(\phi)$ is locally square integrable, but
this is already clear from the fact that it has bounded jumps. The latter follows from Lemma~\ref{increments} and \eqref{e:semimart} which imply
\begin{equation}\label{Mjumps} |\Delta M_s^N(\phi)|\le \Vert\phi\Vert_\infty\begin{cases}
\rho|B_r|/N &\text{if }d\ge 3\\
\rho|B_r|\log N / N&\text{if }d=2
\end{cases} \quad \text{for all }s\ge 0.
\end{equation}

For the drift term $D^N_s(\phi)$ we will need only the
following facts.

\begin{lemma}
\label{l:dNs} \

(a) $d^N_s(\1) \equiv 0$.

(b) For $\phi\in C^3_0(\Rd)$ there is a constant
$C_{\ref{e:dNs2}}=C_{\ref{e:dNs2}}(\phi)>0$ such 
that
\begin{equation}\label{e:dNs2}
\begin{aligned}
d^N_s(\phi) &= \rho|B_r|\bar\sigma^2 X^N_s\Big(\frac12
\Delta\phi\Big) + \cale^N_{\ref{e:dNs2}}(s), \text{ where}\\
|\cale^N_{\ref{e:dNs2}}(s)| &\le C_{\ref{e:dNs2}}
\frac{X^N_s(\1)}{\sqrt N} .
\end{aligned}
\end{equation}
\end{lemma}

Part (a) follows easily from \eqref{e:dNs}, and (b) is the
special case of Lemma~4.2 (and its proof) in  \cite{CDE}  for our choices
of $J,M,K$ in \eqref{e:JMK}. (We note that the constant
$C(d)$ in Definition~4.1 in \cite{CDE} is
$\int_{B_1}(x_1)^2dx$.) Turning next to the martingale
square function, for $\phi\in C_b(\Rd)$ define 
\begin{equation}\label{e:mbarNs}
\bar m^N_s(\phi) = \rho^2N^{1+d/2}K^2  \int_{\Rd}\phi^2(x) 
\int_{\BMrx}\int_{\BMrx}
(1-w^N_s(z_1))w^N_s(z_2)dz_2dz_1 dx\ge 0 .
\end{equation}
\begin{lemma}\label{l:mNs}\

(a) For $\phi\in C^3_0(\Rd)$, $\bar m^N_s(\phi)\vee m^N_s(\phi)\le
  \|\phi\|^2_\infty m^N_s(\1)$. 

(b) $m^N_s(\1)=\bar m^N_s(\1)$. 

(c) For $\phi\in C^3_0(\Rd)$ there is a constant
$C_{\ref{e:mNs2}}=C_{\ref{e:mNs2}}(\phi)>0$ such 
\begin{equation}\label{e:mNs2}
\begin{aligned}
m^N_s(\phi) &= \bar m^N_s(\phi) +
\cale^N_{\ref{e:mNs2}}(s), \text{ where}\\ 
|\cale^N_{\ref{e:mNs2}}(s)| &\le 
C_{\ref{e:mNs2}}
\begin{cases}
\frac{1}{\sqrt N}\,  X^N_s(\1)
&\text{if $d\ge 3$}\\
\frac{\log N}{\sqrt N}\,X^N_s(\1)  
&\text{if $d=2$.}
\end{cases}
\end{aligned}
\end{equation}
\end{lemma}
\begin{proof}
Define
\[
\cali(z_1,z_2,z_3) = (1-w^N_s(z_1))
w^N_s(z_2)w^N_s(z_3) + w^N_s(z_1)
(1-w^N_s(z_2))(1-w^N_s(z_3)).
\]
Then replacing $X^N_s$ with $Kw^N_s$ in \eqref{e:mNs}, 
after expanding and rearranging, we find that
\begin{align*}
m^N_s(\phi) & = 
\frac{\rho^2N^{1+d}}{|B_r|}K^2 \int_{\Rd}
\int_{\BMrx}\int_{\BMrx}\int_{\BMrx}
\phi(z_2)\phi(z_3)
\cali(z_1,z_2,z_3)\, dz_3dz_2dz_1dx
\end{align*}
On account of $0\le w^N_s\le 1$, $\cali$ is nonnegative, 
hence (a) follows for $m^N$ from the above expression, and is
immediate for $\bar m^n$ from \eqref{e:mbarNs} (integrate out $z_3$ in the first line on the right-hand side). 
If we define $\Delta_\phi(x,z_2,z_3) = \phi(z_2)\phi(z_3) -
\phi^2(x)$, then
\begin{equation}\label{e:mNs3}
m^N_s(\phi)  = 
\frac{\rho^2N^{1+d}}{|B_r|}K^2 \int_{\Rd}
\int_{\BMrx}\int_{\BMrx}\int_{\BMrx}
\phi^2(x) \cali(z_1,z_2,z_3)\,  dz_3dz_2dz_1dx
+ \cale^N_{\ref{e:mNs2}}(s),  
\end{equation}
where
\[
\cale^N_{\ref{e:mNs2}}(s)  = 
\rho^2\frac{N^{1+d}}{|B_r|}K^2 
\int_{\Rd}
\int_{\BMrx}\int_{\BMrx}\int_{\BMrx}
\Delta_\phi(x,z_2,z_3) \cali(z_1,z_2,z_3) \, dz_3dz_2dz_1dx .
\]

Consider the integrals over $\BMr(x)$ in \eqref{e:mNs3}. 
By a change of variables and order of integration, 
\begin{align*}
\int_{\BMr(x)}\int_{\BMr(x)}&\int_{\BMr(x)} \cali(z_1,z_2,z_3)\ dz_3dz_2dz_1\\
&= \int_{\BMr(x)}w^N_s(z_3) \int_{\BMr(x)}\int_{\BMr(x)} 
(1-w^N_s(z_1))w^N_s(z_2) dz_2dz_1 dz_3\\
&\qquad+ \int_{\BMr(x)}(1-w^N_s(z_3)) \int_{\BMr(x)}\int_{\BMr(x)} 
(1-w^N_s(z_1))w^N_s(z_2) dz_2dz_1 dz_3\\
&= |\BMr| \int_{\BMr(x)}\int_{\BMr(x)} 
(1-w^N_s(z_1))w^N_s(z_2) dz_2dz_1 .
\end{align*}
Plugging this into \eqref{e:mNs3}, and using the definition
of $\bar m^N_s(\phi)$, we now have
$m^N_s(\phi) = \bar m^N_s(\phi) +
\cale^N_{\ref{e:mNs2}}(s)$. 

In the case that $\phi=\1$, $\Delta_{\1}\equiv 0$ so that 
$\cale^N_{\ref{e:mNs2}}(s)\equiv 0$, proving (b).  
More generally, 
$$|\Delta_\phi(x,z_2,z_3)|\le 
2(r/\sqrt N)\|\phi\|_\infty \|\phi\|_{\text{Lip}} 
\text{ for $z_2,z_3\in \BMrx$}.$$ Using  the fact that 
$|\cali(z_1,z_2,z_3)|\le w^N_s(z_2) +w^N_s(z_1)$, we have
\begin{align*}
|\cale^N_{\ref{e:mNs2}}(s)| &\le  \rho^2\frac{N^{1+d}}{|B_r|}K^2
2(r/\sqrt N)\|\phi\|_\infty
\|\phi\|_{\text{Lip}} \int_{\R^d}\int_{\BMr(x)} \int_{\BMr(x)}\int_{\BMr(x)} 
(w^N_s(z_1)+w^N_s(z_2))dz_2dz_1dz_3dx\\
& =
4\rho^2\frac{r}{N^{1/2}}\|\phi\|_\infty \|\phi\|_{\text{Lip}} 
\frac{N^{1+d}}{|B_r|}|\BMr|^2K^2 \int_{\Rd}
\int_{\BMrx}
w^N_s(z_1)dz_1dx  \\
& =
4r\rho^2|B_r|^2\phi\|_\infty \|\phi\|_{\text{Lip}} 
\begin{cases}
\frac{1}{\sqrt N} X^N_s(\1) &\text{if }d\ge 3\\
\frac{\log N}{\sqrt N} X^N_s(\1) &\text{if }d=2.
\end{cases}
\end{align*}
This proves (c). 
\end{proof}

\section{Basic moment bounds}

We start with the dual particle systems for the rescaled SLFV process $w^N_t$ in \eqref{wNdef}.  If $\eta$ and $(\xi^{1},\xi^{2})$ are as in \eqref{e:dual1}, \eqref{e:dual2}, introduce the rescaled duals, 
\[P^N_{\{x\}}(\eta^{N}_t\in\cdot)=P_{\{\sqrt Nx\}}\Bigl(\frac{\eta_{Nt}}{\sqrt N}\in\cdot\Bigr),\quad x\in\R^d,\]
\[P^N_{\{x_1,x_2\}}((\xi_{t}^{N,1},\xi_t^{N,2})\in\cdot)=P_{\{\sqrt Nx_1,\sqrt Nx_2\}}\Bigl(\Bigl(\frac{\xi_{Nt}^{1}}{\sqrt N},\frac{\xi_{Nt}^{2}}{\sqrt N}\Bigr)\in\cdot\Bigr)\quad x_i\in\R^d,\]
and
\begin{equation}\label{taundef}\tau^N=\inf\{t:\xi^{N,1}_t=\xi^{N,2}_t\}.
\end{equation}
Then \eqref{e:dual1}  and
\eqref{e:dual2} imply for Borel $\psi_1$ on $\R^d$,   and Borel
$\psi_2$ on $(\Rd )^2$,  and $t\ge 0$ (recall \eqref{wphidef}),
\begin{align}\label{e:dualN1}
&E_{w^N_0} [w^N_t(\psi_1)]  = \int_{\R^d} \psi_1(x) E^N_{\{x\}}(
w^N_0(\eta^{N}_t))dx,\\\notag
&E_{w^N_0}\Big[\int_{ \R^d\times \R^d}
\psi_2(x_1,x_2)w^N_t(x_1)w^N_t(x_2)\, dx_1dx_2\Big]\\
&\quad = \int_{ \R^d\times \R^d} \psi_2(x_1,x_2)E^N_{\{x_1,x_2\}}\big[
w^N_0(\xi^{N,1}_t)1_{\{\tau^N\le t\}} + 
w^N_0(\xi^{N,1}_t)w^N_0(\xi^{N,2}_t)1_{\{\tau^N>t\}}\big]\,
dx_1dx_2. 
\label{e:dualN2}
\end{align}
As before, either $\psi_i\ge 0$, or one side is integrable for the above to hold.
A simple change of variables shows that \eqref{e:dualN1}
implies (for $\psi_1$ as above)
\begin{equation}\label{e:dualXN}
E(X_t^N(\psi_1))=E_{X_0^N}(\psi_1(\eta^N_t)):=\int_{\Rd}
E^N_{\{x\}}(\psi_1(\eta^{N}_t))X_0^N(dx).
\end{equation}

\begin{proposition}\label{p:ms-bnd}\ 

(a) There exists $C_{\ref{e:L2bnda}}>0$ such that for $s\ge 0$,
\begin{equation}\label{e:L2bnda}
{\bar m}^N_s(\1)=m^N_s(\1) \le C_{\ref{e:L2bnda}} 
\begin{cases} X^N_s(\1)&\text{if }d\ge 3\\
(\log N) X^N_s(\1)&\text{if }d=2
\end{cases}
\end{equation}

(b) Assume $d=2$. If  $0<\alpha<1$ then there exists
$C_{\ref{e:L2bndb}}=C_{\ref{e:L2bndb}} 
(\alpha)>0$ such that for $s\ge 0$, 
\begin{equation}\label{e:L2bndb}
E(m^N_s(\1)) \le
C_{\ref{e:L2bndb}}\,(1+s^{-\alpha}) X^N_0(\1).  
\end{equation}
\end{proposition}
\begin{proof}
(a) 
The case  
  $d\ge 3$ with $K=N^{d/2-1}$ is straightforward. 
 Using the definition of $\bar m^N_s$, Lemma~\ref{l:mNs}(b)
 and  the fact that $0\le w^N_s\le 1$,  
\begin{align*}
m^N_s(\1)=\bar m^N_s(\1) & \le \rho^2 N^d K
           |\BMr|\int_{\Rd}\int_{\BMrx} w^N_s(z_2)dz_2dx
= \rho^2N^d|\BMr|^2 X^N_s(\1) 
 = \rho^2 |B_r|^2 X^N_s(\1) .
\end{align*}
Now suppose $d=2$, with $K=\log N$.  Write
$(1-w^N_s(z_1))w^N_s(z_2) = w^N_s(z_2) - 
w^N_s(z_1)w^N_s(z_2)$, and so conclude that 
\begin{equation}\label{e:mNII}
 m^N_s(\1)= \bar m^N_s(\1)= m^{N,1}_s - m^{N,2}_s,
\end{equation}
where
\begin{align*}
m^{N,1}_s & = \rho^2N^2 (\log N)^2
  \int_{\Rtwo}\int_{\BMrx}
\int_{\BMrx} w^N_s(z_2)\,dz_1dz_2dx\\
m^{N,2}_s &= \rho^2 N^2 (\log N)^2 \int_{\Rtwo} \int_{\BMrx}
  \int_{\BMrx} w^N_s(z_1) w^N_s(z_2) 
\,dz_1dz_2dx .
\end{align*}
Thus, $m^N_s(\1)\le m^{N,1}_s$, and 
\begin{align}\label{mN1id}
\nonumber m^{N,1}_s & = \rho^2 N^2 |\BMr| (\log N)\int_{\Rtwo}\int_{\BMrx}
  X^N_s(z_2)\, dz_2dx\\
\nonumber & = \rho^2 N^2 |\BMr|^2 (\log N) X^N_s(\1) \\
&= \rho^2|B_r|^2 (\log N) X^N_s(\1) .
\end{align}
This completes the proof of \eqref{e:L2bnda}.

(b) For $d=2$, using the two particle duality
equation \eqref{e:dualN2} with 
  $\psi_2\equiv 1$, we have
\begin{align*}
E[m^{N,2}_s] &= \rho^2 N^2 (\log N)^2 \int_{\Rtwo}  \int_{\BMrx}
  \int_{\BMrx} E\big[w^N_s(z_1) w^N_s(z_2)\big] dz_1dz_2dx\\
&=  \rho^2 N^2 (\log N)^2 \int_{\Rtwo} \int_{\BMrx}
  \int_{\BMrx} 
  \Big(E^N_{\{z_1,z_2\}}\big[w^N_0(\xi_s^{N,1})1\{\tau^N\le s\big\}\big] \\
&\qquad +
  E^N_{\{z_1,z_2\}}\big[w^N_0(\xi_s^{N,1})w^N_0(\xi_s^{N,2})1\{\tau^N>
  s\big\}\big]\Big) dz_1dz_2dx\\
&\ge  \rho^2 N^2 (\log N) \int_{\Rtwo}   \int_{\BMrx}
  \int_{\BMrx} 
  E^N_{\{z_1,z_2\}}\big[X^N_0(\xi_s^{N,1})1\{\tau^N\le s\}\big] dz_1dz_2dx\\
&=  \rho^2 N^2 (\log N) \int_{\Rtwo}  \int_{\BMr}
  \int_{\BMr} 
  E^N_{\{x+z_1',x+z_2'\}}\big[X^N_0(\xi_s^{N,1})1\{\tau^N\le s\}\big] 
dz'_1dz'_2dx\\
& = \rho^2 N^2 (\log N) \int_{\Rtwo} \int_{\BMr}
  \int_{\BMr} 
  E^N_{\{z_1',z_2'\}}\big[X^N_0(x+\xi_s^{N,1})1\{\tau^N\le s\}\big]
dz'_1dz'_2dx,
\end{align*}
the last two equalities by a change of
variables and translation
invariance \eqref{TI}. Let \break
$P_{\{U_1,U_2\}}(\cdot)=\int_{B_r}\int_{B_r}P_{\{z_1',z_2'\}}
(\cdot)dz'_1dz'_2/|B_r|^2$ so that
$P_{\{U_1,U_2\}}(\tau>t)=\gamma_e(t)$. Integrating $x$ out (inside the expectation),  
 we obtain 
\begin{align*}
E[m^{N,2}_s] &\ge 
\rho^2 N^2 (\log N) X^N_0(\1) \int_{\BMr}
  \int_{\BMr} 
  P^N_{\{z'_1,z'_2\}}(\tau^N\le s)dz'_1dz'_2\\
&= \rho^2 N^2 |\BMr|^2 (\log N)P_{\{U_1,U_2\}}(\tau\le Ns) X^N_0(\1)\\
&= \rho^2 |B_r|^2 (\log N) (1-\gamma_e(Ns))
  X^N_0(\1) .
\end{align*}
If we plug this bound and \eqref{mN1id} 
into \eqref{e:mNII}, we get
\begin{equation}\label{e:L2bndb1}
E[m^N_s(\1)] \le \rho^2 |B_r|^2 (\log N)
\gamma_e(Ns)X^N_0(\1).
\end{equation}

If $s\le (\log N)^{-\alpha}$, then
$(\log N)\gamma_e(Ns) \le \log N \le s^{-1/\alpha}$.
By Proposition~\ref{t:noncoal}, if $s\ge (\log
N)^{-\alpha}$,
\[
(\log N)\gamma_e(Ns) \le (\log N) \gamma_e\Big(\frac{N}{(\log
  N)^{\alpha}}\Big) \to \gamma_e
\]
as $N\to\infty$. This implies that there is a $C>0$, depending on
$\alpha$, such that for all $s\ge (\log N)^{-\alpha}$, 
$(\log N)\gamma_e(Ns) \le C$. Plugging these bounds into 
\eqref{e:L2bndb1} we obtain
\eqref{e:L2bndb}. 
\end{proof}

\begin{corollary}\label{c:L2bnd}
 For $d\ge 2$ and $T>0$ there exists
 $C_{\ref{e:L2bnd}}(T)>0$ such that 
\begin{equation}\label{e:L2bnd}
E\Big[\sup_{t\le T} X^N_t(\1)^2\Big] \le 2 X^N_0(\1)^2 +
C_{\ref{e:L2bnd}}(T) X^N_0(\1) .
\end{equation}
Moreover $t\rightarrow X^N_t(\1)$ is a non-negative
square-integrable $(\cF^N_t)$-martingale, and for any $\phi\in C^3_0(\R^d)$, $t\rightarrow M^N_t(\phi)$ is also a square-integrable $(\cF^N_t)$-martingale. 
\end{corollary}
\begin{proof}

By Proposition~\ref{p:sm-decomp}
  and Lemma~\ref{l:dNs}(a),
$X^N_t(\1)= X^N_0(\1) + M^N_t(\1)$, and so 
$$X^N_t(1)^2
\le 2X^N_0(\1)^2 + 2M^N_t(\1)^2.$$ 
Now $X_t^N(\1)$ is a non-negative local martingale which by
\eqref{e:mass} satisfies $E[X^N_t(\1)]=X_0^N(\1)$, and so
is a non-negative martingale.  
By
Doob's $L^2$ submartingale inequality, 
\begin{align*}
E\big[\sup_{t\le T}M^N_t(\1)^2\big] \le 
4E\big[\langle M^N(\1)\rangle_T\big] = 
4 \int_0^TE\big[ m^N_s(\1)\big] ds .
\end{align*}
As we don't know the square integrability yet, the first inequality holds by considering a sequence of localizing stopping times and applying monotone convergence. 
By Proposition~\ref{p:ms-bnd}(a), for $d\ge 3$,
\[
\int_0^T E[m^N_s(\1)]ds \le 
C_{\ref{e:L2bnda}}\int_0^T E[X^N_s(\1)]ds
= C_{\ref{e:L2bnda}} T X^N_0(\1).
\]
By Proposition~\ref{p:ms-bnd}(b), for $d=2$ and taking
$\alpha=1/2$, 
\[
\int_0^T  E[m^N_s(\1)]ds 
\le C_{\ref{e:L2bndb}}(\tfrac12) X^N_0(\1) \int_0^T (1+s^{-1/2})ds 
=  C_{\ref{e:L2bndb}}(\tfrac12)(T + 2T^{1/2})X^N_0(\1) .
\]
Combining the above bounds we obtain 
\eqref{e:L2bnd} and hence the next to last statement as well. 

It is easy to repeat the above reasoning using Lemma~\ref{l:mNs}(a) and see
that 
\begin{equation}\label{maxMbound}E\big[\sup_{t\le T}M^N_t(\phi)^2\big]\le C_{\phi}[T+\sqrt T]X_0^N(\1).
\end{equation}
This in turn shows that the local martingale $M^N(\phi)$ is in fact a square integrable martingale.
\end{proof}

\section{Random walk preliminaries} \label{sec:rwprel}Recall that 
$(\xi^{1}_t,\xi^{2}_t)$ is the two particle dual with
initial state $(x_1,x_2)$. We
will need to work with the difference process,
\[
\txi_t = \xi^{1}_t - \xi^{2}_t, 
\]
which is a Markov process starting at 
$x_1-x_2$.  
When $|\txi_t|>2r$, it makes the same transitions as a rate
$2\rho|B_r|$ random walk with jump distribution that of $\bar
U$, with density $h_{\bar U}(z)$ given in \eqref{e:barUdens}. We
will need basic information about this random walk, as well
as a way to compare $\txi_t$ to it. 

Throughout the paper, $Y_t=Y^x_t$ will denote a  rate $2\rho|B_r|$
random walk with jump 
distribution that of $\bar U$ starting at $x$ under $P_x$. That is,
$Y^x_t$ will be the pure-jump Markov process on $\R^d$ with generator
\begin{equation}\label{e:Ygen}
\cala^{Y}f(x) = 2\rho|B_r| \int_{B_{2r}(0)} (f(x+z)-f(x)) h_{\bar U}(z)dz 
\end{equation}
defined for suitable $f$. 
We will often make use of the Poisson process construction
\[
Y^0_t = S_{N_t}, \quad t\ge0,
\]
where $N_t$ is a rate $\lambda=2\rho|B_r|$ Poisson process on
$[0,\infty)$ which is independent of the iid random
variables $\bar U_1,\bar U_2,\dots$ which have the same
law as $\bar U$, and
$S_n=\bar U_1+\cdots \bar U_n$, $n\ge 1$ ($S_0=0$). We will
often write $Y_t$ for $Y^0_t$ and $Y^x_t$ for $x+Y^0_t$, where $x$ may be random.
Recall from \eqref{barsigdef} that $= E[|\bar U|^2]=d\bar\sigma^2$.

\begin{lemma}\label{l:rwfacts} (a) There is a constant
  $C_{\ref{e:Ydensbnd}}>0$  such that for all $t\ge 0$ and
  Borel $B\subset\Rd$,
\begin{equation}\label{e:Ydensbnd}
\begin{aligned}
P(Y^{\bar U}_t\in B) 
&\le \frac{C_{\ref{e:Ydensbnd}}|B|}{1+t^{d/2}}, \ \text{ and}\\
P_x(Y_t\in B) 
&\le e^{-2\rho|B_r|t}1_B(x) + \frac{C_{\ref{e:Ydensbnd}}|B|}{1+t^{d/2}}\quad\text{for all }x\in\R^d.
\end{aligned}
\end{equation}
In particular, for all $x\in\Rd$, $t>0$, and nonnegative Borel $f$,
\begin{equation}\label{e:Ydensbnd2}
E_x\Big[f\Big( \frac{Y_t}{\sqrt t}\Big)\Big] \le e^{-2\rho|B_r|
  t}\|f\|_\infty + C_{\ref{e:Ydensbnd}}\int_{\Rd}f(z)dz
\end{equation}
(b) For all $t\ge 0$, $E_0[|Y_t|^2]= 2\rho|B_r|d\bar\sigma^2t$. For $k\in\N$
there is a constant  
  $C_{\ref{e:Ymmt}}= C_{\ref{e:Ymmt}}(k)>0$ 
such that
\begin{equation}\label{e:Ymmt}
E_0[\sup_{s\le t}|Y_s|^{2k}] \le C_{\ref{e:Ymmt}}
t^k \text{ for all }t\ge 1. 
\end{equation}
(c) For $k\in\N$ there is a constant  
  $C_{\ref{e:Ysupbnd}}= C_{\ref{e:Ysupbnd}}(k)>0$
  \begin{equation}\label{e:Ysupbnd}
P_0(\sup_{s\le t}|Y_s|\ge a) \le C_{\ref{e:Ysupbnd}} 
\frac{t^{k}}{a^{2k}} \text{ for all } a>0, t\ge 1. 
\end{equation}
\end{lemma}
\begin{proof} (a) According to Theorem~19.1 of \cite{BR}, there is a 
uniform bound on the densities of
$S_n/\sqrt n$, $n=1,2,\dots$, so that  
\begin{equation}\label{e:BRbnd}
P\Big( \frac{S_n}{\sqrt n}\in B \Big) \le C|B|
\quad\forall\ n\ge 1, \text{ Borel }B\subset \R^d. 
\end{equation}
By a standard large deviations estimate, for $0<\alpha<1$, 
\[
P(N_t\le \alpha\lambda t) \le 
\exp(-c_\alpha\lambda t) ,
\]
where $c_\alpha=1-\alpha+\alpha\log \alpha>0$. 
Using $Y^x_t=x+S_{N_t}$ and the density bound \eqref{e:BRbnd}, 
\begin{align*}
P\Big(Y^x_t\in B\Big) &= e^{-\lambda t}1_B(x) +
\sum_{n=1}^\infty e^{-\lambda t}\frac{(\lambda t)^n}{n!}
P\Big(\frac{S_n}{\sqrt n}\in \frac{B-x}{\sqrt n} \Big)\\
&\le e^{-\lambda t}1_B(x) +
C\Big\{\sum_{1\le n< \lambda t/2}+ \sum_{n\ge1\vee(\lambda t/2)}\Big\}
e^{-\lambda t}\frac{(\lambda t)^n}{n!}
\frac{|B|}{n^{d/2}}\\
& \le e^{-\lambda t}1_B(x) +
C|B| \Bigl(P(N_t< \lambda t/2)+
E\big[1\{N_t\ge 1\vee(\lambda t/2)\}N_t^{-d/2}\big]\Bigr) \\
& \le e^{-\lambda t}1_B(x) +
C|B|\Bigl(e^{-c_{1/2}\lambda t}+
(1\vee(\lambda t/2))^{-d/2}\Bigr)\\
&\le e^{-\lambda t}1_B(x) +C|B|(1+t)^{-d/2}, 
\end{align*}
where we have used the large deviation bound with
$\alpha=1/2$. This proves (a) 
for $Y$ starting at $x$. The result for
$Y^{\bar U}_t$ follows from the observation that $Y^{\bar
  U}_t$ has the same law as $S_{N(t)+1}$ and a slight alteration in the above calculation.
  
(b) It is easy to see from the representation $Y^0_t=S_{N_t}$
that $E_0[|Y_t|^2] = 2\rho|B_r| t E[|\bar U|^2]$. Now consider
$k\in\N$. There is a constant $C_k>0$ such that
$E(\max_{m\le n}|S_m|^{2k}) \le C_kn^k$. To see this
we switch to
component notation, and write $\bar
U_j= (\bar U_j^{(1)},\dots,\bar U_j^{(d)})$ and 
$S_n=(S_n^{(1)},\dots,S^{(d)}_n)$, where
$S^{(j)}_n=\sum_{i=1}^n \bar U^{(j)}_i$. Then
\[
E[\max_{m\le n}|S_m|^{2k}] = E\Big[\max_{m\le n} \big( \sum_{j=1}^d|S^{(j)}_m|^2 \big)^k\Big]
\le d^{k-1}E\Big[\sum_{j=1}^d \max_{m\le n}|S^{(j)}_n|^{2k}\Big] =d^{k}
E\big[
\max_{m\le n}|S_m^{(1)}|^{2k} \big].
\]
Now $S^{(1)}_n$ is a sum of bounded, mean zero independent random
variables, so a martingale square function argument (e.g. see Theorem 21.1 of \cite{Burk73}) shows that
for each $k\ge 1$ there
  is a constant
  $C_k=C_k(r)>0$
  such that $E[\max_{m\le n}|S^{(1)}_m|^{2k}]  \le C_k
  n^{k}$ for all $n\in\nn$. This implies 
\begin{align*}
E[|Y_t|^{2k}] &= \sum_{n=1}^\infty e^{-\lambda t}
\frac{(\lambda t)^n}{n!} E[|S_n|^{2k}] 
\le Cd^{k+1}\Big(\sum_{n=1}^{k-1}e^{-\lambda t}
\frac{(\lambda t)^n}{n!} n^k
+ \sum_{n={k}}^\infty e^{-\lambda t}
\frac{(\lambda t)^n}{n!} n^k\Big) .
\end{align*}
The first sum is bounded by $(k-1)^k$.
The second sum is bounded by
\[
(\lambda t)^k\sum_{n=k}^{\infty}e^{-\lambda t}
 \frac{(\lambda t)^{n-k}}{(n-k)!}\prod_{j=1}^{k}\frac{n}{n-j+1}
\le  (\lambda t)^kk^k.
\]
This proves \eqref{e:Ymmt} for
$t\ge 1$. 

(c) This is immediate from (b) and Markov's inequality.
\end{proof}

For $a,A>0$, define the hitting times 
\begin{equation}\label{hittingdefn}
t_a = \inf\{s\ge 0:|Y_s|\le a\}\text{ and }
T_A = \inf\{s\ge 0:|Y_s|\ge A\}.
\end{equation}

\begin{proposition} Assume $d\ge 3$ and $2r<a<|x|$.
Then 
\begin{equation}\label{e:tafinite}
P_x(t_a <\infty) \le \Big(\frac{a}{|x|}\Big)^{d-2}
\end{equation}
\end{proposition} 
\begin{proof} Let $A>|x|$. By radial symmetry and \eqref{e:Ygen},
$f(x)= |x|^{2-d}$ is a harmonic function for $Y$. If we
let $\sigma=t_a\wedge T_A$,
then $|Y_{s\wedge \sigma}|^{2-d}$ is a bounded
martingale (recall $a>2r$), and so
\begin{equation}\label{e:xaA-1}
|x|^{2-d} = E_x[|Y_{t_a}|^{2-d} 1(t_a<T_A)]
+ E_x[|Y_{T_A}|^{2-d} 1(T_A<t_a)] .
\end{equation}
Clearly $|Y_{T_A}|^{2-d}\le A^{2-d}$ if $T_A<t_a$, and
$|Y_{t_a}|^{2-d}\ge a^{2-d}$ if ${t_a<T_A}$. This means that
if we let 
$A\to\infty$ in $\eqref{e:xaA-1}$, then 
\[
|x|^{2-d} \ge E_x[a^{2-d}1( t_a<\infty)]
= a^{2-d}P_x( t_a<\infty),
\]
proving \eqref{e:tafinite}. 
\end{proof}

\begin{lemma}\label{l:logprob} Assume $d=2$. If 
$A>2$ and $2r<a<|x|<A$,  then 
\begin{gather}\label{e:Aa}
P_x(T_A<t_a) \le\frac{\log|x|-\log
    (a-2r)}{\log A-\log (a-2r)},\\
\label{e:aA}
P_x(t_a<T_A)\le\frac{\log (A+2r) - \log|x|}
{\log (A+2r) -\log a},\\ \intertext{and}
\label{e:Aalimit}
\lim_{A\to\infty}(\log A)\; P_x(T_A<t_a)  
= \log|x| - E_x[\log|Y_{t_a}|] .
\end{gather}
\end{lemma}
\begin{proof}
By radial symmetry and \eqref{e:Ygen},
$\log |x|$ is a harmonic function for $Y$. If 
$\sigma=t_a\wedge T_A$ as before then  $\log|Y_{s\wedge \sigma}|$ is a bounded
martingale, and 
\begin{equation}\label{e:xaA-3}
\log |x| = E_x[\log |Y_{t_a}|; t_a<T_A]
+ E_x[\log |Y_{T_A}|; T_A<t_a]. 
\end{equation}
Using  $|Y_{T_A}|\ge A$
and $|Y_{t_a}|> a-2r$ in the above gives
\begin{align*}
\log |x| &\ge  (\log (a-2r))(1-P_x(T_A<t_a)) +
(\log A) P_x(T_A<t_a).
\end{align*}
Rearranging gives \eqref{e:Aa}. A similar argument yields 
\eqref{e:aA}.

For \eqref{e:Aalimit}, rearranging \eqref{e:xaA-3} gives
\begin{align*}
\Big|\log |x|& - E_x[\log |Y_{t_a}|;
t_a<T_A)]-(\log A) P_x(T_A<t_a)\Big| \\
& = E_x[(\log|Y_{T_A}|-\log A)1(T_A<t_a)] 
\le \log\Big( \frac{A+2r}{A}\Big) \to 0
\end{align*}
as $A\to\infty$. 
Also, by recurrence and bounded convergence, 
$E_x[\log |Y_{t_a}|;t_a<T_A] \to E_x[\log
|Y_{t_a}|]$ as $A\to\infty$, which means that
\eqref{e:Aalimit} must hold.  
\end{proof}

\begin{lemma}\label{l:AlogA} Assume $d=2$. 
 There is a constant 
$C_{\ref{e:AlogA}}>0$ such that  for $A>2$ and $|x|<A/2$,
\begin{equation}\label{e:AlogA}
P_x\left( T_A\notin \Big[\frac{A^2}{\log A},A^2\log A\Big]\right) 
\le C_{\ref{e:AlogA}}/(\log A)^{2}. 
\end{equation}
\end{lemma}

\begin{proof} By \eqref{e:Ysupbnd} with $k=2$, for all $x,A$ as in the Lemma,
\begin{align}\notag
P_x(T_A\le  \frac{A^2}{\log A}) &=
P_0\Big(\sup_{s\le A^2/\log A}|x+Y_s| \ge A\Big)
\le P_0\Big(\sup_{s\le A^2/\log A}|Y_s| >A/2\Big)\\
&\le C_{\ref{e:Ysupbnd}}\frac{(A^2/\log A)^2}{(A/2)^4}
=16C_{\ref{e:Ysupbnd}}/(\log A)^2.\label{e:AlogA1}
\end{align}

To handle $P_x(T_A\ge A^2\log A)$ we must first estimate
$E_x(T^2_A)$. 
Let $\sigma_2=E(|\bar U|^2)$, $\sigma_4=E(|\bar U|^4)$ and
$\lambda=2\rho|B_r|$, and
define the functions
\begin{align*}
u_2(t,y) &= |y|^2-\lambda \sigma_2t, \\
u_4(t,y) &= |y|^4 -4\lambda \sigma_2|y|^2t +2(\lambda \sigma_2)^2 t^2
-\lambda \sigma_4t.
\end{align*}
It is a straightforward calculation to check that both
$u=u_2$ and $u=u_4$ satisfy 
\[
\frac{\partial u}{\partial t}  + \cala^Y u \equiv 0.
\]
This and the fact that for $p=2,4$, $|y|^p$,  and $\cala^Y(
|y|^p$) are bounded on $\{|y|\le A+2r\}$,   
 so that $u_p(t\wedge T_A,Y_{t\wedge T_A})$ is uniformly
 bounded for  $t\le t_0$, easily imply 
 that both $u_2(t\wedge T_A,Y_{t\wedge T_A})$ and $u_4(t\wedge T_A,Y_{t\wedge T_A})$ are
martingales.   
Therefore $E_0[u_2(t\wedge T_{A},Y_{t\wedge T_{A}})] =0$, 
and since $|Y_{T_{A}}|\le A+2r$, 
\begin{equation}\label{e:AlogA2}
\lambda \sigma_2 E_0 [T_{A} ]=\lim_{t\to\infty}\lambda \sigma_2 E_0 [T_{A} \wedge t] = \lim_{t\to\infty}E_0\big[(Y_{T_{A}\wedge t})^2\big] 
\le (A+2r)^2. 
\end{equation}
Now, since $E_0[u_4(T_{A}\wedge t,Y_{T_{A}\wedge t})] =0$,
\begin{align*}
2(\lambda \sigma_2)^2 E_0\big[(T_{A}\wedge t)^2\big] 
&= - E_0\big[ |Y_{T_{A}\wedge t}|^4\big] 
+ 4\lambda \sigma_2 E_0\big[(T_{A}\wedge t)|Y_{T_{A}\wedge t}|^2\big]
+ \lambda \sigma_4E_0\big[T_{A}\wedge t\big]  \\
&\le 4\lambda \sigma_2(A+2r)^2 E_0[T_{A}]  +\lambda
  \sigma_4E_0[T_{A}]
\\
&\le 4(A+2r)^4+\frac{\sigma_4}{\sigma_2}(A+2r)^2,
\end{align*}
where we have used \eqref{e:AlogA2}. Let $t\to\infty$ on the
left-hand side of the above to conclude that 
\begin{equation}\label{e:AlogA3}
E_0[T_{A}^2] \le C(A+2r)^4\quad\text{for all }A>1. 
\end{equation}
On account of this bound and Markov's inequality, we have
for $|x|<A/2$ and $A>2$,
\[
P_x(T_A>A^2\log A) \le P_0(T_{3A/2}> A^2\log A) \le 
E_0[T_{3A/2}^2] / (A^2\log A)^2\le C/(\log A)^2.
\]
Together with \eqref{e:AlogA1} this proves \eqref{e:AlogA}. 
\end{proof}
 
The following  technical result will play a key role in
the proof of Lemma~\ref{l:GE} below.

\begin{lemma}\label{l:lR} Assume $d=2$. For
  $\alpha,\beta>0$ and $t>0$  
there is a constant $C_{\ref{e:lR}}=
C_{\ref{e:lR}}(\alpha,\beta, t)\ge 1$ such
  that if $(\log N)^{-\beta}\le s\le t$ and $w\in\R^2$
  satisfies $0< |w|\le (\log
  N)^{-\alpha}$,
  then 
\begin{equation}\label{e:lR}
P_{w\sqrt N}(t_{3r}< Ns) 
\le C_{\ref{e:lR}} \frac{\log(1/|w|)}{\log N}\, .
\end{equation}
\end{lemma}
\begin{proof} We may suppose $|w|> 3r/\sqrt N$, because otherwise 
  $C_{\ref{e:lR}}$ can be chosen large enough so that the
  right side of \eqref{e:lR} is at least one. 
Now for any $A>|w|\sqrt N$, 
\begin{equation}\label{e:lR0}
P_{w\sqrt N}(t_{3r} < Ns) \le
P_{w\sqrt N}(t_{3r}< T_{A}) +
P_{w\sqrt N} ( T_A \le Ns) .
\end{equation}
To handle the first term, we apply 
Lemma~\ref{l:logprob} with $x=w\sqrt N$ and $a=3r$, 
\begin{equation}\label{hittingbounds}
P_{w\sqrt N}(t_{3r}< T_{A}) 
\le \frac{\log\Big(\frac{A+2r}{w\sqrt{N}}\Big)}
{\log\Big(\frac{A+2r}{3r}\Big)}.
\end{equation}
Now set $A=|w|\sqrt N + \sqrt{Ns}(\log
N)^\alpha$.  Using $s\le t$ and $|w|\le (\log N)^{-\alpha}$, and taking
$N\ge N_0(t)$, we see that for some $C(t)>0$,
\[
\log\Big( 
\frac{A+2r}{|w|\sqrt N}\Big) \le
 \log\Big(1 +\frac{2\sqrt t(\log N)^\alpha}{|w|}
\Big) \le \log\Big(1+\frac{2\sqrt t}
{|w|^2} 
\Big) \le C(t)\log(1/|w|).
\]  
Using $s\ge (\log N)^{-\beta}$ gives for $N\ge N_1(\alpha,\beta)$, 
\begin{align*}
\log \Big(\frac{A+2r}{3r}\Big) \ge \log\Big(\frac{\sqrt{Ns}(\log N)^\alpha}{3r}\Big)&\ge
\log\Big(\frac{\sqrt{N}(\log N)^{\alpha-(\beta/2)}}{3r}\Big)\\
&=\frac12\log N + (\alpha-\frac{\beta}{2})\log\log N-\log(3r)
\ge \frac14\log N.
\end{align*}
Plug the above bounds in \eqref{hittingbounds} to see that for $N\ge N_0(t)\vee N_1(\alpha,\beta)$,
\begin{equation}\label{e:lR1}
P_{w\sqrt N}(t_{3r}< T_{A}) 
\le {4C(t)}\frac{\log(1/|w|)}{\log N}. 
\end{equation}

For the second term in \eqref{e:lR0}, take $k\ge
1/\alpha$ and use \eqref{e:Ysupbnd} to get
\begin{align*}
P_{w\sqrt N} ( T_A \le Ns) &= P_{0}\Big(\sup_{u\le
Ns}|w\sqrt N+ Y_u| \ge A \Big)
\le P_{0}\Big(\sup_{u\le
Ns}|Y_u| \ge A-|w|\sqrt N \Big)\\
&= P_{0}\Big(\sup_{u\le
Ns}|Y_u| \ge \sqrt{Ns}{(\log N)^\alpha} \Big)
\le C_{\ref{e:Ysupbnd}}(k)\frac{(Ns)^k}
{\big(\sqrt{Ns}(\log N)^{\alpha}\big)^{2k}}\\
&\le C \big(\log N\big)^{-2k\alpha} 
\le C (\log N)^{-2} 
\end{align*}
for $Ns\ge N(\log N)^{-\beta}\ge 1$. 
Plugging this bound and \eqref{e:lR1} into \eqref{e:lR0}
we get \eqref{e:lR} for $N\ge N_2$ depending only on
$\alpha,\beta,t$. By the upper bound on $|w|$ we may increase $C_{\ref{e:lR}}$ to get
\eqref{e:lR} for all $N\ge 3$ (recall \eqref{Nrange}). 
\end{proof}

\section{The two particle dual}\label{sec:2dual}

In this section we collect some properties of the two-particle dual which will be needed
in our analysis of the martingale square functions. Our main focus will be on the difference
of the two particles.  
Define
\[
\psi_r(a) = \begin{cases} \rho^2|B_r\cap B_r(a)|&\text{ if }a\neq0\\
\rho|B_r|&\text{ if }a=0,
\end{cases}
\]
and observe that $\psi_r(a)$ is 
decreasing in
$|a|$ and $0\le \psi_r(a)/\rho|B_r| \le 1$. 
Consider the two-particle dual
$(\xi^{1}_t,\xi^{2}_t)$ starting at $(x_1,x_2)$, $x=x_1-x_2$, the difference process $\tilde \xi_t=\tilde\xi^x_t
= \xi^{1}_t-\xi^{2}_t$,  and the 
coalescence time $\tau$ defined in \eqref{e:tauc}.
By the dynamics defining the two-particle dual (recall \eqref{e:d2rates}), the
fact that 
$|B_r(a)\cap B_r(b)| = |B_r(0)\cap B_r(a-b)|$ shows that for $y\neq 0$, 
$\txi$ makes transitions 
\begin{equation}\label{tildexirates}
y \to \begin{cases}
y + \bar U &\text{at rate }2\rho|B_r| - 2\psi_r(y)\\
0 &\text{at rate } \psi_r(y),
\end{cases}
\end{equation}
while $0$ is a trap for $\tilde \xi$. 
Let
\begin{equation}\label{taudef1}\tilde \tau=\tilde \tau(x)=\inf\{t\ge 0:
\txi^x_t=0\}(=\tau)
\end{equation}
be 
the time at which $\txi^x_t$ jumps to 0. 
By standard results (see Sections 4.2 and 4.4 of \cite{EK:MP})
$\tilde \xi^x$ is a pure jump Feller process which is the unique
in law solution of the martingale problem for 
its 
generator $\tilde A$ on the space $B(\R^d)$ of bounded Borel measurable
functions. $\tilde A$ is given by 
\begin{equation}\label{txigen}
\tilde Af(x)=(2\rho|B_r|-2\psi_r(x))\int_{\Rd}(f(x+u)-f(x))
h_{\bar U}(u)\,du+\psi_r(x)(f(0)-f(x)).
\end{equation}

Recall from Section~\ref{sec:rwprel} that $Y^x_t$ is the  rate $2\rho|B_r|$ random
walk starting at $x\in \R^d$ under $P_x$, and with jump distribution that
of  $\bar U$ and generator
$A^Y$ given in \eqref{e:Ygen} for $f\in B(\R^d)$. For a
random variable $V$ we let $Y^V$ denote the same random walk
with initial law that of $V$, and will use this notation
with other Markov processes below.

We  will construct a version of $\tilde \xi^x_t$
by absorbing a random time change of $Y^x$ at $0$.  
Define $\beta(y)= 1-\frac{\psi_r(y)}{\rho|B_r|}$ and 
\begin{equation}\label{e:Idef}
I(t) = \int_0^t \frac1{\beta(Y^x_s)}
ds. 
\end{equation}
Note that for $x\ne 0$,
\[
\sup_{s\le t}\frac{\psi_r(Y^x_s)}{\rho |B_r|}
= \sup_{s\le t}\frac{\rho|B_r(Y^x_s)\cap B_r|}{|B_r|}
< 1 \ a.s.
\]
and thus $\inf_{s\le t}\beta(Y^x_s)>0$ a.s.
This implies that 
$I(t)$ is finite and strictly increasing a.s.\ for all $t$.
Evidently $I(t)=\infty$ for
all $t>0$ if $x=0$.  We 
will allow $x=0$ later, but until otherwise indicated we will take our
initial point $x\neq 0$. 
From the definition of $I$ we see that for $0<s<t$,
\begin{equation}\label{e:Idiff}
t-s \le I(t)-I(s).
\end{equation}
Therefore $I^{-1}(t)$ exists for all $t$ a.s., and 
\begin{equation}\label{e:Iinv}
\int_0^{I^{-1}(t)} \frac1{\beta(Y^x_s)} 
ds = t .
\end{equation}
If we define 
$\tilde Y^x_t=Y^x_{I^{-1}(t)}$,
then it follows from \eqref{e:Iinv} that for all but countably many $t$,
 \[
(I^{-1})'(t)= \beta(Y^x_{I^{-1}(t)})
\]
and therefore that
\begin{equation}\label{e:Iinvint}
I^{-1}(t)= \int_0^t \beta(\tilde Y^x_s)ds .
\end{equation}
Clearly, $I^{-1}(t)\le t$. For $x=0$ it is natural to define
$I^{-1}(t)=0$ for all $t\ge 0$, which means that
$\tilde Y^0_t:=Y^0_{I^{-1}(t)}=0$ for all $t\ge 0$. Thus \eqref{e:Iinvint}
holds for all $x$ and
\begin{equation}\label{e:tYdef}
\tilde Y^x_t = Y^x_{I^{-1}(t)} = 
Y^x\Big(\int_0^t \beta (\tilde Y^x_s)ds\Big)
\quad \forall\ x\in{\R}^d. 
\end{equation}
We may apply Theorems 1.1 and 1.3 of Sec. 6.1 of
\cite{EK:MP}
to see that
$\tilde Y^x$ 
is the unique solution of the martingale problem for
\[A^{\tilde Y}f(x)=\beta(x)A^Yf(x)=(2\rho|B_r|-2\psi_r(x))
\int_{\Rd} (f(x+u)-f(x))h_{\bar U}(u)du, \quad f\in B(\R^d).\]
Here we note that the continuity of $f$ is not needed for Theorem~1.3 of \cite{EK:MP} in our jump process setting as the proof there shows.
Uniqueness of the martingale problem is classical for such bounded jump generators, e.g., see Theorem~4.1 in Chapter~4 of \cite{EK:MP}), and so $\tilde Y^x$ is the unique Feller process with generator $A^{\tilde Y}$, and  in particular is strong Markov.
Finally we send  $\tilde Y^x$ to its absorbing state, $0$
according to the continuous additive functional 
\begin{equation*}C_t=C_t^x=\int_0^t\psi_r(\tilde
Y^x_s)\,ds. 
\end{equation*}
For an independent mean one exponential random
variable, ${\bf e}$, define the absorbing time  
\begin{equation}\label{kappadefn}\kappa =\kappa_x= \inf\{t\ge 0: C^x_t> {\bf e}\},
\end{equation} 
and the absorbed process
\[
\txi^{'x}_t= \begin{cases}
\tilde Y^x(t) &\text{if }t<\kappa\\
0 &\text{if }t\ge \kappa.
\end{cases}
\]
Then $\txi^{'x}$ is a Feller jump process  
and an elementary calculation shows that it solves the martingale
problem for 
$A^{\tilde Y}f(x)+\psi_r(x)(f(0)-f(x))=\tilde A f(x)$,
$f\in B(\R^d)$ (from \eqref{txigen}). From \eqref{tildexirates} we see
that the two-particle dual difference, $\tilde \xi$, is the Feller jump process satisfying the same well-posed martingale problem, and  so, as the notation
suggests, $\txi^{'x}$ has the same law as 
 $\tilde\xi^x$. 

We have proved:
\begin{lemma}\label{lem:difftimechange}
If $(x_1,x_2)\in\R^d\times\R^d$, and $x=x_1-x_2$ define $I^{-1}(t)$ by \eqref{e:Iinv} and set $\tilde Y^x(t)=Y^x(I^{-1}(t))$. If
\begin{equation*}\tilde \xi_t^x=\begin{cases}\tilde Y^x(t) &\text{ if }t<\kappa,\\
0&\text{ if }t\ge \kappa,
\end{cases}\end{equation*}
where  $\kappa=\kappa_x$ is as in \eqref{kappadefn}, then 
\[I^{-1}(t)\int_0^t\beta(\tilde Y^x_s)ds
\] 
 and $\tilde\xi^x$ is a version of the dual difference $\xi^{1}_t-\xi^{2}_t$ under $P_{\{x_1,x_2\}}$.  Moreover
\begin{equation}\label{e:taukap}
\tilde \tau(x)=\kappa_x\text{ for all }x\neq 0.
\end{equation}
\end{lemma}

We often denote the starting point $x$ of $\tilde\xi$ in the
underlying probability as $P_x$. The tail behaviour of the 
coalescing time $\kappa_x$ will be important for us.
Introduce
\begin{equation}\label{kdef}
k(a)=\rho|B_r|\frac{\psi_r(a)}{\rho|B_r|
- \psi_r(a)},\quad a\in\R^d.
\end{equation}

\begin{lemma}\label{lem:killtime} If
$x\in\R^d\setminus\{0\}$, then 
\begin{equation}\label{e:killtime}
P_x(\kappa>t) = E_x\left[ \exp\Big(-\int_0^{I^{-1}(t)}
k(Y_s) ds\Big)\right].
\end{equation} 
\end{lemma}
\begin{proof}
By definition of $\kappa$, 
\[
P_x(\kappa>t) = E_x[ e^{-C_t}] =
E_x\Big[ \exp\Big(-\int_0^t \psi_r(\tilde Y_s) ds
\Big)\Big]=
E_x\Big[ \exp\Big(-\int_0^t \psi_r(Y_{I^{-1}(s)}) ds
\Big)\Big].
\]
Now change variables with $I(u)=s$ and use $I'(u)=
1/\beta(Y_u)$ to get the required expression.
\end{proof}

The following result shows that $I^{-1}(t)$ is close to $t$, and so $Y^x_t$ is a good approximation to $\tilde Y^x_t$.
\begin{lemma}\label{l:Ibnd} There is a constant 
$C_{\ref{l:Ibnd}}>0$
such that for all $0<\alpha<1$ and $t>1$, and 
$Y_0=x, |x|>2r$, or $Y_0=\bar U$,
\begin{equation}\label{e:Ibnd}
P_{Y_0}\big( I^{-1}(t)\notin [t- t^\alpha,t]
\big) \le
C_{\ref{l:Ibnd}}
\begin{cases}
 \log(1+t)t^{-\alpha}&\text{if }d=2,\\
 t^{-\alpha}&\text{if }d\ge3 .
\end{cases}
\end{equation}
\end{lemma}

\begin{proof} Let $Y_0=x$, $|x|>2r$. 
By \eqref{e:Idef} and $Y_s\neq 0$ for all $s$, 
\begin{equation}\label{e:Iminust}
0\le I(t)-t  = \int_0^t \frac
{\rho|B_r|}{\rho|B_r|-\psi_r(Y_s)}-1\,
ds = \int_0^t \frac{\rho|B_r\cap B_r(Y_s)|}{|B_r|-\rho|B_r\cap B_r(Y_s)|}
ds .
\end{equation}
By an elementary argument, there is a 
constant $C_{\ref{e:geom}}= C_{\ref{e:geom}}(d,r)>0$ such that 
\begin{equation}\label{e:geom}
\frac{|B_r\cap B_r(a)|}{|B_r|-|B_r\cap B_r(a)|}
\le \frac{|B_r|}{|B_r|-|B_r\cap B_r(a)|} 1(|a|\le 2r)\le C_{\ref{e:geom}} 
\frac1{|a|} 1\{|a|\le 2r\}, \quad a\in\R^d.
\end{equation}
We are assuming $|x|>2r$, so using the density bound \eqref{e:Ydensbnd}, we see that
\begin{align*}
E_x\Big[\frac1{|Y_s|} 1\{|Y_s|
\le 2r\}\Big]& = \int_0^\infty \frac1{u^2}
P_x(|Y_s|\le u\wedge 2r) du\\
& = \int_0^{2r} \frac1{u^2}
P_x(Y_s\in B_{u})du
+\int_{2r}^\infty \frac1{u^2}
P_x(Y_s\in B_{2r})du\\
&\le \frac{C_{\ref{e:Ydensbnd}}}{s^{d/2}+1}
\Big[ \int_0^{2r} \frac{1}{u^2}|B_u|du 
+\int_{2r}^\infty \frac{1}{u^2}|B_{2r}| du\Big]\\
&= \frac{C(r)}{s^{d/2}+1}.
\end{align*}
On account of \eqref{e:geom}, plugging
this bound into \eqref{e:Iminust} gives
\begin{align*}
E_x\big(I(t)-t\big)  &\le C_{\ref{e:geom}}\int_0^t
\frac{C(r)}{s^{d/2}+1} ds
\le C \begin{cases}
\log(1+t) &\text{if }d=2\\
1&\text{if }d\ge 3.
\end{cases}
\end{align*}
Applying Markov's inequality we obtain
\begin{equation}\label{e:Ibnd2}
P_x\big( I(t)-t\ge t^\alpha
\big) \le
C_{\ref{l:Ibnd}}
\begin{cases}
 \log(1+t)t^{-\alpha}&\text{if }d=2,\\
 t^{-\alpha}&\text{if }d\ge3.
\end{cases}
\end{equation}
This proves \eqref{e:Ibnd}, because by \eqref{e:Idiff}, $P_x\big( t-I^{-1}(t)\ge
t^\alpha\big) \le P_x\big(I(t)-t\ge t^\alpha\big)$, and we also have $I^{-1}(t)\le t$ by $I(t)\ge t$. The proof
for $Y^{\bar U}$ is essentially the same.
\end{proof}

\begin{lemma}\label{l:betabnd}
For $\beta\in
(\frac14,\frac 12)$ there exists a constant 
$C_{\ref{l:betabnd}}=C_{\ref{l:betabnd}}(\beta)>0$ such that 
\[
P(|\tilde Y^{\bar U}_t| \le r_0)
\le C_{\ref{l:betabnd}}t^{2\beta -1}\quad \text{for all }t>0,
r_0\le t^\beta. 
\]
\end{lemma}
\begin{proof} The bounds on $\beta$ imply that 
$0<1-2\beta< 2\beta<1$. This means we can choose
$\alpha\in(0,1)$ such that $1-2\beta<\alpha<2\beta$. For
this $\alpha$ choose $t$ large enough so that $\log(1+t)\ge 1$ and 
$t-t^\alpha>t/2$. By \eqref{e:Ibnd},
\begin{align}\notag
P(|\tilde Y^{\bar U}_{t}| \le r_0)
&= P(|Y^{\bar U}_{I^{-1}(t)}| \le r_0)\\ \notag
&\le P\big(I^{-1}(t)\notin [t-t^\alpha,t]\big) + 
P\big(\inf_{s\in[t-t^\alpha,t]}|Y^{\bar U}_{s}|\le r_0 
\big)\\ \notag
&\le C_{\ref{l:Ibnd}}\log(1+t)t^{-\alpha} +
P(|Y^{\bar U}_{t-t^\alpha}|\le r_0+t^\beta) + 
\\ \label{e:beta1}
&\qquad + P\big(|Y^{\bar U}_{t-t^\alpha}|> 
r_0+t^\beta, \inf_{s\in[t-t^\alpha,t]}|Y^{\bar U}_s|\le
  r_0\big) .  
\end{align}
By $r_0\le t^\beta$, \eqref{e:Ydensbnd}, and the above choice of $t$,
\begin{align*}
P\big(|Y^{\bar U}_{t-t^\alpha}|\le 
r_0+t^\beta\big) \le 
C\frac{(r_0+t^{\beta})^d}
{(t-t^{\alpha})^{d/2}}
\le 
C\frac{(2t^{\beta})^d}{(t/2)^{d/2}}
\le Ct^{\frac d2 (2\beta-1)}
\le Ct^{2\beta -1}.
\end{align*}
For the last term in \eqref{e:beta1}, 
the Markov property and \eqref{e:Ysupbnd} imply that (recall
$t> 1$) for all $k\ge 1$,
\begin{align*}
P\big(|Y^{\bar U}_{t-t^\alpha}|> 
r_0+t^\beta, \inf_{s\in[t-t^\alpha,t]} |Y^{\bar U}_s|\le r_0\big)  &\le 
P\big(
\sup_{s\le t^\alpha}
|Y^0_s|\ge t^\beta\big) 
\le C_{\ref{e:Ysupbnd}}(k) t^{k(\alpha-2\beta)}.
\end{align*}
Plugging these bounds into \eqref{e:beta1} gives
\begin{align*}
P(|\tilde Y^{\bar U}_t|\le r_0) \le 
C(t^{-\alpha} \log(1+t) +t^{2\beta -1}+
t^{k(\alpha-2\beta)}). 
\end{align*}
Recalling that $\alpha>1-2\beta$ and $\alpha-2\beta<0$, and
choosing $k$ large such that $k(\alpha-2\beta)<2\beta-1$, it
follows that for large $t$,
$P(|\tilde Y^{\bar U}_t|\le r_0) \le
Ct^{2\beta-1}$. Increasing $C$ appropriately to handle small $t$ completes the proof.
\end{proof}

For $0<a<b$ define
\begin{equation}\label{Gdefn}
G(a,b)  = \Big\{ |Y_{I^{-1}(u)}| >2r \ \forall \ 
u\in[a,b]\Big\},
\end{equation}
and for $q\ge 1$ introduce
\begin{equation}\label{sNdef}
s_N=(\log N)^{-q}.
\end{equation}

\begin{lemma}\label{l:GE} There is a constant
$C_{\ref{e:GNbnd}}(q)>0$ such that 
\begin{equation}\label{e:GNbnd}
P_x\Big(G\Bigl(\frac{Ns_N}4,2Ns_N\Bigr)^c\Big) 
\le C_{\ref{e:GNbnd}}\frac{\log\log N}{\log N}
\quad\text{ for all }|x|>2r.
\end{equation}
\end{lemma}

\begin{proof} Recall from \eqref{Nrange} that $N\ge 3$. 
Let $u_N=Ns_N/4$ and $u'_N=u_N-u_N^{1/2}$.
Then since $I^{-1}(u)\le u$, 
\begin{equation}\label{e:GNbnd1}
P_x\Big(G\big(u_N,{2Ns_N}\big)^c\Big) 
\le P_x\Big(  I^{-1}(u_N ) \le u'_N
\Big) + P_x( |Y_u|\le 2r \text{ for some }u \in 
[u'_N,{2Ns_N}]). 
\end{equation}
By \eqref{e:Ibnd} with $\alpha=1/2$ and $N\ge N_0(q)$ (recall $|x|>2r$), 
\begin{equation}\label{e:GNbnd2}
P_x\Big(  I^{-1}(u_N) \le u'_N \Big)
\le C\frac{\log (1+u_N)}{\sqrt{u_N}}
\le \frac{C}{ N^{1/4}}. 
\end{equation}
Next, using the Markov property at time $u'_N$, we have for $N\ge N_0(q)$,
\begin{align*}
P_x&( |Y_u|\le 2r \text{ for some }u \in 
[u'_N,{2Ns_N}])\\ &\le 
P_x( |Y_{u'_N}|\le s_N\sqrt N) +
\sup_{|y|\ge s_N}P_{y\sqrt N}(t_{2r} \le {2Ns_N}-u'_N)\\
&\le 
\exp( -
  2\rho|B_r|u'_N)+ C\frac {(s_N \sqrt N)^d} {(u'_N)^{d/2}} 
+ C \frac{\log(1/s_N) }{\log N}\quad\text{(by \eqref{e:Ydensbnd}, \eqref{e:lR} (if $d=2$) and \eqref{e:tafinite} (if $d=3$))}\\
&\le 
\frac{C}{\log N} +
C \frac{\log\log N }{\log N}.
\end{align*}
In the next to last line we have used the $d=2$ bound; if $d\ge 3$, \eqref{e:tafinite} gives a much smaller bound. Here we have also used the strong Markov property  and applied \eqref{e:lR} (if $d=2$) with $w\sqrt N$ equal to the location of the first jump into $B_{s_N\sqrt N}$.
Use this bound and \eqref{e:GNbnd2} in
\eqref{e:GNbnd1} derive \eqref{e:GNbnd} for $N\ge N_0(q)$. Now adjust $C_{\ref{e:GNbnd}}$ to handle
the remaining values of $N$.
\end{proof}

We will also need a bound on the two-particle dual $\xi_t=(\xi^{1},\xi^{2}_t)$ after the coalescing time $\kappa$ 
for any $d\ge 2$.  In this setting assume
\begin{align}\label{Ws} &W^{1,x_1}, W^{2,x_2}\text{ and }W^{3,0}\text{ are independent rate $\rho|B_r|$
random walks in $\R^d$ with step}\\
\nonumber&\text{distribution $\bar U$ (now in $\R^d$) and starting at points $x_1,x_2,0\in\R^d$, respectively.}
\end{align}
  Define $W_t=(W^{1,x_1}_t,W^{2,x_2}_t)$,
  and 
\[\psi_W(y_1,y_2)=k(|y_1-y_2|)1(y_1\neq y_2)={{\rho|B_r|\psi_r(y_1-y_2)}\over{\rho|B_r|-\psi_r(y_1-y_2)}}1(y_1\neq y_2),\]
and $D(t)=\int_0^t\psi_W(W_s)ds$. Although $\psi_W(y_1-y_2)$ becomes unbounded (if $\rho=1$) as $y_1-y_2\to0$, as for $I(t)$, $D(t)<\infty$ for all $t>0$ a.s. 
Let ${\bf e}$ be an independent exponential mean $1$ random
variable, and introduce $\bar\kappa=\inf\{t\ge
0:D(t)>{\bf e}\}\le\infty$ (it will be a.s. finite if $d=2$).
Assume also that conditional on
$(W,W^{3,0},{\bf e})$, $U_{W(\bar\kappa)}$ has a uniform
distribution on $B_r(W^{1,x_1}_{\bar\kappa-})\cap
B_r(W^{2,x_2}_{\bar\kappa-})$ (the intersection
is non-empty a.s. by the definition of $\bar\kappa$ because
$\psi_W(y_1,y_2)=0$ if $|y_1-y_2|>2r$), and given
$(W,W^{3,0},{\bf e},U_{W(\bar\kappa)})$, $U$ is an
independent r.v. uniformly distributed on $B_r$.  We can use
the above to define a version of our two-particle dual but
now ``run at a constant rate" by
\begin{equation}\label{barWdefn}
\overline{W}_t=\begin{cases} (W^{1,x_1}_t,W^{1,x_1}_t)&\text{ if }x_1=x_2,\\
W_t&\text{ if }t< \bar\kappa\text{ and }x_1\neq x_2,\\
(U+U_{W_{\bar\kappa}}+W^{3,0}_{t-\bar\kappa},U+U_{W_{\bar\kappa}}+W^{3,0}_{t-\bar\kappa})&\text{ if }t\ge \bar\kappa\text{ and }x_1\neq x_2.
\end{cases}
\end{equation}
Note that when $\bar W$ is at $(y_1,y_2)$, $y_1\ne
  y_2$, $\bar W$ jumps to the diagonal in
  $\Rd\times\Rd$ at rate $\psi_W(y_1,y_2)$.
  
We now extend our earlier time change and define
\begin{align}
\label{betadef1}\beta(y_1,y_2)&=\Bigl[1-\frac{\psi_r(y_1-y_2)}{\rho|B_r|}\Bigr]1(y_1\neq y_2)+1(y_1=y_2)\le 1,\quad \forall y\in\R^d,\\
\label{barIdefn}\bar I(t)&=\int_0^t\beta(\overline{W_s})^{-1}\,ds<\infty \quad\forall  t>0,\\
\label{barxidef}\overline{\xi _t}&=\overline W({\bar I}^{-1}(t)),
\end{align}
where ${\bar I}^{-1}$ is the inverse of the strictly increasing continuous function $\bar I$. 
As in \eqref{e:Iinvint}, one sees that 
\begin{equation}\label{Iinvbeta}
  \bar I^{-1}(t)=\int_0^t\beta(\bar \xi_s)\,ds\quad\forall
  t\ge 0.
\end{equation}
The following result shows that $\bar\xi_t$ is a
version of the two particle dual. 
\begin{lemma}\label{lem:twopartdualconst} If $(x_1,x_2)\in \R^d\times\R^d$ define $\overline
  W$ by \eqref{barWdefn}, $\beta$ by \eqref{betadef1}, $\bar
  I$ by \eqref{barIdefn}, and $\bar\xi$ by \eqref{barxidef}.  
Then
$$\bar I^{-1}(t)=\int_0^t\beta(\bar\xi_s)ds,$$
and $\bar\xi$ is a version of the two-particle dual described in \eqref{e:d2rates} under $P_{\{x_1,x_2\}}$.
\end{lemma}
\begin{proof}
  The jump rate of
$\overline W$ to the diagonal becomes unbounded as it
approaches the diagonal (for $\rho=1$), 
so we proceed more carefully than in the proof of
Lemma~\ref{lem:difftimechange}, making use of 
optional stopping.  Let  
$$R_n=\{(y_1,y_2)\in \R^d\times\R^d:0<|y_1-y_2|< n^{-1}\}$$
and
\[T^W_n=\inf\{t\ge 0:\, \overline W_t\in R_n\}\le \infty.\]
Then $\overline{W_n}(t)=\overline W(t\wedge T^W_n)$ is a pure jump process on $\R^d\times \R^d$ with {\it bounded} jump rates and generator
\begin{align*}{\bar A_nf(y)=&\rho|B_r|[E(f(y_1+\bar U,y_2)+f(y_1,y_2+\bar U)-2f(y))]1(|y_1- y_2|\ge 1/n)\cr
&+ \psi_W(y)[E(f(U+U_y,U+U_y)-f(y))]1(|y_1- y_2|\ge 1/n)\cr
&+\rho|B_r|[E(f(y+(\bar U,\bar U))-f(y)])1(y_1=y_2).}\end{align*}
Here $U_y$ is uniformly distributed on $B_r(y_1)\cap B_r(y_2)$ and is independent of the uniform (on $B_r$) r.v. $U$. 
It is easy to
check that $\overline{W_n}$ solves the martingale problem
for $\bar A_n$ on the domain $B(\R^d\times\R^d)$ of bounded Borel functions on $\R^d\times\R^d$.

Let $\bar T_n=\inf\{t\ge 0:\bar \xi_t\in R_n\}\le
\infty$. Using the properties of $\bar I^{-1}$ and
\eqref{Iinvbeta}, it is easy to check that
\[
  T^W_n=\bar I^{-1}(\bar T_n) = \int_0^{\bar T_n}\beta(\bar
  \xi_s)ds.
  \]
  It follows that
  \[
{I^{-1}(t\wedge \bar T_n)} = I^{-1}(t)\wedge T^W_n =
\int_0^t\beta(\bar\xi_s)ds\wedge T^W_n =
\int_0^t\beta(\bar\xi(s\wedge \bar T_n))ds \wedge T^W_n  .
    \]
If we define $\bar \xi^{\bar T_n}_t = \bar\xi(t\wedge\bar
T_n)$, the above implies $\ds  \bar\xi^{\bar T_n}_t = \bar
W_n\Big(\int_0^t{\beta(\bar\xi^{\bar T_n}_s)}\Big)ds$, and thus
we may apply Theorems~1.1 and 1.3 in Chapter 6 of \cite{EK:MP} to
conclude that $\bar \xi^{\bar T_n}_t$  solves the martingale
problem for 
\[G_nf(y)=\beta(y)\bar A_nf(y),\quad f\in B(\R^d\times\R^d).\]
Here we recall again that the continuity of $f$ assumed in Ch. 6 Theorem 1.3 of \cite{EK:MP} is not needed in our jump process setting.  A bit of arithmetic shows 
\begin{align*}
G_nf(y)=(&\rho|B_r|-\psi_r(y_1-y_2))[E(f(y_1+\bar U,y_2)+f(y_1,y_2+\bar U)-2f(y))]1(|y_1-y_2\ge 1/n)\\
&\ +\psi_r(y_1-y_2)[E(f(U+U_y,U+U_y)-f(y))]1(|y_1-y_2\ge 1/n)\\
&\ +\rho|B_r|\,[E(f(y+(\bar U,\bar U))-f(y))]1(y_1=y_2).
\end{align*}
If $\xi=(\xi^1,\xi^2)$ is the two-particle dual process, as described in \eqref{e:d2rates}, the above is the generator of the Feller pure jump process $\xi(t\wedge T_n)$, where
$T_n=\inf\{t\ge 0:\xi_t\in R_n\}$ and so $\xi^{T_n}(t)=\xi(t\wedge T_n)$ also solves the martingale problem for $G_n$ ($f\in B(\R^d\times\R^d)$).  By well-posedness of this martingale problem ((Section 2 and Thm. 4.1 of Chapter 4 of \cite{EK:MP}) we concluded that $\xi^{T_n}$ and $\bar \xi^{\bar T_n}$ are identical in law for all $n\in\N$. Since $R_n\downarrow\emptyset$ and $\xi(T_n),\bar \xi(\bar T_n)\in R_n$ when
these times are finite, it follows that $T_n,\bar T_n\uparrow \infty$ a.s. as $n\to\infty$ (in fact for large
$n$ they will be infinite a.s.), and therefore $\xi$ and $\bar\xi$ are identical in law.    
\end{proof}

The following result is now an easy consequence of \eqref{barWdefn}, Lemma~\ref{lem:twopartdualconst}  and the bound $\bar I^{-1}(t)\le t$ for all $t\ge 0$.

\begin{lemma}\label{lem:xibound}
Assume $\xi^x$ is the two-particle dual in $\R^d\times \R^d$, starting at $x=(x_1,x_2)$. Then
we may assume there are random walks $W^{i,x_i}$ ($i=0,1,2, x_0=0$) as in \eqref{Ws} such that
\begin{equation}\label{nicexibnd}
\sup_{s\le t}|\xi_s^x|\le \Bigl[\sum_{i=0}^3\sup_{s\le t}|W^{i,x_i}_s|\Bigr]+2\sqrt 2 r.
\end{equation}
\end{lemma}
\section{Proof of Main Result}\label{sec:mainresult}
The proof of Theorem~\ref{t:main} proceeds by taking limits as $N\to\infty$ in Proposition~\ref{p:sm-decomp} to derive the martingale problem for the limiting super-Brownian motion. The main issue is the identification of the square function of the limiting martingale part and the key here is the following result:
\begin{proposition}\label{p:2did} For all
  $A,T>0$, and $\phi\in 
C^3_0$, 
\begin{equation}\label{e:2did}
\sup_{X^N_0(\1)\le A}E\left[ \sup_{0\le t\le T}\Big| \langle M^N(\phi)\rangle_t -
\int_0^t \rho^2|B_r|^2\gamma_e X^N_s(\phi^2)ds\Big|\right] \to 0 \text{
as }N\to\infty.
\end{equation}
\end{proposition}

This will be proved in Section~\ref{sec:sqfunction}.
In this section we will establish Theorem~\ref{t:main}, assuming this result. If $S$ is a metric space, recall that a sequence of laws on $D(\R_+,S)$ is $C$-tight iff it is tight and all limit laws are continuous. $C$-tightness on $D(\R_+,S)\times C(\R_+,S)$ is then defined in the obvious manner. 
The first step is to prove: 
\begin{lemma} \label{l:projt} If  $\phi\in C_0^3(\R^d)$, then $\{
 (X^N(\phi), \langle M^N(\phi)\rangle):N\ge 3\}$ is $C$-tight in $D(\R_+,\R)\times C(\R_+,\R)$.
  \end{lemma}
\begin{proof}   Let $\phi\in C_0^3(\R^d)$ and $A^N(t)=\int_0^t \rho^2|B_r|^2\gamma_eX^N_s(\phi^2)\,ds$.  
Then for $0\le s<t\le T$,
\[E(|A^N_t-A^N_s|^2)\le CE(\sup_{u\le T}X^N_u(\1)^2)(t-s)^2\le C(T)(t-s)^2,\]
by Corollary~\ref{c:L2bnd} and \eqref{ICscaling}. 
Therefore, the collection of continuous increasing processes
$\{A^N:N\ge 3\}$ is tight, and hence relatively compact, in
$C(\R_+,\R)$ by Prohorov's theorem. It then follows from
Proposition~\ref{p:2did} that the sequence of continuous
(recall \eqref{e:MNs}) increasing processes $\{\langle
M^N(\phi)\rangle_\cdot:N\ge 3\}$ is relatively compact in
$C(\R_+,\R)$, and so also tight by Prohorov's theorem again. 

Next, recall $D^N(\phi)$ from Proposition~\ref{p:sm-decomp}. Lemma~\ref{l:dNs} implies that $D^N_t(\phi)=\int_0^t d_s^N(\phi)ds$, where \[|d_s^N(\phi)|\le C_\phi X_s^N(\1)\quad\forall s\ge 0.\] 
Therefore if $0\le s<t\le T$, then by the above and Corollary~\ref{c:L2bnd},
\begin{align*}
E((D^N_t(\phi)-D^N_s(\phi)^2)\le C_\phi^2\int_s^t\int_s^tE(X_u^N(\1)X_v^N(\1))dvdu
\le C(\phi,T)(t-s)^2.
\end{align*}
This implies $\{D^N(\phi):N\ge 3\}$ is tight in $C(\R_+,\R)$.  Lemma~\ref{increments} and Proposition~\ref{p:sm-decomp} imply
 \begin{equation}\label{incrbound}
\sup_{s\le T}|\Delta M_s^N(\phi)|=\sup_{s\le T}|\Delta X^N_s(\phi)|\to 0\text{ a.s. as }N\to\infty.
\end{equation}
Using \eqref{incrbound} and the $C$-tightness of $\{\langle M^N(\phi)\rangle_\cdot:N\ge 3\}$, 
established above,  
in Theorem VI.4.13 and Proposition IV.3.26 of
\cite{JS:LT}, we see that $\{M^N(\phi):N\ge 3\}$ is $C$-tight in $D(\R_+,\R)$. $C$-tightness of $\{X^N(\phi)\}$ now follows from the above, our assumption on the initial conditions $\{X^N_0\}$ (i.e., \eqref{ICscaling}), and the semimartingale decomposition in Proposition~\ref{p:sm-decomp}.
Having obtained $C$-tightness of each component, the result is now immediate.
\end{proof}

\begin{prop}\label{p:tight} For $d\ge 2$,
under the assumptions of Theorem~\ref{t:main}, the family \hfil\break $\{P(X^N\in \cdot): N\ge 3\}$ is $C$-tight in  
$D(\R_+, \cM_F(\R^d))$.
\end{prop}  
\begin{proof}
By the Kurtz-Jakubowski theorem (e.g. see Proposition~3.1 in \cite{CDP}) it suffices to show:
\begin{enumerate}
\item For each $T,\vep>0$ there is a compact set
  $K_{T,\vep}\subset \Rd$ such that
\begin{equation}\label{e:tight1}
\limsup_{N\to\infty} P\Big( \sup_{t\le T} X^N_t(K^c_{T,\vep})
>\vep\Big)<\vep.
\end{equation}
\item For each $T>0$,
\begin{equation}\label{e:tight2}
\lim_{H\to\infty}\limsup_{N\to\infty} 
P\big( \sup_{t\le T} X^N_t(\1)\ge H\big)= 0. 
\end{equation}
\item For each $\phi\in C^{\infty}_0(\Rd)$, 
\begin{equation}\label{e:tight3}
\{X^{N}_\cdot(\phi), N\ge 3\} \text{ is $C$-tight in $D(\R_+,\R)$.}
\end{equation}
\end{enumerate}
The last \eqref{e:tight3} holds by Lemma~\ref{l:projt}, and \eqref{e:tight2} is immediate from Corollary~\ref{c:L2bnd}. The compact containment \eqref{e:tight1} is proved exactly as for the voter model in Lemma~3.3 of \cite{CDP}.  The argument there will now use the semimartingale decomposition in Proposition~\ref{p:sm-decomp}, the convergence of the initial states from \eqref{ICscaling}, {Lemma~\ref{l:dNs}}, and first moment bounds which are immediate from \eqref{e:dualXN}. This completes the proof.
\end{proof}

\medskip
We are ready to turn to the main result.\\
 
\noindent{\it Proof of Theorem~\ref{t:main}.} By Proposition~\ref{p:tight} it suffices to show that every weak subsequential limit is the super-Brownian motion described in the Theorem.  Fix $\phi\in C_0^3(\R^d)$. 
By Lemma~\ref{l:projt} and Skorokhod's theorem, and then taking a further subsequence, we may assume that we are on a probability space where
\begin{equation}\label{asconv} (X^{N_k},\langle M^{N_k}(\phi)\rangle)\rightarrow (X,A^\phi)\text{ a.s. in }D(\R_+,\MF(\R^d))\times C(\R_+,\R).
\end{equation}
Since the limit is continuous a.s. one has in fact a.s. uniform convergence on compact time intervals. It also follows from the above and Corollary~\ref{c:L2bnd} that 
\begin{align} \label{totalmassconvergence}\nonumber \sup_{t\le T} [|X_t^{N_k}(\1)- X_t(\1)|+|X_t^{N_k}(\phi)- X_t(\phi)|+&|X_t^{N_k}(\phi^2)- X_t(\phi^2)|+|X_t^{N_k}(\Delta\phi)- X_t(\Delta\phi)|]\\&\rightarrow 0
\text{ a.s. and in $L^1$ as }k\to\infty \text{ for all }T>0.
\end{align}
This and Proposition~\ref{p:2did} show that 
\begin{equation}\label{Aphi}
A^\phi_t=\int_0^t \rho^2|B_r|^2\gamma_e X_s(\phi^2)ds\text{ for all }t\ge 0.
\end{equation}
It follows from \eqref{asconv}, \eqref{totalmassconvergence}, and Proposition~\ref{p:2did} that 
\begin{equation}\label{sqfnconv}
\sup_{t\le T}|\langle M^{N_k}(\phi)\rangle_t-A^\phi_t|\rightarrow 0\text{ a.s. and in $L^1$ as }k\to\infty\text{ for all }T>0,
\end{equation}
Lemma~\ref{l:dNs}, Corollary~\ref{c:L2bnd}, and \eqref{totalmassconvergence} imply
\begin{equation}\label{Dconvergence} \sup_{t\le T}\Big|D_t^{N_k}(\phi)-\int_0^t\rho|B_r|\bar\sigma^2X_s(\Delta\phi/2)ds\Big|\rightarrow 0\text{ as }k\to\infty\text{ a.s. and in $L^1$ as }k\to\infty.
\end{equation}
Define an a.s. continuous process by $M_t(\phi)=X_t(\phi)-X_0(\phi)-\int_0^t\rho|B_r|\bar\sigma^2X_s(\Delta\phi/2)ds$. Then the above, the convergence of the initial conditions in \eqref{ICscaling}, and the semimartingale decomposition for 
$X^N(\phi)$ in Proposition~\ref{p:sm-decomp} show that
\begin{equation}\label{Mconvergence}
\sup_{t\le T}|M^{N_k}_t(\phi)-M_t(\phi)|\rightarrow 0\text{ in $L^1$ and a.s. as }k\to\infty\text{ for all }T>0.
\end{equation}
Since $M^{N_k}(\phi)$ is an $(\cF^{X^{N_k}}_t)$-martingale by Corollary~\ref{c:L2bnd}, 
it follows from the above that $M(\phi)$ is a continuous martingale and a standard argument (e.g. see the proof of Theorem~3.5 in \cite{CDP}) shows it is in fact an $(\cF^X_t)$-martingale. 

Recalling \eqref{SBMsqfn}, \eqref{Aphi}, and the value of $b$ in Theorem~\ref{t:main}, it remains to identify the square function of $M(\phi)$ as $A^\phi$ by showing
\begin{equation}\label{sqfnid} 
M_t(\phi)^2-A_t(\phi)\text{ is a local martingale.}
\end{equation}
For $d\ge 3$ this is fairly easy, but we give a stopping argument to include the more delicate $2$-dimensional case.
For $J\in \N$ define
\[T^N_J=\inf\{t:|M^N_t(\phi)|\ge J\},\quad T_J=\inf\{t:|M_t(\phi)|\ge J\}.\]
The convergence in \eqref{Mconvergence} readily shows that 
\begin{equation}\label{lscTJ}\liminf_{k\to\infty}T_J^{N_k}\ge T_J\quad\forall J\in\N\ \text{a.s.}
\end{equation}
We claim that 
\begin{equation}\label{Mstopconv}
\lim_{k\to\infty}\sup_{t\le T}\Big[|M^{N_k}_{t\wedge T_J^{N_k}}(\phi)-M_{t\wedge T_J}(\phi)|+|\langle M^{N_k}(\phi)\rangle_{t\wedge T^{N_k}_J}-A^\phi_{t\wedge T_J}|\Big]=0\text{ for all }T>0\ \text{a.s.}
\end{equation}
The reason there is an issue here is that we do not know whether or not $\lim_kT_J^{N_k}=T_J$ a.s. 
It follows from \eqref{lscTJ} that for $t\le T_J$ we have $\lim_kT_J^{N_k}\wedge t=t=T_J\wedge t$  (the convergence is uniform for $t\le T_J\wedge T$ for any fixed $T$) and therefore by \eqref{Mconvergence} and \eqref{sqfnconv},
\begin{equation}\label{tleTJ}
\lim_{k\to\infty}\sup_{t\le T_J\wedge T}\Big[|M^{N_k}_{t\wedge T_J^{N_k}}(\phi)-M_{t\wedge T_J}(\phi)|+|\langle M^{N_k}(\phi)\rangle_{t\wedge T^{N_k}_J}-A^\phi_{t\wedge T_J}|\Big]=0\text{ for all }T>0\ \text{a.s.}
\end{equation}
A simple calculation using \eqref{lscTJ} shows that ($\sup\emptyset:=0$) with probability one for any $T>0$,
\begin{align}\label{bigt}
\limsup_{k\to\infty}&\sup_{T_J^{N_k}\le t\le T}\Big[|M^{N_k}_{t\wedge T_J^{N_k}}(\phi)-M_{t\wedge T_J}(\phi)|+|\langle M^{N_k}(\phi)\rangle_{t\wedge T^{N_k}_J}-A^\phi_{t\wedge T_J}|\Big]\\
\nonumber&=\limsup_{k\to\infty}1(T_J^{N_k}\le T)\Big[|M^{N_k}_{T_J^{N_k}}(\phi)-M_{T_J}(\phi)|+|\langle M^{N_k}(\phi)\rangle_{T_J^{N_k}}-A^\phi_{T_J}|\Big].
\end{align}
In view of the above and \eqref{tleTJ}, to prove \eqref{Mstopconv} it suffices to show that for $T>0$ fixed,
\begin{equation}\label{smallintconv}
\limsup_{k\to\infty} \sup_{T_J< t\le T_J^{N_k}\wedge T}\Big[|M_t^{N_k}(\phi)-M_{T_J}(\phi)|+|\langle M^{N_k}(\phi)\rangle_t-A^\phi_{T_J}|\Big]=0\text{ a.s.}
\end{equation}
By \eqref{sqfnconv} and \eqref{Mconvergence} this would follow from
\begin{equation}\label{smallintMconv}
\limsup_{k\to\infty} \sup_{T_J< t\le T_J^{N_k}\wedge T}\Big[|M_t(\phi)-M_{T_J}(\phi)|+|A^\phi_t-A^\phi_{T_J}|\Big]=0\text{ a.s.}
\end{equation}

For this we will use the following lemma, whose proof is deferred to the end of this section.
\begin{lemma}\label{l:flatspots} With probability one, for all $0\le s<t$, $M_u(\phi)=M_s(\phi)$ for all $u\in[s,t]$ implies that $A^\phi_t=A^\phi_s$. 
\end{lemma}
By the Dubins-Schwarz theorem we may assume
$M_t(\phi)=B(\langle M(\phi)\rangle_t)$ for some
  Brownian motion, $B$, on our probability space. We let
  $T^B_J$ and $T^B_{J+}$ denote the exit times of $B$ from
  $(-J,J)$ and $[-J,J]$, respectively.  
On $\{\limsup_k T_J^{N_k}\wedge T\le T_J\}$, \eqref{smallintMconv} follows from the a.s. continuity of $M(\phi)$ and $A^\phi$. So assume $\omega$ is in $\{\limsup_kT^{N_k}_J\wedge T>T_J\}$ and also outside of a null set so that:\\
(i) \eqref{sqfnconv} and \eqref{Mconvergence} hold;\\
(ii) for all $s<t$, $\langle M(\phi)\rangle_t=\langle M(\phi)\rangle_s$ implies $M_u(\phi)=M_s(\phi)$ for all $u\in[s,t]$; \\
(iii) the conclusion of Lemmas~\ref{increments} and
\ref{l:flatspots} hold;\\
 (iv) $T^B_J=T^B_{J+}$. \\ 
Use the conclusion of Lemma~\ref{increments} and our choice of $\omega$ to see that 
\begin{equation*} \forall t\in[0,\limsup_k T^{N_k}_J\wedge T),\quad |B(\langle M(\phi)\rangle_t)|=|M_t(\phi)|=\lim_k|M^{N_k}_t(\phi)|\le \limsup_kJ+C\frac{\log N_k}{N_k}=J.
\end{equation*}
We easily see that $T^B_J=\langle(M(\phi)\rangle_{T_J}$, and so the above shows that $|B(u)|\le  J$ on \break
$[T^B_J, \langle M(\phi)\rangle(\limsup_k T^{N_k}_J\wedge T))$, so the fact that $T^B_{J+}=T^B_J$ (by our choice of $\omega$) implies this interval must be empty.  We conclude that $\langle M(\phi)\rangle(\limsup_k T^{N_k}_J\wedge T)= \langle(M(\phi)\rangle(T_J)$, which by our choice of $\omega$ implies that $M_u(\phi)=M_{T_J}(\phi)$ for all $u\in[T_J,\limsup_kT_J^{N_k}\wedge T]$.  Lemma~\ref{l:flatspots} (and again our choice of $\omega$) shows that this gives $A^\phi_u=A^\phi_{T_J}$ for all $u\in[T_J,\limsup_kT_J^{N_k}\wedge T]$.  We have proved \eqref{smallintMconv}, and hence completed the derivation of \eqref{Mstopconv}.

Turning at last to \eqref{sqfnid}, we see from Lemma~\ref{increments} and Proposition~\ref{p:sm-decomp} that 
\begin{equation}\label{Munifbound}
\sup_t|M^{N_k}(\phi)_{T^{N_k}_J\wedge t}|\le J+C_\phi\frac{\log N_k}{N_k}\le C(J,\phi)\quad\text{for all }k.
\end{equation}
This, together with the $L^1$ convergence of $\langle M^{N_k}(\phi)\rangle_t$ to $A^\phi_t$ from \eqref{sqfnconv}, implies that \break
$\{\sup_{t\le T}|[M^{N_k}(\phi)_{T^{N_k}_J\wedge t}]^2-\langle M^{N_k}(\phi)\rangle_{T^{N_k}_J\wedge t}|:k\in\N\}$ is uniformly integrable.  From \eqref{Mstopconv} and the above we can conclude that for each $t\ge 0$, 
\begin{align*}N^k_{T^{N_k}_J\wedge t}&:=(M^{N_k}(\phi)_{T^{N_k}_J\wedge t})^2-\langle M^{N_k}(\phi)\rangle_{T^{N_k}_J\wedge t}\\
&\rightarrow M(\phi)_{T_J\wedge t}^2-A^\phi_{T_J\wedge t} \text{ in $L^1$ and a.s.}
\end{align*}
As $N^k_{T_J^{N_k}\wedge t}$ is a martingale (Corollary~\ref{c:L2bnd}), this implies that $N_t=M(\phi)_{T_J\wedge t}^2-A^\phi_{T_J\wedge t} $ is a martingale for all $J$ which establishes \eqref{sqfnid} and so completes the proof.\qed

\medskip

\noindent{\it Proof of Lemma~\ref{l:flatspots}.} By continuity it suffices to prove the result for a fixed pair of times $0\le s<t$. For $J,k,n\in\N$, define
\[U_n^{J,k}=\inf\{u\ge s: |(M^{N_k}(\phi)(T^{N_k}_J\wedge u))^2-(M^{N_k}(\phi)(T^{N_k}_J\wedge s))^2|\ge n^{-1}\}\quad(\inf\emptyset=\infty).\]
It follows from our jump bounds in \eqref{Mjumps} that 
\[|\Delta(M^{N_k}_s(\phi)^2)|\le 2J\Vert\phi\Vert_\infty C\frac{\log N_k}{N_k}\quad\text { for all }s\ge 0.\]
Recalling that $M^N(\phi)$ is a square integrable martingale (from Corollary~\ref{c:L2bnd}), we have by optional stopping,
\begin{align} \label{sqfnincrbound}
E\Big(&\langle M^{N_k}(\phi)\rangle_{U_n^{k,J}\wedge T_J^{N_k}\wedge t}-\langle M^{N_k}(\phi)\rangle_{T_J^{N_k}\wedge s}\Big)\\
\nonumber&=E\Big(M^{N_k}(\phi)(U^{k,j}_n\wedge T_J^{N_k}\wedge t)^2-M^{N_k}(\phi)(U^{k,j}_n\wedge T_J^{N_k}\wedge s)^2\Big)\\
\nonumber&\le \frac{1}{n}+2J\Vert\phi\Vert_\infty C\frac{\log N_k}{N_k}.
\end{align}
Next use \eqref{lscTJ}, and the convergence in \eqref{sqfnconv} and \eqref{Mconvergence}, together with Fatou's lemma, to see that 
\begin{align*}
E\Big(&(A_t^\phi-A_s^\phi)1(T_J>t)1\Big(\sup_{s\le u\le t}|M_u(\phi)^2-M_s(\phi)^2|<\frac{1}{2n}\Big)\Big)\\
&\le E\Big(\liminf_{k\to\infty}(\langle M^{N_k}(\phi)\rangle_t-\langle M^{N_k}(\phi)\rangle_s)1(T_J^{N_k}>t)1(\sup_{s\le u\le t}|M_u^{N_k}(\phi)^2-M_s^{N_k}(\phi)^2|<\frac{1}{n})\Big)\\
&\le \liminf_{k\to\infty} E\Big((\langle M^{N_k}(\phi)\rangle_{t\wedge T_J^{N_k}\wedge U_n^{k,J}}-\langle M^{N_k}(\phi)\rangle_{s\wedge T_J^{N_k}})1(T_J^{N_k}>t)1(\sup_{s\le u \le t}|M_u^{N_k}(\phi)^2-M_s^{N_k}(\phi)^2|<\frac{1}{n})\Big)\\
&\le \frac{1}{n},
\end{align*}
where the last is by \eqref{sqfnincrbound}. Let $J\to\infty$ and then $n\to\infty$ to prove the result for $s<t$ fixed, as required.\qed

\section{Analysis of the square function: Proof of Proposition~\ref{p:2did}}\label{sec:sqfunction}
In this section we analyze the martingale square function $\langle
M^N(\phi)\rangle_t$ for $\phi\in C_0^3(\R^d)$, and in particular give the Proof of Proposition~\ref{p:2did}.  We recall from
Proposition~\ref{p:sm-decomp} and 
Lemma~\ref{l:mNs} that 
\[
\langle M^N(\phi)\rangle_t = \int_0^t m^N_s(\phi) ds
= \int_0^t \big (\bar m^N_s(\phi) +
\cale^N_{\ref{e:mNs2}}(\phi,s)
\big) ds,
\]
and that the integral of $\cale^N_{\ref{e:mNs2}}(\phi,s)$ is negligible 
for all $d\ge 2$ on account of \eqref{e:mNs2}. We can thus
focus on $\bar m^N_s(\phi)$ defined in ~\eqref{e:mbarNs},
which we write in the form
\begin{equation}\label{e:barmdecomp}
\bar m^N_s(\phi) = \bar m^{N,1}_s(\phi) - \bar m^{N,2}_s(\phi),
\end{equation}
where 
\begin{equation}\label{e:barm12}
\begin{aligned}
\bar m^{N,1}_s(\phi) &=  \rho^2N^{1+d/2}|\BMr|K^2 
\int_{\Rd}\phi^2(x) dx
\int_{\BMrx}w^N_s(z) dz, \\ 
\bar m^{N,2}_s(\phi) &= \rho^2N^{1+d/2}K^2 \int_{\Rd}
\phi^2(x)dx
\int_{\BMrx} \int_{\BMrx}w^N_s(z_1)w^N_s(z_2)dz_1dz_2. 
\end{aligned}
\end{equation}

Recall that $K=K'N^{d/2-1}$ where $K'=1$ if $d\ge 3$ and
$K'=\log N$ if $d=2$, and define
\begin{equation}\label{e:gbarN}
\bar\gamma^N(s) = K'\gamma_e(Ns).
\end{equation}
\begin{lemma}\label{l:mNone} 
There is a constant 
$C_{\ref{e:mNone}}=C_{\ref{e:mNone}}(\phi)>0$ such that for $s\ge0$,
\begin{equation}\label{e:mNone}
E(\bar m^{N,1}_s(\phi)) = \rho^2|B_r|^2K' X^N_0(\phi^2) + \cale_{\ref{e:mNone}}(\phi,s)
\end{equation}
where
\[
|\cale_{\ref{e:mNone}}(\phi,s)| \le C_{\ref{e:mNone}}K'\Big(\frac{1}{\sqrt
  N} + \sqrt s\Big) X^N_0(\1) .
\]
\end{lemma}
\begin{proof}
By a change of variables, 
\begin{align*}
\int_{\Rd} \phi^2(x) 
\int_{\BMrx} w^N_s(z)\ dzdx
& = \int_{\Rd} w^N_s(z)
\int_{\BMrz} \phi^2(x)\ dx dz\\
& = \int_{\Rd} w^N_s(z)
\int_{\BMrz} [\phi^2(z) +(\phi^2(x)-\phi^2(z))]\ dxdz \\
& = |\BMr|
\int_{\Rd} \phi^2(z) w^N_s(z)\ dz
+ \vep^N_s,
\end{align*}
where we have set $\vep^N_s = \int_{\Rd} 
\int_{\BMrx} w^N_s(z) [\phi^2(x)-\phi^2(z)]dxdz$. For $z\in
\BMr(x)$ and $C_\phi=2\|\phi\|_\infty
\|\phi\|_{\text{Lip}}$, $|\phi^2(x)-\phi^2(z)|\le C_\phi
2r/\sqrt N$.  
Thus
\begin{align*}
|\vep^N_s| &\le C_\phi\frac{2r}{\sqrt N}\int_{\Rd}
\int_{\BMrx} w^N_s(z)dx dz
= C_\phi\frac{2r}{\sqrt N}|\BMr| w^N_s(\1).
\end{align*}
Returning to the definition of $\bar m^{N,1}_s(\phi)$ 
we see that
\begin{equation}\label{e:mNonea}
E(\bar m^{N,1}_s(\phi))  =
\rho^2N^{1+d/2}|\BMr|K^2 \Big(
|\BMr| E\big(w^N_s(\phi^2)\big) + E(\vep^N_s)\Big)
= \rho^2|B_r|^2 K' E(X^N_s(\phi^2)) + \cale,
\end{equation}
with 
\begin{equation}\label{e:mNoneb}
|\cale| = \rho^2N^{1+d/2}|\BMr|K^2 E(|\vep^N_s|)
 \le C_\phi\rho^2|B_r|^2
\frac{2rK'}{\sqrt N} E(X^N_s(\1))
= \frac{CK'}{\sqrt N}
X^N_0(\1)
\end{equation}
by the martingale property of $X^N_s(\1)$ (Corollary~\ref{c:L2bnd}).

Next, we bound the difference
$|E(X^N_s(\phi^2))-X^N_0(\phi^2)|$. 
By the single particle duality
equation \eqref{e:dualN1} and a change of variables,
\begin{align*}
E\big(X^N_s(\phi^2)\big) &=
\int_{\R^d}\phi^2(x) E^N_{\{x\}}(X^N_0(\eta^N_s)) dx 
=\int_{\R^d}\phi^2(x) E^N_{\{0\}} (X^N_0(x+\eta^{N}_s)) dx\\
&= \int_{\R^d}E^N_{\{0\}}(\phi^2(x'-\eta^{N}_s)) X^N_0(x') dx'
= X^N_0(\phi^2) + \int_{\R^d}E^N_{\{0\}}\Big[\phi^2(x'-\eta^{N}_s)
-\phi^2(x')\Big] X^N_0(x') dx' .
\end{align*}
Using the smoothness of $\phi$ and scaling, we see that 
Lemma~\ref{l:rwfacts}(b) (it applies to the rate $\rho|B_r|$ walk $\eta$ as well) implies
\begin{align*}
\int_{\R^d}E^N_{\{0\}}\Big[\large|\phi^2(x'-\eta^{N}_s)-\phi^2(x')\large|\Big] 
X^N_0(x') dx'
 &\le C_\phi E_{\{0\}}\Big(\frac{\large|\eta_{Ns}\large|}{\sqrt N}\Big)X^N_0(\1)
\le C_\phi C\sqrt{s} X^N_0(\1).
\end{align*}
Combining this bound with
\eqref{e:mNonea} and \eqref{e:mNoneb} gives \eqref{e:mNone}.
\end{proof}

To handle $\bar m^{N,2}_s(\phi)$ we apply
the two-particle duality equation \eqref{e:dualN2} and then split
the resulting expression into two pieces, 
obtaining  
\begin{equation}\label{m2decompn}
E(\bar m^{N,2}_s(\phi)) = J^{N,1}_s(\phi)+ J^{N,2}_s(\phi) ,
\end{equation}
with 
\begin{align}\notag
J^{N,1}_s&(\phi)\\ \label{e:JN1}
&= \rho^2N^{1+d/2}K^2 \int_{\Rd}\,\phi^2(x)
\int_{\BMrx} \int_{\BMrx} E^N_{\{z_1,z_2\}}\big[w^N_0(\xi^{N,1}_s)
1\{\tau^N\le s\}\big]dz_1dz_2dx\\\notag
J^{N,2}_s&(\phi)\\ \label{e:JN2}
& = 
\rho^2N^{1+d/2}K^2 \int_{\Rd}\phi^2(x)
\int_{\BMrx}  \int_{\BMrx} E^N_{\{z_1,z_2\}}\big[(w^N_0(\xi^{N,1}_s))
w^N_0(\xi^{N,2}_s)
1\{\tau^N> s\}\big]dz_1dz_2dx .
\end{align}

\begin{lemma}\label{l:JNone} 
There is a constant 
$C_{\ref{e:JNone}}=C_{\ref{e:JNone}}(\phi)>0$ such that for
$s\ge 0$,
\begin{equation}\label{e:JNone}
\begin{aligned}
J^{N,1}_s(\phi) &= \rho^2|B_r|^2(K'-\bar\gamma^N(s))X^N_0(\phi^2) + 
\cale_{\ref{e:JNone}}, \text{ where }\\
|\cale_{\ref{e:JNone}}| &\le 
C_{\ref{e:JNone}} K'\Big( \frac1{\sqrt{N}}+\sqrt s \Big) X^N_0(\1). 
\end{aligned}
\end{equation}
\end{lemma}

\begin{proof}
By translation invariance, changing of variables and 
order of integration, we see
\begin{align*}
\int_{\Rd}
&\int_{\BMrx}\int_{\BMrx}\phi^2(x)E^N_{\{z_1,z_2\}}\Big[w^N_0(\xi^{N,1}_s)
1\{\tau^N\le s\}\Big]\ dz_1dz_2dx \\
&=
\int_{\Rd}
\int_{\BMr}\int_{\BMr}\phi^2(x) 
E^N_{\{x+z'_1,x+z'_2\}}\Big[w^N_0(\xi^{N,1}_s)
1\{\tau^N\le s\}\Big]\ dz'_1dz'_2dx\\
& =  
\int_{\BMr}\int_{\BMr} E^N_{\{z'_1,z'_2\}}\Big[\int_{\Rd}\phi^2(x) 
w^N_0(x+\xi^{N,1}_s)
1\{\tau^N \le s\} dx\Big]dz'_1dz'_2.
\end{align*}
Changing variables again with 
$x'=x+\xi^{N,1}_s$ and adding and subtracting $\phi^2(x')$,
the right-side above equals
\begin{align*}
&
\int_{\BMr}\int_{\BMr}E^N_{\{z'_1,z'_2\}}\Big[\int_{\Rd}
\phi^2(x'-\xi^{N,1}_s)w^N_0(x')
1\{\tau^N\le s\}dx'\Big]dz'_1dz'_2\\
&= 
\int_{\BMr}\int_{\BMr}\int_{\Rd}
\phi^2(x')w^N_0(x')
P^N_{\{z'_1,z'_2\}}(\tau^N\le s\}\
dx'dz'_1dz'_2 \\
&\qquad +
\int_{\BMr}\int_{\BMr}E^N_{\{z'_1,z'_2\}}\Big[\int_{\Rd}
\Big(\phi^2(x'-\xi^{N,1}_s) -\phi^2(x')\Big)
1\{\tau^N\le s\}w^N_0(x') dx'\Big]dz'_1dz'_2\\
&= w^N_0(\phi^2) \int_{\BNr}\int_{\BNr}
P^N_{\{z'_1,z'_2\}}(\tau^N\le s\}\
dz'_1dz_2' + \vep^{N}_s\\
&= w^N_0(\phi^2) N^{-d}\int_{B_r}\int_{B_r}
P_{\{z_1,z_2\}}(\tau\le Ns\}\
dz_1dz_2 + \vep^{N}_s\\
&= |B_r|^2N^{-d}w^N_0(\phi^2)(1-\gamma_e(Ns)) + \vep^{N}_s,
\end{align*}
where 
\begin{align}
\nonumber|\vep^{N}_s| &=
\Big|
\int_{\BMr}\int_{\BMr}\int_{\Rd}
E^N_{\{z'_1,z'_2\}}\Big[\Big(\phi^2(x'-\xi^{N,1}_s) -\phi^2(x')\Big)
1\{\tau^N\le s\}\Big]w^N_0(x')\ dx'dz'_1dz'_2\Big|\\
\nonumber&\le C_\phi
\int_{\BMr}\int_{\BMr}\int_{\Rd}
E^N_{\{z'_1,z'_2\}}\Big[|\xi^{N,1}_s|1\{\tau^N\le s\}\Big]w^N_0(x')
dx' dz'_1dz'_2\\
\nonumber& = C_\phi w^N_0(\1)
\int_{\BMr}\int_{\BMr}
E^N_{\{z'_1,z'_2\}}(|\xi^{N,1}_s|1\{\tau^N\le s\})
\ dz'_1dz'_2\\
\label{epsexp}& = C_\phi w^N_0(\1) N^{-d}\frac1{\sqrt N}
\int_{B_r}\int_{B_r}
E_{\{z_1,z_2\}}(|\xi^{1}_{Ns}|1\{\tau\le Ns\})
\ dz_1dz_2.
\end{align}
For fixed $z_1,z_2\in B_r$, letting $z_3=0$,
Lemma~\ref{lem:xibound}  implies that  
\[
E_{\{z_1,z_2\}}(|\xi^{1}_{Ns}|)
\le 2\sqrt 2r + \sum_{i=1}^3E\Big[\sup_{t\le Ns}
|W^{i,z_i}_{t}|\Big]
\le 2\sqrt 2r + 3\Big(r+E(\sup_{t\le Ns}
|Y^0_t|)\Big)
\le C_r + C\sqrt{Ns}
\]
using Lemma~\ref{l:rwfacts}(b) for the last
inequality. Plugging this bound into \eqref{epsexp}, we obtain
\[
|\vep^N_s|\le C\Big(\frac 1{{\sqrt N}} +\sqrt{s}\Big)N^{-d}w^N_0(\1) .
\]
Returning to $J^{N,1}_s$, 
we now have
\begin{align*}
J^{N,1}_s(\phi) &= \rho^2N^{1+\frac d2 }K^2\Big(
|B_r|^2N^{-d}w^N_0(\phi^2)(1-\gamma_e(Ns)) +
\vep^N_s\Big)\\
&= \rho^2|B_r|^2K'X^N_0(\phi^2)(1-\gamma_e(Ns)) + \cale\\
&=\rho^2|B_r|^2X_0^N(\phi^2)[K'-\bar \gamma^N(s)]+\cale,
\end{align*}
where
\begin{align*}
|\cale| &\le C N^{1+\frac d2}K^2|\vep^N_s|
\le
 C N^{1+\frac d2}K
N^{-d}\Big(\frac 1{\sqrt N} + \sqrt s\Big)X^N_0(\1)\\
&= C K' \Big(\frac 1{\sqrt N} + \sqrt s\Big) X^N_0(\1),
\end{align*}
which proves \eqref{e:JNone}.
\end{proof}

Using Lemmas~\ref{l:mNone} and \ref{l:JNone} in \eqref{e:barmdecomp} and \eqref{m2decompn}, we arrive at the following: 
\begin{corollary}\label{c:progress}
For 
$s\ge 0$,
\begin{equation}\label{e:progress}
\Big| E\big(\bar m^N_{s}(\phi)\big) -\rho^2|B_r|^2
\bar\gamma^N(s) X^N_0(\phi^2) \Big|  \le 
J^{N,2}_s(\phi) + C_{\ref{e:progress}}K'\Big(\frac 1{\sqrt N} + 
\sqrt s\Big) X^N_0(\1).
\end{equation}
\end{corollary}

We turn now to the analysis of 
$J^{N,2}_s(\phi)\le
\|\phi\|^2_{\infty} J^{N,2}_s(\1)$. Recall 
$Y^x_t$ and $\tilde Y^x_t= Y^x(I^{-1}(t))$ from 
Section 5, $\tilde \tau$ from \eqref{taudef1}, and the process
$\txi^x_t= \tilde Y^x_t1\{\tilde\tau>t\}$, which by 
Lemma~\ref{lem:difftimechange} has the same law as
$\xi_t^1-\xi_t^2$  under $P_{\{z_1,z_2\}}$ when 
$z_1-z_2=x$. As in Lemma~\ref{l:Ibnd} we will write $Y^{\bar U}_t$ when the
initial law of $Y_t$ is the law of $\bar U$.
\begin{lemma}\label{l:JNtwo} 
For $s\ge 0$,
\begin{equation}\label{e:JNtwo}
J^{N,2}_s(\1) 
=\rho^2|B_r|^2N^{1-\frac{d}{2}} \int_{\Rd}
X^N_0(y)
E\Big[X^N_0\Big(y-\frac{\tilde Y^{\bar U}_{Ns}}{\sqrt N}\Big)
1\{\tilde\tau> Ns\}\Big]dy.
\end{equation}
\end{lemma}
\begin{proof} 
By changing variables and orders of integration, and using the difference
process $\txi^N_s=\xi^{N,1}_s-\xi^{N,2}_s$, we have
\begin{align*}
J^{N,2}_s(\1) &=
\rho^2N^{1+d/2}K^2 \int_{\Rd}
\int_{\BMr}\int_{\BMr}E^N_{\{z_1,z_2\}}\Big[
w^N_0(x+\xi^{N,1}_s) w^N_0(x+\xi^{N,2}_s)
1\{\tau^N> s\}\Big]dz_1dz_2dx\\
&=\rho^2N^{1+d/2}K^2 
\int_{\BMr}\int_{\BMr}\int_{\Rd}w^N_0(y)
E^N_{\{z_1,z_2\}}\Big[w^N_0(y-\xi^{N,1}_s+\xi^{N,2}_s)
1\{\tau^N> s\}\Big]dy dz_1dz_2\\
&=\rho^2N^{1+d/2}K^2 \int_{\Rd} w^N_0(y) \int_{\BMr}\int_{\BMr} 
E^N_{\{z_1,z_2\}}\big(w^N_0(y-\txi^{N}_s)
1\{\tau^N> s\}\big)dz_1dz_2dy\\
&=\rho^2N^{1-d/2}K^2 \int_{\Rd} w^N_0(y) \int_{B_r}\int_{B_r}
E_{\{z_1,z_2\}}\Big[w^N_0\Big(y-\frac{\txi_{Ns}}{\sqrt N}\Big)
1\{\tilde\tau> Ns\}\Big]dydz_1dz_2\\
&=\rho^2N^{1-d/2}K^2 \int_{\Rd} w^N_0(y) \int_{B_r}\int_{B_r}
E\Big[w^N_0\Big(y-\frac{\tilde Y^{z_1-z_2}_{Ns}}{\sqrt N}\Big)
1\{\tilde\tau> Ns\}\Big]dydz_1dz_2
\end{align*}
which is \eqref{e:JNtwo}.
\end{proof}

Recall $s_N$ from \eqref{sNdef} and define $\delta_N$ so that
\begin{equation}\label{e:qdelta}
s_N = (\log N)^{-q}, \quad \delta_N=(\log N)^\delta s_N 
= (\log N)^{\delta -q},
\end{equation}
where $\delta\in(0,\frac12]$  and $q>4$ are fixed constants.

\begin{lemma}\label{l:JNtwo2d} There is a
  constant $C_{\ref{e:JNtwo2d}}>0$ 
such that  
\begin{multline}\label{e:JNtwo2d}
\sup_{s\in[s_N,2s_N]}J^{N,2}_s (\1) \le C_{\ref{e:JNtwo2d}} 
\Big[\frac{\log\log N}{\log N}X^N_0(\1) \\+ 
\frac{N^{\frac{1}{2}-\frac{d}{4}}}{s_N\log N} \int_{\Rd} \int_{\Rd} 
X^N_0(y_1)X^N_0(y_2)  
1\big\{ |y_1-y_2|\le \sqrt{\delta_N}\big\} dy_1 dy_2\Big].
\end{multline}
\end{lemma}

\begin{proof} 
Let $s\in[s_N,2s_N]$, let $t_N=Ns/4$, and define
\begin{align*}
\cale_1 &=
N^{1-\frac{d}{2}}\int_{\Rd}
X^N_0(y)E\Big[X^N_0\Big(y-\frac{\tilde Y^{\bar U}_{Ns}}{\sqrt N}\Big)
1\Big\{\frac{|\tilde Y^{\bar U}_{Ns}|}{\sqrt N}>\sqrt{\delta_N}, 
 \tilde\tau> Ns\Big\}\Big]dy\\
\cale_2 &=N^{1-\frac{d}{2}}\int_{\Rd}
X^N_0(y)E\Big[X^N_0\Big(y-\frac{\tilde Y^{\bar U}_{Ns}}{\sqrt N}\Big)
1\Big\{|\tilde Y^{\bar U}_{t_N}|\le 2r, \frac{|\tilde Y^{\bar U}_{Ns}|}{\sqrt N}\le\sqrt{\delta_N}, 
  \tilde\tau> Ns\Big\}\Big]dy\\
\cale_3 &=N^{1-\frac{d}{2}}\int_{\Rd}
X^N_0(y)E\Big[X^N_0\Big(y-\frac{\tilde Y^{\bar U}_{Ns}}{\sqrt N}\Big)
1\Big\{|\tilde Y^{\bar U}_{t_N}|> 2r, \frac{|\tilde Y^{\bar U}_{Ns}|}{\sqrt N}\le\sqrt{\delta_N}, 
  \tilde\tau> Ns\Big\}\Big]dy
\end{align*}
To prove \eqref{e:JNtwo2d}, by Lemma~\ref{l:JNtwo} it suffices to show that each
$\cale_i$ is uniformly bounded in $s\in[s_N,2s_N]$ by terms in the right side of
\eqref{e:JNtwo2d}. 

Since  $I^{-1}(Ns)\le Ns\le2Ns_N$, for
$k>2/\delta$ and $N\ge
N_0({q,r})$ large enough,
\begin{align*}
P\Big(\frac{|\tilde Y^{\bar U}_{Ns}|} 
{\sqrt N}>\sqrt{\delta_N}, \tau>Ns\Big) 
&\le 
P\Big(\sup_{t \le 2Ns_N}|Y^{\bar U}_{t}|\ge \sqrt {N\delta_N}\Big)
\le P\Big(\sup_{t \le 2Ns_N}|Y^{0}_{t}|\ge \sqrt {N\delta_N}-2r\Big)\\
&\le C \frac{(2Ns_N)^{k}}{(\sqrt{N\delta_N}-2r)^{2k}}
\le C (s_N/\delta_N)^k \le C/{(\log N)^2,}
\end{align*}
where we have used \eqref{e:Ysupbnd}.
With this bound and $X^N_0(\cdot)\le K$ we obtain from the
definition of $\cale_1$ that
\[
\cale_1
\le  N^{1-\frac{d}{2}}K\frac{C}{{(\log N)^2}}\int_{\Rd}
X^N_0(y)dy
\le\frac{C}{\log N} X^N_0(\1).
\]
The above bound is then extended to all
$N\ge 3$ by increasing $C$.

Using $X^N_0(\cdot)\le K$ again, we have
\begin{align*}
\cale_2 \le N^{1-\frac{d}{2}}K\int_{\Rd}
X^N_0(y) 
P \Big(|\tilde Y^{\bar U}_{Ns/4}|\le 2r)dy
=  K' X^N_0(\1) P\Big(|\tilde Y^{\bar U}_{Ns/4}|\le 2r\Big) .
\end{align*}
It follows from Lemma~\ref{l:betabnd}, taking $\beta=1/3$,
that for $N\ge N_0(q)$, 
\[
\cale_2 \le C\frac{K'}{(Ns/4)^{1/3}}X^N_0(\1) \le
C\frac{\log N}{(Ns_N/4 )^{1/3}}X^N_0(\1) \le C\frac{(\log
  N)^{1+q/3}}{N^{1/3}} X^N_0(\1).
\]
As before, the above bound is then valid for all $N\ge 3$ by increasing $C$.

We split $\cale_3$ into two parts, letting
$G(a,b)=\big\{ |{Y^{\bar U}_{I^{-1}(u)}}|>2r \text{ for 
  all }u\in[a,b]\big\}$. If we let
$G_N=G(Ns/2,Ns)$ then we can write $\cale_3 =
\cale'_3+\cale''_3$, where
\begin{align*}
\cale'_3&=N^{1-\frac{d}{2}}\int_{\Rd}
X^N_0(y)
E\Big[X^N_0\Big(y-\frac{\tilde Y^{\bar U}_{Ns}}{\sqrt N}\Big)
1_{G_N^c}1\Big\{|\tilde Y_{t_N}^{\bar U}|> 2r, 
\frac{|\tilde Y^{\bar U}_{Ns}|}{\sqrt N}\le \sqrt{\delta_N}, 
 \tilde\tau(\bar U)> Ns\Big\}\Big]dy\\
\cale''_3 &=N^{1-\frac{d}{2}}\int_{\Rd}
X^N_0(y)
E\Big[X^N_0\Big(y-\frac{\tilde Y^{\bar U}_{Ns}}{\sqrt N}\Big)
1_{G_N}1\Big\{|\tilde Y_{t_N}^{\bar U}|> 2r, 
\frac{|\tilde Y^{\bar U}_{Ns}|}{\sqrt N}\le \sqrt{\delta_N}, 
  \tilde\tau(\bar U)> Ns\Big\}\Big]dy
\end{align*}
Applying the Markov property at
time $t_N=Ns/4$, and letting $P_{\tilde Y^{\bar U}_{t_N}(\omega)}$ denote
the law of $Y$ with $Y_0=\tilde Y^{\bar U}_{t_N}(\omega)$,
\begin{align*}
\cale'_3 &\le KN^{1-\frac{d}{2}}
           \int_{\Rtwo}X^N_0(y)P\Big(\Big\{|\tilde Y_{t_N}^{\bar
           U}|> 2r,  \tilde\tau(\bar U)> t_N \Big\}\cap G_N^c\Big)dy \\
& = K'  E\left[1\Big\{|\tilde Y_{t_N}^{\bar U}|>
2r, \tilde\tau(\bar U)>t_N\Big\} 
P_{\tilde Y^{\bar U}_{t_N}}\big( |Y(I^{-1}(u))|\le 2r \text{ for some }
u\in[Ns/4,3Ns/4]\big)
 \right]X^N_0(\1) \\
& \le K' P\Big( \tilde\tau(\bar U)>Ns_N/4\Big)
\sup_{|x|>2r}P(Y^x(I^{-1}(u))|\le 2r \text{ for some }
u\in[Ns_N/4,3Ns_N/2]\big)X^N_0(\1)\\
& \le C\frac{\log\log N}{\log N} X^N_0(\1),
\end{align*}
where we have used
Lemma~\ref{l:GE}, $s_N=(\log N)^{-q}$,
and the fact that $ K'P(\tau(\bar U)>Ns_N/4)$ is bounded in
$N$ for $d=2$ by Proposition~\ref{t:noncoal}, and at most one if $d>2$. 

Set
$t'_N=3Ns/4$ and $t''_N=t'_N-t_N=Ns/2$.
By the Markov property of $\tilde Y^{\bar U}$ at time $t_N=Ns/4$, 
\begin{align}
\label{E"3}\cale''_3 &=N^{1-\frac{d}{2}}\int_{\Rd}
X^N_0(y) E\Big[1\{|\tilde Y^{\bar U}_{t_N}|>2r,\tilde\tau(\bar U)>t_N \}
\\
\nonumber&\qquad \times  E_{\tilde Y^{\bar U}_{t_N}}
\Big(X^N_0 
\Big(y-\frac{\tilde Y_{t'_N}}{\sqrt N}\Big) 
1\Big\{\frac{|\tilde Y_{t'_N}|}{\sqrt N}\le \sqrt{\delta_N}\Big\} 
1\Big\{G(t_N,t'_N)\cap\{\tilde\tau>t'_N\Big\}\Big)\Big]dy.
\end{align}
On the event $G(t_N,t'_N) \cap \{\tilde\tau>t'_N\}$, for all $u\in[t_N,t'_N]$, 
\begin{align*}
I^{-1}(u)&=I^{-1}(t_N)+\int_{t_N}^u\Bigl[1-\frac{\psi_r(\tilde Y_s)}{\rho|B_r|}\Bigr]ds=I^{-1}(t_N)+u-t_N.
\end{align*}
Using the above and Lemma~\ref{lem:difftimechange}, we see that on the above event and for $u$ as above, under $P_{\tilde Y^{\bar U}_{t_N}}$,
\begin{equation}\label{dualdecomp}
\tilde Y(u)-\tilde Y(t_N)\\
=Y(I^{-1}(t_N)+(u-t_N))-Y(I^{-1}(t_N))\\
:=\hat Y^0(u-t_N).
\end{equation}
Set $\tilde\cF_s=\cF^Y_{I^{-1}(s)}$, where we recall that
$\cF^Y_t$ is the right-continuous filtration generated by
the random walk $Y$. Then $I^{-1}(s)$ is an
$(\cF^Y_t)$-stopping time.  By the strong Markov property of
$Y$, $\hat Y^0$ is a copy of $Y$ starting at $0$ and is
independent of $\tilde \cF_{t_N}$. Since $\tilde Y_{t_N}$
is $\tilde \cF_{t_N}$-measurable, we may conclude from
\eqref{dualdecomp} and \eqref{E"3} that
\begin{align}
\nonumber
\cale''_3 &\le N^{1-\frac{d}{2}}\int_{\Rtwo}
X^N_0(y) E\Big[1\{|\tilde Y^{\bar U}_{t_N}|>2r,\tilde\tau(\bar U)>t_N \}\\
\nonumber&\qquad\quad \times E_{\tilde Y^{\bar U}_{t_N}}\Big(X^N_0
\Big(y-\Big(\frac{\tilde Y_{t_N} +
\hat Y^0_{t''_N}}{\sqrt{t''_N}}\Big)\sqrt{\frac{t''_N}{N}}\Big)
1\Big\{\Big|\frac{\tilde Y_{t_N}+\hat Y^0_{t''_N}}
{\sqrt{t''_N}}\Big|\sqrt{\frac{t''_N}{N}} 
\le \sqrt{\delta_N}\Big\}\Big)\Big]dy\\
\nonumber&\le N^{1-\frac{d}{2}}\int_{\Rtwo}
X^N_0(y) E\Big[1\{\tilde\tau(\bar U)>t_N), |\tilde Y^{\bar U}_{t_N}|>2r\}\Big]\\
\label{e"3bound}&\qquad 
\times\sup_{{x}}E\left[X^N_0
\left(y-\frac{\hat Y^x_{t''_N}}{\sqrt{t''_N}}
\sqrt{\frac{t''_N}N}\right)
1\Big \{\Big|\frac{\hat Y^x_{t''_N}}
{\sqrt{t''_N}}\Big|\sqrt{\frac{t''_N}{N}} 
\le \sqrt{\delta_N}\Big\}\right]dy,
\end{align}
where we have set $\hat Y^x=x+\hat Y^0$. 
By \eqref{e:Ydensbnd2}, applied to 
$f(z)=X^N_0\Big(y-z\sqrt{\frac{t_N''}{N}}\Big)
1\{|z| \frac{t_N''}{N}\le \sqrt{\delta_N}\}$, we have, uniformly in $x$,
\begin{align*}
E\Big[X^N_0
&\Big(y-
\frac{\hat Y^x_{t''_N}}{\sqrt{t''_N}}\Big)\sqrt{\frac{t''_N}N}\;\Big)
1\Big \{\Big|
\frac{\hat Y^x_{t''_N}}{\sqrt{t''_N}}\Big|\sqrt{\frac{t''_N}{N}} 
\le \sqrt{\delta_N}\Big\}\Big]\\
&\le e^{-2\rho|B_r| t''_N}K + C\int_{\Rd } 
X^N_0(y-z\sqrt{\frac{t''_N}{N}}) 
1\Big\{|z|\sqrt{\frac{t''_N}{N}}\le \sqrt{\delta_N}\Big\}
dz\\
&\le e^{-\rho|B_r| Ns_N}K + C\Big(\frac{N}{t''_N}\Big)^{d/2}\int_{\Rd} 
X^N_0(y-x')1\{|x'|\le \sqrt{\delta_N}\}dx'\\
&\le \frac{C}{\log N} + \frac{C}{s_N^{d/2}}\int_{\Rd}
X^N_0(y-x')1\{|x'|\le \sqrt{\delta_N}\}dx'.
\end{align*}
Use this and the fact that $P(\tilde\tau(\bar U)>t_N)\le \frac{C}{K'}$ (from Proposition~\ref{t:noncoal} if $d=2$)
 in \eqref{e"3bound}, to see that
\begin{align*}
\cale''_3 &\le C (s_NN)^{1-\frac{d}{2}}P(\tilde\tau(\bar U)>t_N)\Bigl[\frac{X_0^N(\1)}{\log N}+ \frac{1}{s_N}\int_{\Rd}\int_{\Rd}
X^N_0(y) 
X^N_0(y-x')1\{|x'|\le \sqrt{\delta_N}\}dx'dy\Bigr]\\
&\le C\Big[\frac{X^N_0(\1)}{(\log N)} +
\frac{(s_NN)^{1-\frac{d}{2}}}{s_NK'}
\int_{\Rd} \int_{\Rd} 
X^N_0(y_1)X^N_0(y_2)1\{|y_1-y_2|\le \sqrt{\delta_N}dy_1dy_2\Bigr]\\
&\le C\Big[\frac{X^N_0(\1)}{(\log N)} + \frac{N^{\frac{1}{2}-\frac{d}{4}}}{s_N\log N}
\int_{\Rd} \int_{\Rd} 
X^N_0(y_1)X^N_0(y_2)1\{|y_1-y_2|\le \sqrt{\delta_N}dy_1dy_2\Bigr],
\end{align*}
where the last line is very crude if $d\ge 3$, and is an equality if $d=2$. 
Combining the bounds for
$\cale_1,\cale_2,\cale_3',\cale_3''$, we establish \eqref{e:JNtwo2d}.
\end{proof}

Corollary~\ref{c:progress}, Lemma~\ref{l:JNtwo2d}, and $q>4$ imply the following:
\begin{lemma}\label{l:key2db}
There is a constant $C_{\ref{e:key2db}}=   
C_{\ref{e:key2db}}(\phi)$ such that    
\begin{multline}\label{e:key2db}
\sup_{s\in[s_N,2s_N]}\Big| E\big(\bar m^N_s(\phi)\big)
-\rho^2|B_r|^2 \bar\gamma^N(s) X^N_0(\phi^2) \Big|  
\le C_{\ref{e:key2db}}\Big[ \frac{\log\log N}{\log N} X^N_0(\1) 
\\
+\frac{N^{\frac{1}{2}-\frac{d}{4}}}{s_N\log N}
\int_{\Rd}\int_{\Rd} X^N_0(y_1)X^N_0(y_2) 1\big\{|y_1-y_2|\le 
\sqrt{\delta_N}\big\} dy_1dy_2 \Big].
 \end{multline}
\end{lemma}

\begin{remark} To identify the square function of $M^N(\phi)$ we will need to use the above and 
the Markov property to bound
\[
\Big|E\big(\bar m^N_s(\phi)\mid\calf^N_{s-s_N}\big)
-\rho^2|B_r|^2 \bar\gamma^N({s_N}) X^N_{s-s_N}(\phi^2) \Big|. 
\]
This
means we will need to bound the expected value of the last
term in \eqref{e:key2db} with $X^N_{s-s_N}$ replacing
$X^N_0$. For $d\ge 3$ we only need to bound the resulting
double integral  on the right-hand side of \eqref{e:key2db}
by $X^N_{s-s_N}(\1)^2$, but for $d=2$ we require the
following additional result. 
\end{remark}

\begin{lemma}\label{l:key2da}
Assume $d=2$. For $T>0$ there exists $C_{\ref{e:key2da}}(T)>0$ such that
for $\delta_N\le s\le T$, 
\begin{equation}\label{e:key2da}
\begin{aligned}
E\Big[ \int_{\Rtwo}\int_{\Rtwo} 
&1\big\{|y_1-y_2| \le \sqrt{\delta_N}\big\} X^N_s(y_1)
X^N_s(y_2)dy_1dy_2\Big] \\
&\le C_{\ref{e:key2da}}(T)\Big(
\frac{\delta_N}{s}X^N_0(\1)^2 + \delta_N \log(1/\delta_N) X^N_0(\1)
\Big).  
\end{aligned}
\end{equation}
\end{lemma}
\begin{proof}
By the duality equation
\eqref{e:dualN2} and then a change of
variables,
\begin{align*}
E\Big[ \int_{\Rtwo}\int_{\Rtwo}& 
1\big\{|y_1-y_2| \le \sqrt{\delta_N}\big\} X^N_s(y_1)
X^N_s(y_2)dy_1dy_2\Big] \\
&= (\log N)\int_{\Rtwo}\int_{\Rtwo} 1\big\{|y_1-y_2| \le \sqrt{\delta_N}  
\big\}E^N_{\{y_1,y_2\}}\Big[X^N_0(\xi^{N,1}_s)
1\big\{\tau^N\le  s\big\}\Big]dy_1dy_2\\
&\quad +\int_{\Rtwo}\int_{\Rtwo} 
1\big\{|y_1-y_2| \le \sqrt{\delta_N} \big\}
E^N_{\{y_1,y_2\}}\Big[X^N_0(\xi^{N,1}_s)X^N_0(\xi^{N,2}_s)
1\big\{\tau^N> s\big\}\Big]dy_1dy_2\\
&= \cale_1+\cale_2.
\end{align*}
Here
\begin{align*}
\cale_1&= (\log N)\int_{\Rtwo}\int_{\Rtwo} 1\big\{|y_1-y_2| \le \sqrt{\delta_N}  
\big\}E^N_{\{y_1,y_2\}}\Big[X^N_0(\xi^{N,1}_s)
1\big\{\tau^N\le  s\big\}\Big]dy_1dy_2\\
&=(\log N)  \int_{\Rtwo}\int_{\Rtwo} 1\big\{|y_1-y_2| \le \sqrt{\delta_N} \big\} 
E^N_{\{y_1-y_2,0\}}\Big[X^N_0(y_2+\xi^{N,1}_s)
1\big\{\tau^N\le  s\big\}\Big]dy_2dy_1\\
&=(\log N)  \int_{\Rtwo}\int_{\Rtwo} 1\big\{|y| \le \sqrt{\delta_N} \big\} 
E^N_{\{y,0\}}\big[X^N_0(y_2+\xi^{N,1}_s)
1\big\{\tau^N\le  s\big\}\Big]dy_2dy,
\end{align*}
and
\begin{align*}
\cale_2 &=   \int_{\Rtwo}\int_{\Rtwo} 
1\big\{|y| \le \sqrt{\delta_N} \big\}
E^N_{\{y,0\}}\Big[X^N_0(y_2+\xi^{N,1}_s)X^N_0(y_2+\xi^{N,2}_s)
1\big\{\tau^N> s)\big\}\Big]dy_2dy.
\end{align*}

First, integrating $y_2$  out in $\cale_1$ yields
\begin{align*}
\cale_1 &\le (\log N) X^N_0(\1) \int_{\Rtwo} 1\big\{|y|\le \sqrt{\delta_N}\Big\}
P^N_{\{y,0\}}(\tau^N\le s)dy\\ 
&= (\log N) X^N_0(\1) \int_{\Rtwo} 1\big\{|y|\le \sqrt{\delta_N}\Big\}
P_{\{y\sqrt N,0\}}(\tau\le Ns)dy .
\end{align*}
By Lemma~\ref{lem:difftimechange}, switching to
$\txi^{y\sqrt N}$, and using $I^{-1}(u)\le u$,
\begin{align*}
P_{\{y\sqrt N,0\}}(\tau\le Ns) &=
P(\tilde \xi^{y\sqrt N}_{Ns}=0) \le
P(\inf_{u\le Ns}|\tilde Y^{y\sqrt N}_{u}|\le 3r) \\
&\le P(\inf_{u\le Ns}|Y^{y\sqrt N}_{u}|\le 3r) 
= P_{y\sqrt N}(t_{3r}\le Ns) .
\end{align*}
By Lemma~\ref{l:lR}, which is applicable
because {$T\ge s\ge
\delta_N$} and we consider only $|y|\le
\sqrt{\delta_N}$, the last probability above is bounded by 
$C\frac{\log 1/|y|}{\log N}$ ($C=C(T)$). 
It follows that
\begin{align*}
\cale_1&\le C X^N_0(\1) \int_{\Rtwo} 1\big\{|y|\le 
\sqrt{\delta_N}\big\}\log (1/|y|)dy\\
&=  C X^N_0(\1) \int_0^{\sqrt{\delta_N}} u \log(1/u)du\\
&\le  C X^N_0(\1) \delta_N\,\log(1/\delta_N).
\end{align*}

By definition, 
\begin{align*}
\cale_2& =   \int_{\Rtwo}\int_{\Rtwo} 
1\big\{|y| \le \sqrt{\delta_N} \big\}
E_{\{y\sqrt N,0\}}\Big[X^N_0\Big(y_2+\frac{\xi^{1}_{{sN}}}{\sqrt
  N}\Big)X^N_0\Big(y_2+ \frac{\xi^{2}_{{sN}}}{\sqrt N}\Big)
1\big\{\tau> Ns)\big\}\Big]dy_2dy\\
&=\cale'_2+\cale''_2,
\end{align*}
where, with  $G_N =  \{ |\xi^{1}_u-\xi^{2}_u |\ge 3r \text{
  for all }u\le Ns\},$
\begin{align*}
\cale'_2& =   \int_{\Rtwo}\int_{\Rtwo} 
1\big\{|y| \le \sqrt{\delta_N} \big\}
E_{\{y\sqrt N,0\}}\Big[X^N_0\Big(y_2+\frac{\xi^{1}_{sN}}{\sqrt
  N}\Big)X^N_0\Big(y_2+ \frac{\xi^{2}_{sN}}{\sqrt N}\Big)
1_{G_N}1\big\{\tau> Ns)\big\}\Big]dy_2dy,\\
\cale''_2& =   \int_{\Rtwo}\int_{\Rtwo} 
1\big\{|y| \le \sqrt{\delta_N} \big\}
E_{\{y\sqrt N,0\}}\Big[X^N_0\Big(y_2+\frac{\xi^{1}_{sN}}{\sqrt
  N}\Big)X^N_0\Big(y_2+ \frac{\xi^{2}_{sN}}{\sqrt N}\Big)
1_{G_N^c}1\big\{\tau> Ns)\big\}\Big]dy_2dy.
\end{align*}

We use the representation of $\xi=(\xi^1,\xi^2)$ 
 in Lemma~\ref{lem:twopartdualconst} and the fact that on $G_N\cap\{\tau>Ns\}$, $\xi_{sN}=W_{sN}$. Then, dropping the
indicator of $G_N\cap\{\tau>Ns\}$, we see
that 
\begin{align*}
\cale'_2 
&\le \int_{\Rtwo}\int_{\Rtwo}1\big\{|y| \le \sqrt{\delta_N} \big\}
E\Big[X^N_0\Big(y_2+\frac{W^{y\sqrt N}_{sN}}{\sqrt
  N}\Big)X^N_0\Big(y_2+ \frac{W^0_{sN}}{\sqrt N}\Big)\Big]dy_2dy\\
&= \int_{\Rtwo}\int_{\Rtwo}1\big\{|y| \le \sqrt{\delta_N} \big\}
E\Big[X^N_0\Big(y_2+\frac{W^{y\sqrt N}_{sN}}{\sqrt
  N}\Big)\Big] E\Big[X^N_0\Big(y_2+ \frac{W^0_{sN}}{\sqrt N}\Big)\Big]dy_2dy,
\end{align*}
the last equality by independence of the walks $W^{y\sqrt
  N},W^0$.  By the density bound
\eqref{e:Ydensbnd2} applied to 
$f(z) = X^N_0(y_2+z \sqrt{s})$, and $s\ge s_N$,
for all $y_2$,
\begin{align*}
E\Big[ X^N_0\Big(y_2+\frac{W^{0}_{Ns}}{\sqrt N}\Big) \Big]
&= E\Big[ f\Big(\frac{W^0_{Ns}}{\sqrt{ Ns}}\Big) \Big]
\le e^{-2\rho|B_r|Ns}\log N + C\int_{\Rtwo} X^N_0(y_2-z
\sqrt s) dz \\
&\le
\frac{C}{\log N}+ \frac{C}{s}
\int_{\Rtwo} X^N_0(y_2-z')dz'
= \frac{C}{\log N}+ \frac{C}{s} X^N_0(\1).
\end{align*}
With this bound, integrating out $y_2$ first and then $y$, it follows that
\begin{align*}
\cale_2'&\le C\Big(\frac{1}{\log N}+ \frac{1}{s} X^N_0(\1) \Big)
\int_{\Rtwo}\int_{\Rtwo} 1\big\{|y| \le \sqrt{\delta_N}
\big\}E\Big[X^N_0(y_2+\frac{W^{y\sqrt N}_{Ns}}{\sqrt N})\Big]
dy_2dy\\
&=
C\Big(\frac{1}{\log N}+ \frac{1}{s} X^N_0(\1) \Big)X^N_0(\1)|B_{\sqrt{\delta_N}}|
\le 
C\Big[\frac{1}{\log N}+ \frac{1}{s} X^N_0(\1) \Big]X^N_0(\1)
\delta_N.
\end{align*}

Using $X^N_0(\cdot)\le K'=\log N$ and integrating
out $y_2$ gives
\begin{align*}
\cale''_2 &\le (\log N) X^N_0(\1)\int_{\Rtwo}
1\big\{|y| \le \sqrt{\delta_N} \big\}
P_{\{y\sqrt N,0\}}\Big(G_N^c\cap\big\{\tau>
Ns)\big\}\Big)dy.
\end{align*}
By the representation of $\xi=(\xi^1,\xi^2)$ 
given in Lemma~\ref{lem:twopartdualconst}, for any $y$,
\begin{align*}P_{\{y\sqrt N,0\}}\Big(G_N^c\cap\big\{\tau>
Ns)\big\}\Big)&\le P(|W^{y\sqrt N}_u-W^0_u)|\le 3r
\text{ for some }u\le Ns).
\end{align*}
Using Lemma~\ref{l:lR} again, for $\delta_N\le s\le T$, we have 
\begin{align*}
\cale''_2&\le (\log N) X^N_0(\1)\int_{\Rtwo}
1\big\{|y| \le \sqrt{\delta_N} \big\}
P\Big(|Y^{y\sqrt N}_u|\le 3r\text{ for some }u\le
Ns\Big)dy\\
&\le C{(T)} X^N_0(\1)\int_{\Rtwo}
1\big\{|y| \le \sqrt{\delta_N} \big\}\log (1/|y|)
dy\\
&\le C(T) X^N_0(\1)\delta_N\log(1/\delta_N). 
\end{align*}
Combining the bounds on $\cale_1,\cale'_2,\cale''_3$ gives
\eqref{e:key2da}. 
\end{proof}

We are almost ready for the proof of
Proposition~\ref{p:2did}. The proof is lengthy, so we 
separate out one of its key steps in the following lemma.
For $s\ge s_N$ define
\[
[s]_N = (j-1)s_N \text{ if } s\in[js_N, (j+1)s_N),
\quad j\ge 1.
\]
For $T>0$  
define $j_T=j_T(N)\in \nn$ by
\[
(j_{T}-1)s_N\le T<j_Ts_N.
\]

\begin{lemma}\label{l:7prime}
Assume $\phi\in
  C^3_0(\Rd)$ and $T>0$. There exists
  $\vep_{\ref{l:7prime}}^N =
  \vep^N_{\ref{l:7prime}}(\phi,T)\to 0$ as $N\to\infty$ such
  that   
\begin{equation}\label{e:7prime}
E\Big[\sup_{s_N\le t\le T}\Big|\int_{s   _N}^t
\big( \bar m^N_s(\phi) - E(\bar m^N_s(\phi)\mid \calf^N_{[s]_N})\big) 
ds \Big|^2\Big] \le \vep_{\ref{l:7prime}}^N
\big(X^N_0(\1)^2+1\big)
\end{equation}
\end{lemma}
\begin{proof}
Define
\[
\Delta^N_j= \int_{(j-1)s_N}^{js_N}\big( \bar m^N_s(\phi) -
E(\bar m^N_s(\phi) \mid \calf^N_{[s]_N})\big) ds, \quad j\ge 2.
\]
Then 
\[
\int_{s_N}^{js_N} (\bar m^N_s(\phi) - E(\bar m^N_s(\phi)\mid
\calf^N_{[s]_N})ds
= \sum_{i=2}^{j} \Delta^N_i = 
Q^{N,o}_j + Q^{N,e}_j 
\]
where
\[
 Q^{N,o}_j = \sum_{i=2}^{j} \Delta^N_i1\{i\text{ is odd}\}
 \text{ and }
Q^{N,e}_j = \sum_{i=2}^{j} \Delta^N_i1\{i\text{ is even}\}.
\]
Noting that $\calf^N_{[s]_N} = \calf^N_{(2i-1)s_N}$ for all $s
\in[2is_N,(2i+1)s_N)$, 
we see that for $i\ge 1$, 
\begin{equation}\label{e:mginc}
\begin{aligned}
E(\Delta^N_{{2i+1}}\mid \calf^N_{(2i-1)s_N})
&= \int_{2is_N}^{(2i+1)s_N}E\Big[\big(\bar m^N_s(\phi) -
E(\bar m^N_s(\phi)\mid \calf^N_{[s]_N})\big) \mid
\calf^N_{{(2i-1)s_N}}\Big]
ds  \\
&= \int_{2is_N}^{(2i+1)s_N}
\Big(E(\bar m^N_s(\phi)\mid \calf^N_{(2i-1)s_N}) -
E(\bar m^N_s(\phi)\mid \calf^N_{(2i-1)s_N})\Big) ds\\
&=0,
\end{aligned}
\end{equation}
which shows that 
$(Q^{N,o}_{2j-1},j\ge 2)$ is an $(\calf^N_{(2j-1)s_N})$
martingale. Similarly, $(Q^{N,e}_{2j},j\ge 1)$ is an
$(\calf^N_{2js_N})$ martingale. Using Doob's
 $L^2$ inequality,  we have
\begin{equation}\label{e:red00}
\begin{aligned}
E\Big[ \sup_{2\le j\le j_T}& \Big| \int_{s_N}^{js_N} 
\big(\bar m^N_s(\phi) - E(\bar m^N_s(\phi)\mid \calf^N_{[s]_N})\big)ds
\Big|^2\Big]\\
&= E\Big[\sup_{2\le j\le j_T} \Big|  Q^{N,o}_j1_{\{j\text{
    odd}\}}
+Q^{N,e}_j1_{\{j\text{ even}\}} 
\Big|^2\Big]\\
&\le 2 E\Big[ \sup_{2\le j\le j_T, j\text{ odd}}(Q^{N,o}_j)^2
+ \sup_{2\le j\le J_T, j\text{ even}}(Q^{N,e}_j)^2\Big]
\le C E\Big[ (Q^{N,o}_{j_T})^2+(Q^{N,e}_{j_T})^2\Big].
\end{aligned}
\end{equation}

Now if $j_T=2k-1$, then by \eqref{e:mginc},
\begin{align}\notag
E\Big[ (Q^{N,o}_{j_T})^2\Big]  &=
\sum_{i=2}^{k}E\Big[\Big(\Delta^{N}_{2i-1}\Big)^2\Big]
\le \sum_{i=2}^{j_T}E\Big[\Big(\Delta^{N}_i\Big)^2\Big]\\
&\le 2\sum_{i=2}^{j_T}E\Big[ \Big(\int_{(i-1)s_N}^{is_N}
    \bar m^N_s(\phi)ds\Big)^2\Big] +
2\sum_{i=2}^{j_T}E\Big[ \Big(\int_{(i-1)s_N}^{is_N}
    E(\bar m^N_s(\phi)\mid \calf^N_{(i-1)s_N})ds\Big)^2\Big],
\label{e:red0} 
\end{align}
which also holds if $j_T$ is even. 
Define $\bar X^N_T(\1) = \sup_{t\le T}X^N_t(\1)$.
By Lemma~\ref{l:mNs}(a), \eqref{e:L2bnda} and \eqref{e:L2bnd} we have
\begin{equation}\label{e:redted}
\begin{aligned}
\sum_{i=2}^{j_T}E\Big[ \Big(\int_{(i-1)s_N}^{is_N}
    \bar m^N_s(\phi)ds\Big)^2\Big] &\le 
\sum_{i=2}^{j_T}E\big[ \Big(\int_{(i-1)s_N}^{is_N}
    C_\phi(\log N) X^N_s(\1)ds\Big)^2\Big]\\
&\le 
C_\phi (\log N)^2\sum_{i=2}^{j_T}s_N^2 E\big[ \bar X^N_{T+1}(\1)^2\big]\\
&\le C(\phi,T)(X^N_0(\1)^2+1) (j_Ts_N)s_N(\log N)^2\\
&\le C(\phi,T)(X^N_0(\1)^2+1)(\log N)^{2-q} .
\end{aligned}
\end{equation}
This and \eqref{e:red0} imply $E\Big[\big(Q^{N,o}_{j_T}\big)^2\Big] \le
C(\phi,T)(X^N_0(1)^2+1) (\log N)^{2-q}$, and it is clear this
bound also holds for $E\Big[\big( Q^{N,e}_{j_T}\big)^2\Big]$. 
Plugging
into \eqref{e:red00} we obtain
\begin{equation}\label{e:red3}
E\Big[ \sup_{2\le j\le j_T} \Big| \int_{s_N}^{js_N} 
\big(\bar m^N_s(\phi) - E(\bar m^N_s(\phi)\mid \calf^N_{[s]_N})\big)ds
\Big|^2\Big]
\le C(\phi,T)(X^N_0(1)^2+1) (\log N)^{2-q} .
\end{equation}
 
Now if $(i-1)s_N \le t\le is_N$, then
\begin{align*}
\Big(&\int_{s   _N}^t 
\big( \bar m^N_s(\phi) - E(\bar m^N_s(\phi)\mid \calf^N_{[s]_N})\big) 
ds \Big)^2 \\
&\le \Big(\Big|\int_{s_N}^{(i-1)s_N}
\big( \bar m^N_s(\phi) - E(\bar m^N_s(\phi)\mid \calf^N_{[s]_N})\big) 
ds \Big|+ \int_{(i-1)s   _N}^{is_N}
\big( \bar m^N_s(\phi) +E(\bar m^N_s(\phi)\mid \calf^N_{[s]_N})\big) 
ds  \Big)^2\\
&\le 2\Big(\int_{s_N}^{(i-1)s_N}
\big( \bar m^N_s(\phi) - E(\bar m^N_s(\phi)\mid \calf^N_{[s]_N})\big) 
ds \Big)^2 + 2\Big(2C_\phi(\log N) s_N \bar X^N_{T+1}(\1)
\Big)^2\\
&\le 2\sup_{2\le j\le j_T}\Big(\int_{s_N}^{j s_N}
\big( \bar m^N_s(\phi) - E(\bar m^N_s(\phi)\mid \calf^N_{[s]_N})\big) 
ds \Big)^2
+C_\phi(\log N)^{2-2q}\bar X^N_{T+1}(\1)^2,
\end{align*}
where we have used Lemma~\ref{l:mNs}(a), \eqref{e:L2bnda}, and the martingale property of $X^N_s(\1)$ (Corollary~\ref{c:L2bnd}). 
Thus from the above, \eqref{e:red3}, and Corollary~\ref{c:L2bnd}, the left side of
\eqref{e:7prime} is 
bounded above by
\begin{multline}
2E\Big[ \sup_{2\le j\le j_T} \Big( \int_{s_N}^{js_N} 
\big(\bar m^N_s(\phi) - E(\bar m^N_s(\phi)\mid \calf^N_{[s]_N})\big)ds
\Big)^2\Big] + C_\phi(\log N)^{{2-2q}}E\Big[ \bar X^N_{T+1}(\1)^2\Big]\\
\\ \le C(\phi,T)(X^N_0(1)^2+1) (\log N)^{{2-q}}.
\end{multline}
\end{proof}

\bigskip
Define $\tgamma^N(s) = \bar \gamma^N(s-[s]_N)
= K' \gamma_e(N(s-[s]_N))$. Note that  for $s\ge s_N$, $s-[s]_N\ge s_N$, and so by
Proposition~\ref{t:noncoal},
\begin{equation}\label{gammauc}
\tgamma^N(s)\to \gamma_e\text{ uniformly in $s\ge s_N$ as }N\to\infty,
\end{equation}
and there is a constant $C_\gamma>0$ such that
\begin{equation}\label{gammaub}
\sup_{N\ge 3,s\ge s_N}\tgamma^N(s)\le C_\gamma.
\end{equation} 

\bigskip\noindent
{\it Proof of Proposition~\ref{p:2did}.}
It follows from Proposition~\ref{p:sm-decomp}, 
Lemma~\ref{l:mNs}, and Corollary~\ref{c:L2bnd} that
\[
E\Big[\sup_{0\le t\le T}\Big|\langle M^N(\phi) \rangle_t  
-\int_0^t \bar m^N_s(\phi)ds\Big|\Big]
\le C\frac{\log N}{\sqrt N}\int_0^T E(X^N_s(\1))ds
= CT\frac{\log N}{\sqrt N}X^N_0(\1),
\]
and so to prove \eqref{e:2did} it suffices to show
\begin{equation}
\label{e:2did2}
\sup_{X^N_0(\1)\le A}E\left[ \sup_{0\le t\le T}\Big| 
\int_0^t\big(\bar m^N_s(\phi)- \rho^2|B_r|^2\gamma_e
X^N_s(\phi^2)\big)ds\Big|\right] \to 0 \text{ 
as }N\to\infty.
\end{equation}
Furthermore, by Lemma~\ref{l:mNs}(a), Proposition~\ref{p:ms-bnd}(a) and Corollary~\ref{c:L2bnd},
\begin{equation}\label{barmNphibound}
E\big(\bar m^N_s(\phi)+\rho^2|B_r|^2\gamma_e X^N_s(\phi^2)\Big)
\le E\big(CK'X^N_s(\1)\big) = CK'X^N_0(\1),
\end{equation}
and therefore
\begin{align*}
E\left[ \int_0^{3\delta_N}( \bar m^N_s(\phi) +\rho^2|B_r|^2\gamma_e
X^N_s(\phi^2))ds \right] \le 
CK'\delta_N X^N_0(\1) 
\le C(\log N)^{1+\delta-q}X^N_0(\1) .
\end{align*}
So by our choice of $\delta$, $q$ (recall \eqref{e:qdelta}), we need only integrate over $[3\delta_N,T]$ in
\eqref{e:2did2}. 

Let us write 
\begin{equation}\label{e:esum}
E\Big[\sup_{3\delta_N\le t\le T}\Big| \int_{3\delta_N}^t (\bar m^N_t(\phi) - \rho^2|B_r|^2
\gamma_e X^N_s(\phi^2)) ds\Big|\Big] \le \cale_1 +\cale_2 +
\cale_3 +\cale_4,
\end{equation}
where
\begin{align*}
\cale_1 &= E\Big[\sup_{3\delta_N\le t\le T}\Big|
  \int_{3\delta_N}^t (\bar m^N_s(\phi) - 
E(\bar m^N_s(\phi)\mid \calf_{[s]_N]}))ds\Big|\Big],\\
\cale_2& =E\Big[\int_{3\delta_N}^T \Big| E(\bar m^N_s(\phi)\mid
         \calf_{[s]_N}) -
         \rho^2|B_r|^2\tgamma^N(s)X^N_{[s]_N}(\phi^2)\Big|ds \Big],\\
\cale_3& =\rho^2|B_r|^2 E\Big[\int_{3\delta_N}^T \tgamma^N(s)
\big|X^N_{[s]_ N}(\phi^2)-X^N_{s}(\phi^2)\big|ds \Big],\\
\cale_4& =\rho^2|B_r|^2 E\Big[\int_{3\delta_N}^T |\tgamma^N(s)-\gamma_e|
X^N_s(\phi^2)ds \Big].
\end{align*}

By Lemma~\ref{l:7prime},
\begin{equation}\label{e:e1}
\cale_1^2 \le \vep^N_{\ref{l:7prime}}
(X^N_0(\1)^2+1).
\end{equation}
The error term 
$\cale_4$ is simple:
\begin{equation}\label{e:e4}
\cale_4 \le C \sup_{s\ge 3\delta_N}|\tgamma^N(s)-\gamma_e|
\rho^2|B_r|^2\int_{3\delta_N}^T 
E(X^N_s(1))ds \le
CT \sup_{s\ge 3\delta_N}|\tgamma^N(s)-\gamma_e| X^N_0(\1)\to 0\text{ as }N\to\infty,
\end{equation}
by \eqref{gammauc} and $s\ge 3\delta_N\ge s_N$.

We now consider $\cale_3$.
Let $j_T$ be as defined before Lemma~\ref{l:7prime}, 
and let $j_0\in \nn$ be defined by 
\[
(j_0-1) s_N \le 3\delta_N<j_0s_N.
\]
By the martingale decomposition in 
Proposition~\ref{p:sm-decomp} and then Cauchy-Schwarz,
\begin{align}
\nonumber E( |X^N_s(\phi^2)-X^N_{[s]_N}(\phi^2)| ) &\le
E(|D^N_{s}(\phi^2)) - D^N_{[s]_N}(\phi^2)|)
+ E(|M^N_{s}(\phi^2) - M^N_{[s]_N}(\phi^2)|)\\
\label{e:Xlag}&\le
\int_{[s]_N}^sE(|d^N_{u}(\phi^2)|)du + 
\Big[ E\big(\langle M^N(\phi^2)\rangle_s
 - \langle M^N(\phi^2)\rangle_{[s]_N}\big)\Big]^{1/2}.
\end{align}
By Lemma~\ref{l:dNs} and Corollary~\ref{c:L2bnd}, 
\[
\int_{[s]_N}^sE(|d^N_{u}(\phi^2)|)du \le C_\phi
\int_{[s]_N}^sX^N_0(\1)du \le C_\phi 2s_N X^N_0(\1).
\]
By Proposition~\ref{p:sm-decomp}, Lemma~\ref{l:mNs} and \eqref{barmNphibound}, 
\begin{align*}
E\big(\langle M^N(\phi^2)\rangle_s
 - \langle M^N(\phi^2)\rangle_{[s]_N}\big) &=
E\Big[ E_{X^N_{[s]_N}}\Big(\int_0^{s-[s]_N} m^N_u(\phi^2)du\Big)\Big] \\
&\le C_\phi(\log N) \int_0^{2s_N} E\big(X^N_{u}(\1)\big)du\\
&=C_\phi (\log N)s_NX^N_0(\1). 
\end{align*}
Plugging these bounds into \eqref{e:Xlag}, and using
\eqref{gammaub} and \eqref{e:qdelta}, we obtain
\begin{equation}\label{e:e3}
\cale_3\le CT \Big((\log N)^{-q}X^N_0(\1)+
( \log N)^{\frac{1-q}2}X^N_0(\1)^{\frac12}\Big)
\end{equation}

Turning to $\cale_2$,  first use the Markov property and Lemma~\ref{l:key2db},
and then Corollary~\ref{c:L2bnd}, to conclude that 
\begin{align}
\nonumber\cale_2 &\le E\Big[\int_{3\delta_N}^TC\Big[ 
\frac{\log \log N}{\log N} X^N_{[s]_N}(\1) 
\\
\nonumber&\qquad + \frac{N^{\frac{1}{2}-\frac{d}{4}}}{s_N\log N} 
\int_{\Rd} \int_{\Rd} 
1\{|y_1-y_2|\le \sqrt{\delta_N}\} X^N_{[s]_N}(y_1)
 X^N_{[s]_N}(y_2)dy_1dy_2 \Big]ds\Big]\\
\nonumber&\le CT\frac{\log\log N}{\log N}X^N_0(\1)
\\
\label{E2summationbnd}&\qquad+ \frac{CN^{\frac{1}{2}-\frac{d}{4}}}{s_N\log N} \int_{3\delta_N}^T
\int_{\Rd} \int_{\Rd} 
1\{|y_1-y_2|\le \sqrt{\delta_N}\} E\big(X^N_{[s]_N}(y_1)
 X^N_{[s]_N}(y_2) \big)dy_1dy_2ds.
\end{align}
If $d\ge 3$, use Corollary~\ref{c:L2bnd} to bound the last term by
\begin{equation}\label{e:d3quadterm}
\frac{C_TN^{\frac{1}{2}-\frac{d}{4}}}{s_N\log N}(X_0(\1)+X_0^N(\1)^2).
\end{equation}
If $d=2$, then by Lemma~\ref{l:key2da},
\begin{align*}
\int_{3\delta_N}^T
\int_{\Rd} \int_{\Rd} 
&1\{|y_1-y_2|\le \sqrt{\delta_N}\} E\big(X^N_{[s]_N}(y_1)
 X^N_{[s]_N}(y_2)\big)dy_1dy_2ds\\
&\le C_T\int_{3\delta_N}^T 
\Big[\frac{\delta_N}{[s]_N} X^N_0(\1)^2 + \delta_N\log(1/\delta_N) 
X^N_0(\1)\Big]ds\\
&\le C_T\delta_N X^N_0(\1)^2 \int_{\delta_N}^T\frac{1}{u}du
+ C_TT\delta_N\log(1/\delta_N) X^N_0(\1) \\
&= C_T \delta_N\log(T/\delta_N)(X^N_0(\1)^2+X^N_0(\1)). 
\end{align*}
Now insert the above or \eqref{e:d3quadterm} into \eqref{E2summationbnd} to arrive at
\begin{align}
\label{e:e2}
\nonumber\cale_2 &\le C_T\frac{\log\log N}{\log N} X^N_0(\1)  
+C_T \frac{\delta_N\log(1/\delta_N)+N^{\frac{1}{2}-\frac{d}{4}}}{s_N\log N}
(X^N_0(\1)^2 +1)\\
&  \le C_T\frac{\log\log N}{\log N} X^N_0(\1)  
+C_T \frac{(\log N)^\delta\log\log N}{\log N}
(X^N_0(\1)^2 +1),
\end{align}
where the last inequality is rather crude if $d\ge 3$.
The required result now follows from \eqref{e:esum}, \eqref{e:e1}, \eqref{e:e4}, \eqref{e:e3} and \eqref{e:e2}
\qed

\section{Appendix: Proof of Proposition~\ref{t:noncoal}}

From the discussion following the statement of Proposition~\ref{t:noncoal} we may
assume $d=2$ throughout.  Recall the definitions and time
change construction from Section~5 (especially
Lemmas~\ref{lem:difftimechange} and
\ref{lem:killtime}), using the rate $2\rho|B_r|$ random walk 
$Y_t$, the difference process $\txi^x_t$, and absorption
time $\tilde\tau=\kappa$. 
By Lemma~\ref{lem:killtime}, if $x=x_1-x_2$,
\begin{equation}\label{e:taukappa}
P_x(\tilde\tau>t) =
P_x(\kappa>t) =
E_x\Big[ \exp(-\int_0^{I^{-1}(t)} k(Y_s)ds\Big)\Big].
\end{equation}
Thus, to prove \eqref{e:noncoal2d} in
Proposition~\ref{t:noncoal} 
we need to show existence and positivity of
\begin{equation}\label{e:kgoal}
\lim_{t\to\infty} (\log t) E_x
\Big[\exp\Big(-\int_0^{I^{-1}(t)}k(Y_s)ds\Big)\Big], \quad
x\ne 0. 
\end{equation}
As the exact form of the killing rate $k(x)$ will not be important
in our arguments, we will replace it with a general radial function
$\phi:\Rtwo\to[0,\infty]$  satisfying
\[
\phi(x)<\infty \text{ for }
x\ne 0, \ \phi(x) \text{ is }\downarrow  \text{ in $|x|$ and }
\phi(x) = 0 \text{ for }|x|>2r,   
\]
where $\downarrow$ in $|x|$ means non-increasing in $|x|$, and similarly for $\uparrow$.
We assume throughout that $\phi$ has these properties. Recall the stopping times $t_A$ and $T_A$ from \eqref{hittingdefn}.

\begin{proposition}\label{p:reduction}
Let $Y_0=x\in \Rtwo$, $|x|>2r$, or $Y_0=\bar U$. If there is
a constant $c_\phi(Y_0)>0$ such 
that the limit
\begin{equation}\label{e:reduction1}
c_\phi(Y_0)=\lim_{A\to\infty} \log(A^2) E_{Y_0}\Big[
\exp\Big(-\int_0^{T_A}\phi(Y_s)dx\Big)\Big]
\end{equation}
exists, then
\begin{equation}\label{e:reduction2}
\lim_{t\to\infty} (\log t) E_{Y_0}
\Big[\exp\Big(-\int_0^{I^{-1}(t)}\phi(Y_s)ds\Big)\Big]
= c_\phi(Y_0).
\end{equation}
\end{proposition}

\begin{proof}
Suppose \eqref{e:reduction1} holds. We first prove that this
implies
\begin{equation}\label{e:reduction3}
\lim_{t\to\infty} (\log t) E_{Y_0}
\Big[\exp\Big(-\int_0^{t}\phi(Y_s)ds\Big)\Big] 
= c_\phi(Y_0).
\end{equation}
If we let 
$\Gamma_A$ be the event $\{ A^2/\log A\le T_A\le
  A^2\log A\}$, then by Lemma~\ref{l:AlogA}, for $A>(2|Y_0|)\vee 2$,
$P_{Y_0}(\Gamma^c_A)\le C/(\log A)^2$. Thus
\begin{align*}
E_{Y_0}\Big[\exp\Big(-\int_0^{T_A}\phi(Y_s)ds\Big)\Big] &\le  
E_{Y_0}\Big[1_{\Gamma_A}\exp\Big(-\int_0^{T_A}\phi(Y_s)ds\Big)\Big] +
P_{Y_0}(\Gamma^c_A)\\
&\le E_{Y_0}\Big[\exp\Big(-\int_0^{A^2/\log A}\phi(Y_s)ds\Big)\Big] +
\frac{C}{(\log A)^2}.
\end{align*}
It follows that 
\begin{multline*}
\log \Big(\frac{A^2}{\log A}\Big)
E_{Y_0}\Big[\exp\Big(-\int_0^{A^2/\log A}\phi(Y_s)ds\Big)\Big]
\\\ge \log \Big(\frac{A^2}{\log A}\Big) \left(
E_{Y_0}\Big[\exp\Big(-\int_0^{T_A}\phi(Y_s)ds\Big)\Big]
-\frac{C}{(\log A)^2}\right)
\to c_\phi(Y_0) \text{ as }A\to\infty.
\end{multline*}
This proves 
\begin{equation}\label{e:liminf1}
\liminf_{t\to\infty}(\log t) E_{Y_0}\Big[
\exp\Big(-\int_0^{t}\phi(Y_s)dx\Big)\Big] \ge c_{\phi}(Y_0).
\end{equation} Similarly,
\begin{align*}
E_{Y_0}\Big[\exp\Big(-\int_0^{T_A}\phi(Y_s)ds\Big)\Big] &\ge  
E_{Y_0}\Big[1_{\Gamma_A}\exp\Big(-\int_0^{T_A}\phi(Y_s)ds\Big)\Big]\\
&\ge E_{Y_0}\Big[\exp\Big(-\int_0^{A^2\log A}\phi(Y_s)ds\Big)\Big] -
P_{Y_0}(\Gamma^c_A). 
\end{align*}
It follows that 
\begin{equation}\label{e:Atbnd}
\begin{aligned}[b]
\log \Big(A^2\log A\Big)
E_{Y_0}\Big[\exp&\Big(-\int_0^{A^2\log A}\phi(Y_s)ds\Big)\Big]\\
&\le \log \Big(A^2\log A\Big) \left( 
  E_{Y_0}\Big[\exp\Big(-\int_0^{T_A}\phi(Y_s)ds\Big)\Big]
 +\frac{C}{(\log A)^2}\right)\\  
&\to c_\phi(Y_0) \text{ as }A\to\infty.
\end{aligned}
\end{equation}
Along with \eqref{e:liminf1}, this  proves
\eqref{e:reduction3}. 
 
We can now prove \eqref{e:reduction2}. 
By \eqref{e:taukappa}, \eqref{e:reduction3} and $I^{-1}(t) \le t$,
\[
\liminf_{t\to\infty}
(\log t) E_x
\Big[\exp\Big(-\int_0^{I^{-1}(t)}\phi(Y_s)ds\Big)\Big]
\ge
\lim_{t\to\infty} (\log t) E_{Y_0}
\Big[\exp\Big(-\int_0^{t}\phi(Y_s)ds\Big)\Big]  = c_\phi(Y_0).
\]
On the other hand, by
Lemma~\ref{l:Ibnd}, taking
$\alpha=1/2$,
\begin{align*}
(\log t) &E_x
\Big[\exp\Big(-\int_0^{I^{-1}(t)}\phi(Y_s)ds\Big)\Big]
\le (\log t) E_{Y_0}
\Big[1\{I^{-1}(t)\ge t-t^{1/2}\}\exp
\Big(-\int_0^{I^{-1}(t)}\phi(Y_s)ds\Big)\Big]\\
&\qquad + P_{Y_0}(I^{-1}(t)< t-t^{1/2})\\
&\le (\log t) E_{Y_0}
\Big[\exp\Big(-\int_0^{t-\sqrt t}\phi(Y_s)ds\Big)\Big]
+
(\log t) C_{\ref{l:Ibnd}}\frac{\log(1+t)}{\sqrt t}\\
&\to c_\phi(Y_0) \text{ as }t\to\infty
\end{align*}
by \eqref{e:reduction3}. 
Along with the previous $\liminf$ bound this proves
\eqref{e:reduction2}. 
\end{proof}

To prove \eqref{e:reduction1} we first establish a
number of properties of 
\begin{equation}\label{e:PhiDef}
  \Phi(x,A) =
  E_x\Big[\exp\Big(-\int_0^{T_A}\phi(Y_s)ds\Big)\Big] .
\end{equation}
It is elementary that $0\le \Phi(x,A)\le 1$
and that by recurrence, $\Phi(x,A)\to 0$ as $A\to\infty$. 
The next two results will show that $\Phi(x,A)$ is increasing in
$|x|$. 

\begin{lemma}\label{l:monoup} Let $N\in\nn$. 
  If $0=s_0<s_1<\cdots<s_N$ and
$f:\R^{N+1}\to \R$ is bounded and $\uparrow$ in each coordinate
then 
\begin{equation}\label{e:monoup}
y \mapsto E_{y}[ f(|Y_{s_0}|,\dots,|Y_{s_N}|)]
\text{ is $\uparrow$ in $|y|$.}
\end{equation}
\end{lemma}
\begin{proof}
Let $U,U_1,U_2,\dots$ be iid rv's uniform on $B_r$, and let
$S_m=U_1+\dots+U_m$. The first step is to prove that if $N=1$
then for $m=1,2\dots$, 
\begin{equation}\label{e:g1}  E[ f(|y|,|y+S_m|)
  ]\text{ is increasing in }|y|.
\end{equation}
Let $u\ge 0$, and define 
\begin{equation}\label{e:hu}
  h_u(y) = P(|y+U|\le u) = \frac{|B_r\cap B_u(-y)|}{|B_r|}
= \frac{|B_r\cap B_u(y)|}{|B_r|}.
\end{equation}
It is easy to see that $h_u(y)$ is decreasing in $|y|$, so that $|y+S_1|$ is
stochastically increasing in $|y|$, which proves \eqref{e:g1} for
$m=1$. Now suppose $m=2$, and consider
\begin{align*}
  P(|y+S_2|\le u ) =
  P(|y+S_1+U_2|\le u ) = E[h_u(y+S_1)].
\end{align*}
Clearly $h_u(y)$ depends only on $|y|$, and having
established it is decreasing in $|y|$, the 
$m=1$ case of \eqref{e:g1} implies that
$E[h_u(y+S_2)]$
is decreasing in $|y|$, which shows $|y+S_2|$ is
stochastically increasing in $|y|$, proving 
\eqref{e:g1} for $m=2$. The general inductive step for  
\eqref{e:g1} is similar.

Consider next, the $N=1$ case of \eqref{e:monoup}. With $\lambda
= 2\rho|B_r|$, 
\[
  E_{y_0}[f(|Y_{s_0}|,|Y_{s_1}|)] = 
  \sum_{m=0}^\infty e^{-\lambda s_1}\frac{(\lambda
    s_1)^{m}}{m!} E[f(|y_0|,|y_0+S_{2m}|)]
    \]
By \eqref{e:g1}, this shows $E_{y_0}[f(|Y_{s_0}|,|Y_{s_1}|)]$
is increasing in $|y_0|$, proving \eqref{e:monoup} for $N=1$.
Now suppose $N>1$, let $(Y'_t)$ under $P'$ be an independent copy of
$(Y_t)$, and define $\tilde f:\R^N\to\R$ by 
\begin{equation}\label{e:mono1}
  \tilde f(|y_0|,|y_1|,\dots,|y_{N-1}|) =
E'_{y_{N-1}}[f(|y_0|,\dots,|y_{N-1}|,|Y'_{s_N-s_{N-1}}|)].
\end{equation}
Then $\tilde f$ is increasing in $|y_0|,\dots,|y_{N-2}|$ by
definition, and is increasing in $|y_{N-1}|$ by the $N=1$
case of \eqref{e:monoup} just established. 
By the Markov property applied at time $s_{N-1}$, 
\begin{align*}
E_{y_0}[f(|Y_{s_0}|,\dots,|Y_{s_{N-1}}|,|Y_{s_N}|)] &=
E_{y_0}\Big[E'_{Y_{s_{N-1}}}[f(|Y_{s_0}|,\dots,|Y_{s_{N-1}}|,
|Y'_{s_N-s_{N-1}}|)]\Big]\\
& =
E_{y_0}[\tilde f(|Y_{s_0}|,\dots, |Y_{s_{N-1}}|)].
\end{align*}
This provides the inductive step to complete the
proof of \eqref{e:monoup} for general $N$.

\end{proof} 

\begin{lemma}\label{l:monoup2} Let $g:(0,\infty)\to [0,\infty)$ be
continuous and $\downarrow$. Then for all
$A,t\in[0,\infty]$, 
\begin{equation}\label{e:monoup2}
y \mapsto E_{y}\Big[\exp\Big(-\int_0^{T_A\wedge
  t}g(|Y_s|) ds\Big)\Big] \text{ is $\uparrow$ in } |y| .
\end{equation}
In particular, $x\to \Phi(x,A)$ is increasing in $|x|$. 
\end{lemma}

\begin{proof} By monotone convergence, we may assume $A,t$
  are finite and $g$ is bounded. Let $g(0)=\lim_{s\downarrow
    0}g(s)$. For $N\in \nn$ let 
  $M_N\in\nn$ and $0=s^N_0<s^N_1<\dots<s^N_{M_N}=t$ satisfy
  $s^N_{i+1}-s^N_i<2^{-N}$ for $0\le i<M_N$, and define
  \[
\tau^N = \min\{s^N_i : |Y_{s^N_i}|>A\}\wedge t . 
    \]
By right-continuity of $|Y_s|$, $\tau^N\downarrow T_A$ a.s.\
as $N\to\infty$. By continuity of $g$ on $[0,\infty)$ and
dominated convergence,
\begin{align*}
E_{y}\Big[\exp\Big(-\int_0^{T_A\wedge
  t}g(|Y_s|) ds\Big)\Big]  & =
\lim_{N\to\infty} E_{y}\Big[\exp\Big(-
\sum_{i=1}^{M_N-1}1\{s^N_i<\tau^N\} g(|Y_{s^N_{i}}|)(s^N_{i+1}-s^N_{i})\Big)
\Big]\\
& =
\lim_{N\to\infty} E_{y}\Big[\prod_{i=0}^{M_N-1}
G^N_i(|Y_{s^N_0}|,\dots,|Y_{s^N_i}|)
\Big]
\end{align*}
where
\[
  G^N_i(|Y_{s^N_0}|,\dots,|Y_{s^N_i}|) =
  \exp\Big(-1\{s^N_i<\tau^N\} g(|Y_{s^N_{i}}|)(s^N_{i+1}-s^N_{i})\Big).
\]
It is easy to check that $G^N_i$ is increasing in each
of its variables, and hence applying Lemma~\ref{l:monoup}
to their product,  $E_{y}\Big[\prod_{i=0}^{M_N-1}
G^N_i(|Y_{s_0}|,\dots,|Y_{s_i}|)\Big]$ is increasing in
$|y|$. The result \eqref{e:monoup2} now follows from the above.
\end{proof}

A consequence of the strong Markov property we 
will use repeatedly is
\begin{equation}\label{e:PhiMarkov}
\Phi(x,A) = P_x(T_A<t_a) + E_x\Big[ 1\{t_a<T_A\}\Phi(
Y_{t_a},A)\Big] \text{ if }
2r<a<|x|<A .
\end{equation}

\begin{lemma}\label{l:logAbnd}
There exists $C_{\ref{e:logAbnd}}=C_{\ref{e:logAbnd}}(r)
>1$ such that for all $k\ge
2$ and $0<|x|\le k<A$, 
\begin{equation}\label{e:logAbnd}
 \Phi(x,A)\ \le C_{\ref{e:logAbnd}}\frac{\log k}{\log (A^2)}.
\end{equation}
\end{lemma}
\begin{proof} By the monotonicity in
  Lemma~\ref{l:monoup2}, it suffices to prove \eqref{e:logAbnd}
  for $x=x_k=(k,0)$. Assume additionally $k> 6r\vee r^{-1}$
 and $A>r^2$. 
  By \eqref{e:PhiMarkov}, $|Y_{t_{3r}}|\le 3r$,  and monotonicity, we have 
\begin{equation}\label{e:logAbnd1}
  \begin{aligned}
  \Phi(x_k,A)& = P_{x_k}(T_A<t_{3r}) + 
E_{x_k}\Big[1\{t_{3r}<T_A\} \Phi(Y_{t_{3r}},A)\Big]\\
&\le  P_{x_k}(T_A<t_{3r}) +   P_{x_k}(t_{3r}<T_A)
\Phi((3r,0),A) .
\end{aligned}
\end{equation}
Using the strong Markov property at time $T_{4r}$, and
noting $|Y_{T_{4r}}|\le 6r\le |x_k|<A$, we again have from Lemma~\ref{l:monoup2},
\[ \Phi((3r,0),A) = 
E_{(3r,0)}\left[ \exp\Big( 
-\int_0^{T_{4r}}\phi(Y_s)ds\Big) \Phi(Y_{T_{4r}},A)
\right] \le \alpha(r)\Phi(x_k,A),
\]
where we have set $\alpha(r)= \Phi(x_{3r},4r)<1$. Insert this 
into \eqref{e:logAbnd1} and rearrange to conclude
\begin{equation}\label{e:logAbnd2}
\Phi(x_k,A) \le \frac{P_{x_k}(T_A<t_{3r})}{1-\alpha(r)} .
\end{equation}
By Lemma~\ref{l:logprob}, taking $a=3r$,
\begin{equation}
  P_{x_k}(T_A<t_{3r}) \le  \frac{\log(k/r)}{\log(A/r)}\le
  \frac{\log(k^2)}{\log(\sqrt A)}
  =  \frac{8\log k}{\log(A^2)}
\end{equation}
where the second inequality uses $k>1/r$ and $A>r^2$.
In view of 
\eqref{e:logAbnd2}, letting $C=8/(1-\alpha(r))$, we now have
\begin{equation}\label{e:logAbnd3}
  \Phi(x_k,A)\le C
  \frac{\log k}{\log (A^2)} 
\end{equation}
for all $k>6r\vee r^{-1}\vee 2$ and $A>k\vee r^2$. It is easy to see that 
$C$ can be increased so that \eqref{e:logAbnd3} will hold
for all
$k\ge 2$ and $A> k$, completing the proof of \eqref{e:logAbnd}.  
\end{proof}

We will construct a  coupling of the random
walks $Y_t$ started at $x'\ne x$ in order to obtain good bounds
on the difference $\Phi(x',A)-\Phi(x,A)$. We start in 
discrete time.  Let $\{U_i\}$ be iid r.v.'s which are
uniformly distributed over $B_r$, and for
$x'\in H_r=\{(x_1,x_2):x_1>0\}$ define
\[S^{x'}_n=x'+\sum_{i=1}^n U_i.\]
Let 
$\pi$ denote the reflection mapping
$\pi(x_1,x_2)=(-x_1,x_2)$ and set $x=\pi(x')$.
We will use a reflection coupling to define
$(S^{x}_n:n\ge 0)$. 
Let $H_\ell=\{(x_1,x_2):x_1\le 0\}$,
, and define
\[N_c=N_c^{x,x'}=\min\{n\ge 1:S^{x'}_n\in B_r(\pi(S^{x'}_{n-1}))\}.\]

\begin{lemma}\label{Ncbound}
$N_c\le N_\ell:=\min\{n\ge 0:S^{x'}_n\in H_\ell\}$ a.s., and so
$S^{x'}_n\in H_r$ for all $0\le n<N_c$ a.s. 
\end{lemma}
\begin{proof} $S^{x'}_{N_\ell}\in H_\ell$ and $\pi(S^{x'}_{N_\ell-1})\in H_\ell$ imply that 
\[|S^{x'}_{N_\ell}-\pi(S^{x'}_{N_\ell-1})|\le |\pi(S^{x'}_{N_\ell})-\pi(S^{x'}_{N_\ell-1})|=|S^{x'}_{N_\ell}-S^{x'}_{N_\ell-1}|<r.\]
The result follows.
\end{proof}

We now define $(S^{x}_n)_{n\ge 0}$ by 
\begin{equation}\label{e:sn-xdef}
  S_n^{x}=\begin{cases}\pi(S^{x'}_n)&\text{ if }n<N_c\\
S_n^{x'}&\text{ if }n\ge N_c.
\end{cases}
\end{equation}
Then $S^{x}_0=x$, and 
it follows from Lemma~\ref{Ncbound} that for $n<N_c$,
$S^{x'}_n$  is in $H_r$ and so $S_n^{x'}\neq S_n^{x}$, which
implies that  
\begin{equation}\label{Nccouple}
N_c=\min\{n\ge 0: S^{x'}_n=S^x_n\}.
\end{equation}
That is,  $N_c$ is the coupling time of $(S^x_n)$ and
$(S^{x'}_n)$. If we let $\calf^{S^{x'}}_n=\sigma(S^{x'}_m,m\le
n)$, then $N_c$ is an
$(\calf^{S^{x'}}_n)$-stopping time, 
and $S^{x}$ is $(\calf^{S^{x'}}_n)$-adapted. We next show that
$S_n^{x}$ is an $(\calf^{S^{x'}}_n)$-random walk starting at
$x$ with 
step distribution $U_1$, as the notation suggests.

\begin{lemma}\label{coupledwalks}
  For any Borel $A\subset\Rtwo$,
  $P(S_{n+1}^{x}\in A\mid\calf^{S^{x'}}_n)(\omega)=
  P(S^{x}_n(\omega)+U_{n+1}\in A)$ a.s.
\end{lemma}
\begin{proof} This is obvious on $\{N_c\le n\}$ (in $\calf^{S^{x'}}_n$) since then $S^{x}_n$ and $S^{x}_{n+1}$ equal $S^{x'}_n$ and $S^{x'}_{n+1}$, respectively. 
Suppose now that $N_c>n$, and define
$\hat B=\hat B(\omega)=B_r(S^{x'}_n)\cap B_r(\pi(S^{x'}_n))$, so that 
\begin{equation}\label{hatB1}
\pi(\hat B)=\hat B,
\end{equation}
and 
\begin{equation}\label{hatB2}
\hat B\subset B_r(S_n^{x}).
\end{equation}
This last inclusion holds because $S_n^{x}=S_n^{x'}$ or
$\pi(S^{x'}_n)$ for all $n$.

For simplicity we will write $\calf_n$ for
$\calf^{S^{x'}}_n$ in the rest of this proof.  By the definition
of $S^{x}_n$, 
\begin{align*}
&P(S^{x}_{n+1}\in A|\calf_n)1(N_c>n)\\
&=P(S^{x'}_{n+1}\in B_r(\pi(S_n^{x'})), N_c>n, S_n^{x'}+U_{n+1}\in A|\calf_n)\\
&\quad +P(S^{x'}_{n+1}\notin B_r(\pi(S_n^{x'})), N_c>n, \pi(S_{n+1}^{x'})\in A|\calf_n)\\
&=P(S^{x'}_{n+1}\in \hat B\cap A, N_c>n|\calf_n)+P(S^{x'}_{n+1}\notin\hat B, N_c>n, \pi(S^{x'}_{n+1})\in A|\calf_n)\\ 
&=P(\pi(S^{x'}_{n+1})\in\hat B\cap\pi(A),N_c>n|\calf_n)+P(\pi(S_{n+1}^{x'})\in \hat B^c\cap A,N_c>n)|\calf_n)\ \text{(by \eqref{hatB1})}\\
&=[P(S_n^{x}+\pi(U_{n+1})\in \hat B\cap\pi(A)|\calf_n)+P(S_n^{x}+\pi(U_{n+1})\in \hat B^c\cap A|\calf_n)]1(N_c>n).
\end{align*}
Next introduce the dependence on $\omega$ in the above, and
use the fact that, conditionally on $\calf_n$,
$S_n^{x}(\omega)+\pi(U_{n+1})$  is uniformly distributed over $B_r(S_n^{x}(\omega))$ to see that if $|C|$ is the Lebesgue measure of $C$, then the above evaluated at $\omega$ is a.s. equal to
\begin{align*}
&[|\pi(A)\cap\hat B(\omega)\cap B_r(S_n^{x}(\omega))|+|A\cap\hat B^c(\omega)\cap B_r(S_n^{x}(\omega))|]1(N_c(\omega)>n)/|B_r|\\
&=[|A\cap\hat B(\omega)|+|A\cap \hat B^c(\omega)\cap B_r(S_n^{x}(\omega))|] 1(N_c(\omega)>n)/|B_r| \quad \text{(by $|\pi(C)|=|C|$, \eqref{hatB1}, and \eqref{hatB2})}\\
&=[|A\cap\hat B(\omega)\cap B_r(S_n^{x}(\omega))|+|A\cap \hat B^c(\omega)\cap B_r(S_n^{x}(\omega))|] 1(N_c(\omega)>n)/|B_r| \quad \text{(by \eqref{hatB2})}\\
&=|A\cap B_r(S_n^{x}(\omega)|1(N_c(\omega)>n)/|B_r|\\
&=P(S_n^{x}(\omega)+U_{n+1}\in A)1(N_c(\omega)>n).
\end{align*}
The result follows.
\end{proof}

Let $S_n=(S^{(1)}_n,S^{(2)}_n)$ denote a copy of the
random walk starting at $x'$ under $P_{x'}$. 

\begin{lemma}\label{Scouplingrate} There is a constant
  $C_{\ref{Scouplingrate}}$ so that for all $x'$ in the
  positive $x_1$-axis and all $n\in\N$, 
  $$P(N^{x,x'}_c\ge n)\le
  \frac{C_{\ref{Scouplingrate}}}{\sqrt n}
\Big (1+ \frac{|x'|}{2r}\Big).$$
\end{lemma}
\begin{proof}
Use Lemma~\ref{Ncbound} and then the reflection principle to see that
\begin{align*}
P(N_c^{x,x'}\ge n)\le P_{x'}(N_\ell\ge n)
&=1-P_{x'}(N_\ell<n)\\
&=1-2P_{x'}(S_n^{(1)}<S_{N_\ell}^{(1)}, N_\ell<n)\\
&\le 1-2P_0(S_n^{(1)}<-|x'|-r)\\
&=P_0(|S_n^{(1)}|\le r+|x'|).
\end{align*}
The step distribution of $(S_n^{(1)})$ has
density 
$f(u)= 2\sqrt{r^2- u^2}/|B_r| \le 1/r$ on $[-r,r]$.  It follows from the 
$d=1$ version of \eqref{e:BRbnd} applied to random
variables with this distribution that for a constant $C=C(r)$,
\[
P_0(|S^{(1)}_n|\le r+|x'|) \le C\frac{2(r+|x'|)}{\sqrt n}\le\frac{4Cr}{\sqrt n}\Big(1+\frac{|x'|}{2r}\Big),
\]
so we are done.  
\end{proof}
We now use translation invariance to extend the above to
points $x,x'\in\{(x_1,0):x_1\ge 0\}$ such that $0\le |x|<|x'|$,
where now $\frac{x+x'}{2}=(m,0)$ plays the role of the origin,
$H^m_\ell=\{x:x_1\le m\}$, and $\pi^m$ is reflection in the
plane $\{x_1=m\}$.  So we define 
\begin{equation}\label{NcS2}
N^{x,x'}_c=N_c=\min\{n\ge 1:S^{x'}_n\in B_r(\pi^m(S_{n-1}^{x'}))\}\le N^{x'}_\ell=\min\{n\ge 0:S^{x'}_n\in H^m_\ell\},
\end{equation}
where the inequality is by Lemma~\ref{Ncbound}, and 
\begin{align} \label{Sndefn}S_n^{x}=\begin{cases}\pi^m(S^{x'}_n)&\text{ if }n<N_c\\
S_n^{x'}&\text{ if }n\ge N_c.
\end{cases}
\end{align}
The above results imply that both $S^x$ and $S^{x'}$ are $(\calf^{S^{x'}}_n)$-random walks with step distribution $U_1$, 
\begin{equation}\label{Scoupdef}N^{x,x'}_c=\min\{n\ge 0:S^x_n=S^{x'}_n\}
\end{equation}
is their coupling time, and 
\begin{equation}
\label{CoupleStail}
P(N_c^{x,x'}\ge n)\le \frac{C_{\ref{Scouplingrate}}}{\sqrt n}
\Bigl(1+\frac{|x'-x|}{2r}\Bigr).
\end{equation}

Next define coupled copies of the discrete time random walk with step distribution $U_1+U_2$ by 
$\hat Y_n^x=S_{2n}^x$ and $\hat Y_n^{x'}=S_{2n}^{x'}$, and also set $\hat\calf^{x'}_n=\calf^{S^{x'}}_{2n}$. 
We will write $\hat \calf_n$ for $\hat\calf^{x'}_n$ if there is no ambiguity.  Then it follows from Lemma~\ref{coupledwalks} that both $\hat Y^x_n$ and $\hat Y^{x'}_n$ are $(\hat\calf_n)$-random walks with step distribution $U_1+U_2$, that is, they are $(\hat\calf_n)$-adapted and 
\begin{equation}\label{hatmarkov}
P(\hat Y^x_{n+1}\in A|\hat \calf_n)(\omega)=P(\hat Y^x_n(\omega)+U_1+U_2\in A)\ \text{a.s.},
\end{equation}
and similarly for $\hat Y^{x'}$.  It follows easily from \eqref{Scoupdef} that 
\begin{equation}\label{hatNdefn}
\hat N_c^{x,x'}:=\min\{n\ge 0:\hat Y^x_n=\hat Y^{x'}_n\}=\Bigl\lceil{\frac{N_c^{x,x'}}{2}}\Bigr\rceil.
\end{equation}
Next use \eqref{Sndefn} and the fact that $2n\ge N_c^{x,x'}$ iff $n\ge \Bigl\lceil{\frac{N_c^{x,x'}}{2}}\Bigr\rceil$
to conclude that
\begin{align}\label{hatYdefn}\hat Y_n^{x}=\begin{cases}\pi^m(\hat Y^{x'}_n)&\text{ if }n<\hat N^{x,x'}_c\\
\hat Y_n^{x'}&\text{ if }n\ge \hat N^{x,x'}_c.
\end{cases}
\end{align}
Letting $\hat Y_n^{x',(1)}$ be the first coordinate of $\hat
Y_n^{x'}$,
define
\[\hat N^{x'}(m)=\min\{n\ge 0:\hat Y_n^{x',(1)}\le m\}.\]
\begin{lemma} \label{hatNcbound1} With $x,x',m$ as above,
$\hat N_c^{x,x'}\le \hat N^{x'}(m)$ a.s. 
\end{lemma}
\begin{proof} It follows from Lemma~\ref{Ncbound} that for
  all $n$, $S_{2n}^x\neq S_{2n}^{x'}$ implies  
that $S^{(1),x'}_{2n}>m$. This shows that $n<\hat N_c^{x,x'}$ implies $n<\hat N^{x'}(m)$ which clearly gives 
the required result.
\end{proof}
\begin{lemma} \label{Ycouplingrate} With $x,x',m$ as above,
  and for all $n\in\N$,
$$P(\hat N_c^{x,x'}\ge n)\le C_{\ref{Scouplingrate}} n^{-1/2}\Bigl(1+\frac{|x'-x|}{2r}\Bigr).$$
\end{lemma}
\begin{proof} By \eqref{hatNdefn} 
\[P(\hat N_c^{x,x'}\ge n)=P\Bigl(\Bigl\lceil{\frac{N_c^{x,x'}}{2}}\Bigr\rceil\ge n\Bigr)\le P(N_c^{x,x'}\ge 2n-1).\]
The result follows from \eqref{CoupleStail}.
\end{proof}
We move now to the continuous time random walks. Let
$N(t)$ be an independent Poisson process with rate
$\lambda=2\rho|B_r|$ and jump time sequence $(s_n)_{n\in\Z_+}$, i.e., 
$s_n=\inf\{t\ge 0:N_t=n\}$. For $K>0$ 
put $x'=(K+2r,0)$,
and let $x\in [K,K+2r)\times\{0\}$.
Define coupled
continuous time rate $\lambda$ random walks with step
distribution $U_1+U_2$, starting at $x'$ and $x$,
respectively, by
\[Y^{x'}_t=\hat Y^{x'}_{N_t}= \hat Y_n^{x'}\text{ if }s_n\le t<s_{n+1},\]
and 
\[Y^{x}_t=\hat Y^{x}_{N_t}= \hat Y_n^{x}\text{ if }s_n\le t<s_{n+1}.\]
The coupling time of these random walks is 
\begin{align}\label{Ycouplingdefn}
S_c^{x,x'}:=\inf\{t\ge 0:Y^{x}_t=Y^{x'}_t\}&
=\inf\{t\ge 0:\hat Y^{x}_{N_t}=\hat Y_{N_t}^{x'}\}\\
\nonumber&=\inf\{t\ge 0:N_t=\hat N_c^{x,x'}\}=s_{\hat N_c^{x,x'}}.
\end{align}
Note that $t<S_c^{x,x'}=s_{\hat N_c^{x,x'}}$ iff $N_t<\hat
N_c^{x,x'}$, and so by setting $n=N_t$ in \eqref{hatYdefn},
we have
\begin{align}\label{Ydefinition}
Y_t^{x}=\begin{cases}\pi^m(Y^{x'}_t)&\text{ if }t<S^{x,x'}_c\\
Y_t^{x'}&\text{ if }t\ge S^{x,x'}_c.
\end{cases}
\end{align}
Let $\calf_t$ be the right-continuous filtration generated
by $(Y^{x},Y^{x'},N)$, and  let $Y_t$ (respectively $\hat Y_n$) denote a generic
rate $\lambda$ continuous time (respectively, discrete
time) random walk  with step distribution $U_1+U_2$, 
 starting at $0$ under $P_0$.

\begin{lemma}\label{conttimeSMP}
(a) Both $Y^{x}$ and $Y^{x'}$ are rate $\lambda$ continuous
time $(\calf_t)$-random walks (and $(\calf_t)$-strong Markov
processes) with jump distribution $U_1+U_2$.  That is for
$y=x$ or $x'$, $t>0$, and any a.s. finite
$(\calf_t)$-stopping time $S$,  
\begin{equation}\label{YSMP}P(Y^y_{S+t}\in
  A|\calf_S)(\omega)=P_0(Y^y_S(\omega)+Y_t\in A)\ a.s.\text{
    for any Borel }A\subset\R^2.\end{equation} 
 (b) $S^{x,x'}_c$ is an $(\calf_t)$-stopping time.
 \end{lemma}
 \begin{proof} (b) is obvious from the definition of $S_c^{x,x'}$.\\
  (a) This is an easy and standard consequence of \eqref{hatmarkov}.
 \end{proof}

For $y=x$ or $x'$ and $2r\le \delta<A$ we let
\[t^y_\delta=\inf\{t\ge 0:|Y^y_s|\le \delta\}, \quad
  T_A^y=\inf\{t\ge 0: |Y^y_t|\ge A\},\] 
and also set
\[t^{x,x'}_\delta=t^{x}_\delta\wedge t^{x'}_\delta,\quad
  T^{x,x'}_A=T^{x}_A\wedge T^{x'}_A.\]
We define $\hat t^{y}_\delta, \hat T_A^y, \hat t^{x,x'}_\delta$,
and $\hat T^{x,x'}_A$ in a similar way, using the discrete time
random walks $\hat Y^x,\hat Y^{x'}$, for example, 
\[\hat t_\delta^y=\min\{n\ge 0:|\hat Y^y_n|\le \delta\}.\]

\begin{lemma}\label{l:coupbnd}
  Let $K>3r$ and $x'=(K+2r,0)$.
 
\noindent (a) For all $x\in[K,K+2r)\times\{0\}$, $t^{x'}_{3r}\ge S_c^{x,x'}\vee
t^{x}_{3r}$.

\medskip\noindent
(b) For all $x\in[K,K+2r)\times\{0\}$, if
$t^{x}_{3r}<S_c^{x,x'}$, then $|Y^{x'}_{t_{3r}^{x}}|\le
2K+10r$.

\medskip\noindent
(c) 
For $\veps\in(0,1)$ and $\delta\ge 3r$ there is a constant
$C_{\ref{e:coupbnd}}=C_{\ref{e:coupbnd}}(\delta,\veps)>0$
such that for all $K>\delta\vee 1$,
\begin{equation}\label{e:coupbnd}
\inf_{x\in[K,K+2r)\times\{0\}}P(S^{x,x'}_c < t^{x,x'}_\delta \wedge T^{x,x'}_{2K})
\ge 1- \frac{C_{\ref{e:coupbnd}}}{K^{1-\vep}} .
\end{equation}
\end{lemma}
\begin{proof} (a) First consider the discrete time
  walks. For $x\in[K,K+2r)\times\{0\}$, we have \break
$m=\frac{|x+x'|}{2}>K>3r$. This shows that
$B_{3r}\subset \{x_1<m\}$ and so $\hat Y^{x'}$ must first
enter $\{x_1\le m\}$ before it can enter $B_{3r}$. That is,
$\hat t_{3r}^{x'}\ge \hat N^{x'}(m)$.  Therefore by
Lemma~\ref{hatNcbound1},
$\hat N_c^{x,x'}\le \hat t_{3r}^{x'}$ and hence
\[S_c^{x,x'}=s_{\hat N_c^{x,x'}}\le s_{\hat
    t_{3r}^{x'}}=t_{3r}^{x'}.\]
Since the random walks must couple before $Y^{x'}$ can enter
$B_{3r}$, we also have $t^{x}_{3r}\le
t^{x'}_{3r}$ and (a) follows. 

\medskip\noindent (b) Let $x$ and $m$ be as in (a). 
If 
$t^{x}_{3r}<S_c^{x,x'}$, then by the coupling definition
\eqref{Ydefinition},
\begin{equation}\label{refprop}
Y^{x}_{t^{x}_{3r}}=\pi^m(Y^{x'}_{t_{3r}^{x}}).
\end{equation} 
For any $a=(a_1,a_2)\in B_{3r}$, 
$|\pi^m(a)| = |(2m-a_1,a_2)| \le 2m+3r+3r \le 2K+10r$, so by
\eqref{refprop}, 
\[
|Y^{x'}_{t^{x}_{3r}}|= |\pi^m(Y^{x}_{t^{x}_{3r}})| \le 2K+10r.
\]
This proves (b). 

\medskip\noindent (c)  Let $\hat Y^{(1)}$ be the first
coordinate of $\hat Y$, and let $x,x',m$ be as above, with
$K>\delta\ge 3r$, so that $ \delta <|x|\le |x'| <2K$. Let $n=\lceil 
  K^{2-2\veps}\rceil$. Then, using Lemma~\ref{Ycouplingrate}
  for the second inequality and symmetry for the second to
  last inequality, we have
  \begin{equation}
    \begin{aligned}[b]
  P(\hat N_c^{x,x'}<\hat t^{x,x'}_\delta\wedge &\hat T^{x,x'}_{2K})\\
&\ge P(\hat N_c^{x,x'}\le n)-P(\hat t_\delta^{x'}\wedge \hat T_{2K}^{x'}\le n)-P(\hat t_\delta^{x}\wedge \hat T^{x}_{2K}\le n)\\ 
&\ge
1-\frac{C_{\ref{Scouplingrate}}}{\sqrt n}\Big(1+\frac{|x'-x|}{2r}\Big)-P_0(\max_{k\le
  n}|\hat Y_k|\ge (K-2r)\wedge(K+2r-\delta))\\
&\qquad-P_0(\max_{k\le n}|\hat Y_k|\ge
(K-2r)\wedge(K-\delta))\\ 
&\ge 1-\frac{2C_{\ref{Scouplingrate}}}{\sqrt{n}}-
2P_0(\max_{k\le n}|\hat Y_k|\ge K-\delta)\\
&\ge 1-\frac{C}{K^{1-\vep}}-4P_0\Bigl(\max_{k\le
  n}|\hat Y^{(1)}_k|\ge\frac{K-\delta}{\sqrt
  2}\Bigr)\\
&\ge 1-\frac{C}{K^{1-\vep}}-4P_0\Bigl(\max_{k\le
  n}|\hat Y^{(1)}_k|\ge K/2\Bigr), 
\end{aligned}\label{e:upbnd}
\end{equation}
provided $K$ is larger than some $K_0(\delta)>0$. 
We recall Theorem~21.1 in \cite{Burk73}, which in the present context
implies
\begin{equation}\label{e:Bu}
  E_0\Big[ \max_{k\le n}|\hat Y^{(1)}_k|^p\Big] \le c_p[(nE[|
\bar U^{(1)}|^2])^{p/2} + (2r)^p] \le C_p n^{p/2}
\end{equation}
for a constant $C_p>0$. By Markov's inequality,
\[
  P_0\big(\max_{k\le n}|\hat Y^{(1)}_k|\ge K/2\big)
\le \frac{C_pn^{p/2}}{(K/2)^p}
 \le \frac{C}{K^{p\vep}}.
  \]
If we take $p=p_0(\vep)$ large enough so that 
$K^{p\vep}>K^{1-\vep}$, substituting into \eqref{e:upbnd} we
obtain for a constant $C>0$ depending on $\vep$,  
\[
  P(\hat N_c^{x,x'}<\hat t^{x,x'}_\delta\wedge \hat
  T^{x,x'}_{2K})
\ge 1-\frac{C}{K^{1-\vep}},
\]
provided $K\ge K_0(\delta)$. Multiplication of $C$
by a large
enough constant depending on $\delta$ allows us to remove
the restriction $K>K_0(\delta)$.  That is, for some 
$C_{\ref{e:coupbnd}}(\delta,\vep)>0$, 
\begin{equation}\label{e:upbnd1}\inf_{x\in [K,K+2r)\times\{0\}}P(\hat
  N^{x,x'}_c<\hat t^{x,x'}_\delta\wedge
  \hat
  T^{x,x'}_{2K})\ge1-\frac{C_{\ref{e:coupbnd}}(\delta,\vep)}{K^{1-\veps}}
\end{equation}   
for all $K>\delta\vee 1$.

The corresponding inequality for
the continuous time walks follows at once. First, note that 
\[
t^{x}_\delta=\inf\{t:|\hat Y^{x}_{N_t}|\le \delta\}=
\inf\{t:N_t=\hat t^{x}_\delta\}=s_{\hat t^{x}_\delta},\]
and similarly for the other hitting times.  So in view of
\eqref{Ycouplingdefn},  for every $x\in[K,K+2r)\times\{0\}$,
\begin{align*}
P(T_c^{x,x'}<t^{x,x'}_\delta\wedge T^{x,x'}_{2K})
&=P(s_{\hat N_c^{x,x'}}<s_{\hat t^{x,x'}_\delta}
\wedge s_{\hat T^{x,x'}_{2K}})\\ 
&=P(\hat N_c^{x,x'}<
\hat t^{x,x'}_\delta\wedge T^{x,x'}_{2K}). 
\end{align*}
So (c) is now immediate from \eqref{e:upbnd1}.

\end{proof}

\begin{lemma}\label{l:incbound} There is a constant $C_{\ref{e:incbound}}=
C_{\ref{e:incbound}}(r,\phi)>0$ such that 
for $K>5r\vee2$, $A>2K+2r$, $x,x'\in [K,K+2r]\times \{0\}$  with $|x|\le
|x'|$, 
and $\vep_K= (\log K)/\sqrt K$,
\begin{equation}\label{e:incbound}
0\le \Phi(x',A) - \Phi(x,A)\le 
C_{\ref{e:incbound}}\frac{\vep_K}{\log(A^2)}.
\end{equation}
\end{lemma}
\begin{proof} The first inequality in \eqref{e:incbound}
follows from monotonicity (Lemma~\ref{l:monoup2}). For the second, it suffices to take $x'=(K+2r,0)$ and $x\in[K,K+2r)\times\{0\}$.
Recall the times $t^{x,x'}_\delta=t^x_\delta\wedge
t^{x'}_\delta$, $T^{x,x'}_A=T^x_A\wedge T^{x'}_A$, etc.\ for
the coupled walks $(Y^{x},Y^{x'})$, and write 
$S_c$  for the coupling time $S^{x,x'}_c$. By Lemma~\ref{l:coupbnd}(a),
\begin{equation}\label{Torder} t^{x,x'}_{3r}=t_{3r}^{x}.\end{equation} 
If $S_c\le
t_{3r}^{x}\wedge T^{x,x'}_A$, then 
\[
\int_0^{T_A^{x'}}\phi(Y_s^{x'})ds=
\int_0^{T_A^{x}}\phi(Y_s^{x})ds.
\]
As a consequence,
\begin{align*}
  \Delta(x,x',A) &:= \Phi(x',A)-\Phi(x,A)\\
  &=
  E\Big[ \exp\Big(-\int_0^{T^{x'}_A}\phi(Y^{x'}_s)ds\Big)
-  \exp\Big(-\int_0^{T^{x}_A}\phi(Y^{x}_s)ds\Big)
  \Big]\\
  &\le
  E\Big[ 1\{S_c>
t_{3r}^{x}\wedge T^{x,x'}_A\} \exp\Big(-\int_0^{T^{x'}_A}\phi(Y^{x'}_s)ds\Big)
  \Big].
\end{align*}

We now introduce
\begin{align*}
    \Delta_1(x,x',A)&=E\Big[1\{T^{x,x'}_A<S_c,T^{x,x'}_A\le t^{x}_{3r}\}\exp
    \Bigl(-\int_0^{T_A^{x'}}\phi(|Y_s^{x'}|)ds\Bigr)\Big],\\
    \Delta_2(x,x',A)&=E\Big[1\{t^{x}_{3r}<S_c,t^{x}_{3r}<T^{x,x'}_A \}\exp
    \Bigl(-\int_0^{T_A^{x'}}\phi(|Y_s^{x'}|)ds\Bigr)\Big],
  \end{align*}
so that $\Delta=  \Delta_1 +\Delta_2$, and bound
$\Delta_1,\Delta_2$ separately.
For $\Delta_1$, using \eqref{Torder} and $T_A^{x'}\neq
t_{3r}^{x'}$,
\begin{equation}\label{delta1}
    \Delta_1(x,x',A)
\le P(T_A^{x}<S_c\wedge t^{x}_{3r})+P(T_A^{x'}<S_c\wedge t_{3r}^{x'}).
\end{equation}
It suffices to
consider the first term, as the second follows in the same way.  
By the strong Markov property 
\begin{equation}\label{DeltaIub}
  P(T_A^{x}<S_c\wedge t_{3r}^{x})\le E[1\{T^{x}_{2K}<S_c\}
  P_{Y^{x}_{T^x_{2K}}}(T_A<t_{3r})].
\end{equation}
Now taking $a=3r$ in Lemma~\ref{l:logprob}, and noting that $3r<2K\le |Y^{x}_{T_{2K}^{x}}|\le 2K+2r<A$, 
we have
\begin{align*}
P_{Y^{x}_{2K}}(T_A<t_{3r})
\le \frac{\log(|Y^{x}_{T_{2K}^{x}}|)-\log r}{\log A-\log r}
\le\frac{\log\Bigl(\frac{2K}{r}+2\Bigr)}{\log (A/r)} \ a.s.
\end{align*}
Choose $K_0$ large enough so that
$K> K_0$ implies $K^2> \frac{2K}r + 2$. If, in addition we have $K> K_0$ and
$A>r^2$, then
\[
  P_{Y^{x}_{2K}}(T_A<t_{3r}) \le
  \frac{\log(K^2)}{\log \sqrt{A}} = \frac{8\log K}{\log
    (A^2)} .
\]
By replacing 8 with a sufficiently large constant $C$
we may drop the additional conditions $K>K_0$ and $A>r^2$,
and so obtain for all $K>5r\wedge 2$ and $A>2K+2r$, 
\begin{equation*}
P_{Y^x_{T_{2K}}}(T_A<T_{3r})\le C\frac{\log K}{ \log(A^2)}.
\end{equation*}
Plug this bound into \eqref{DeltaIub} and use the coupling bound
Lemma~\ref{l:coupbnd}(c) with $\delta=3r,\vep=1/2$ to obtain
\[
  P(T_A^{x}<S_c\wedge t^{x}_{3r})
  \le C\frac{\log K}{
    \log(A^2)}P(T_{2K}^{x}<S_c) 
\le
C_{\ref{e:coupbnd}}(3r,\tfrac12) C\frac{\log K}{\sqrt{K}
  \log(A^2)} .
\]
The above 
and \eqref{delta1} imply
\begin{equation}\label{delta1bnd}
\Delta_1(x,x',A)\le C\frac{\log K}{\sqrt K\log(A^2)}.
\end{equation}

Now consider $\Delta_2$.  Recalling from \eqref{Torder}
that $t_{3r}^{x}\le t_{3r}^{x'}$, $\Delta_2$ is bounded by
the sum of
\begin{align*}
  \Delta_{2a}(x,x',A)&=E\Big[1\{t^{x}_{3r}<S_c,t^{x}_{3r}\le
 t^{x'}_{3r}<T^{x'}_A \}\exp
    \Bigl(-\int_0^{T_A^{x'}}\phi(|Y_s^{x'}|)ds\Bigr)\Big],\\
    \Delta_{2b}(x,x',A)&=E\Big[1\{t^{x}_{3r}<S_c,t^{x}_{3r}<T^{x'}_A<t^{x'}_{3r} \}\exp
    \Bigl(-\int_0^{T_A^{x'}}\phi(|Y_s^{x'}|)ds\Bigr)\Big].
  \end{align*}
In $\Delta_{2a}$,  the event in the indicator function
belongs to $\calf_{t_{3r}^{x'}}$, and so by the strong
Markov property,
$|Y^{x'}_{t_{3r}^{x'}}|\le 5r\le K$, monotonicity from Lemma~\ref{l:monoup2}, 
the coupling bound 
\eqref{e:coupbnd} with $\vep=1/2$ and $\delta=3r$, and
Lemma~\ref{l:logAbnd},  
\begin{equation}
  \begin{aligned}[b]\label{delta2abnd}
\Delta_{2a}(x,x',A)&=E\left[1\{t_{3r}^{x}<S_c,t_{3r}^{x}\le
t^{x'}_{3r}<T_A^{x'}\}E_{Y^{x'}_{t_{3r}^{x'}}}\Big[\exp\Bigl(
-\int_0^{T_A}\phi(|Y_s|)ds\Bigr)\Big]\right]\\
&\le
P(t_{3r}^{x}<S_c)\Phi((K,0),A)\\
&\le C\frac{\log K}{\sqrt K\log (A^2)}.
\end{aligned}
\end{equation}

Finally, consider $\Delta_{2b}$.  Dropping the exponential
and applying the strong Markov property
to $Y^{x'}$ at time $t_{3r}^{x}$, we have
\begin{equation}
  \begin{aligned}[b]  \label{DelIIii}
    \Delta_{2b}(x,x',A)&\le
P(t_{3r}^{x}<S_c, t_{3r}^{x}\le
T_A^{x'}<t^{x'}_{3r})\\
&= E\Big[ 1\{t_{3r}^{x}<S_c\wedge T^{x'}_A\}
P_{Y^{x'}_{t^x_{3r}}}(T_A<t_{3r})\Big].
\end{aligned}
\end{equation}
By Lemma~\ref{l:coupbnd}(b),
on the event $\{t_{3r}^{x}<S_c\}$,
$ |Y^{x'}_{t^{x}_{3r}}|\le 2K+10r$. Let $K_0$ be large enough
so that $K>K_0$ implies $K^2>(2K/r)+10$, and assume
additionally that
$K>K_0$ and  $A>(2K+10r)\vee r^2$.
By the hitting probability bound \eqref{e:aA}, with $a=3r$
we see that if $|Y^{x'}_{t_{3r}^{x}}|>3r$, then
\begin{align*}
P_{Y^{x'}_{t_{3r}^{x}}}(T_A<t_{3r})\le
\frac{\log(\frac{2K}r+10)}{\log \frac Ar}
\le \frac{\log (K^2)}{\log(\sqrt{A})}
=\frac{8\log K}{\log(A^2)}.
\end{align*} 
The same bound holds if $|Y^{x'}_{T_{3r}^x}|\le
3r$ because then the left-hand side is zero. Now the
additional restrictions on $K,A$ can be dropped by replacing
8 with a larger constant $C$,
so we may conclude
that for $A,K$ as in the Lemma and on
$\{t_{3r}^{x}<S_c\}$,
\[P_{Y^{x'}_{t_{3r}^{x}}}(T_A<t_{3r})\le C\frac{\log K}{\log (A^2)}.\]
Insert this into \eqref{DelIIii}, apply
Lemma~\ref{l:coupbnd}(c) as before, to obtain
\begin{equation}\label{delta2bnd}
\Delta_{2b}(x,x',A)\le
P(t^{x}_{3r}<S_c)\,\frac{C\log K}{\log(A^2)}
\le C \frac{\log K}{\sqrt{K} \log(A^2)}.
\end{equation}
Combine \eqref{delta1bnd}, \eqref{delta2abnd} and
\eqref{delta2bnd} to complete the proof with $\veps_K$ as
claimed. 
\end{proof}

\begin{theorem}\label{t:Philim}
For all $x\ne 0$ there exists $c_\phi(x)>0$ such that
\begin{equation}\label{e:Philim}
\lim_{A\to\infty}  \log(A^2)\Phi(x,A) = c_\phi(x)
\end{equation}
\end{theorem}

\begin{proof} We may assume $x$ is on the positive  
$x_1$-axis, and for now that $|x|>3r$.  Assume \break 
$K>(|x|+2r)\vee 2$
and $A>2K+2r$, put $x_K=(K,0)$ and $\vep_K = (\log K)/\sqrt
K$ as in Lemma~\ref{l:incbound}.  By the strong Markov property,  
\begin{equation}\label{e:Philim0}
\Phi(x,A) = E_x\left[\exp\Big(-\int_0^{T_{K}}\phi(Y_s)ds
\Big) \Phi(Y_{T_K},A)\right] .
\end{equation}
By Lemma~\ref{l:incbound}, noting that $|Y_{T_{K}}|\in[K,K+2r]$,
\begin{equation}\label{e:PhikA}
|\Phi(Y_{T_k},A) - \Phi(x_K,A)| \le C_{\ref{e:incbound}}
\frac{\vep_K}{\log(A^2)}.
\end{equation}
Using this bound in \eqref{e:Philim0} we obtain
\begin{equation}\label{e:PhikA1}
\Phi(x,A) = \Phi(x,K)\Phi(x_K,A) + \cale_1
\end{equation}
where $|\cale_1| \le C \vep_K/\log(A^2)$.  
Now consider $\Phi(x_K,A)$. On account of 
$\phi(Y^{x_K}_s)=0$ for $s>t_{3r}$, \eqref{e:PhiMarkov} and the strong Markov
property, we see that
\begin{equation}\label{e:PhikA2}
\begin{aligned}[b]
\Phi(x_K,A) &= P_{x_K}(T_A<t_{3r}) + E_{x_K}\Big[ 1\{t_{3r}<T_A\}
E_{Y_{t_{3r}}}\Big[\exp\Big(- \int_{0}^{T_K}\phi(Y_s)ds\Big)
\Phi(Y_{T_K},A)\Big]\\
&= P_{x_K}(T_A<t_{3r}) + \Phi(x_K,A)E_{x_K}\Big[ 1\{t_{3r}<T_A\}
\Phi(Y_{t_{3r}},K) \Big] +\cale_2
\end{aligned}
\end{equation}
where $ |\cale_2| \le C\frac{\vep_K}{\log(A^2)}$
by \eqref{e:PhikA}. 
If we set 
\[\alpha(x_K,A) = E_{x_K}\big[ 1\{t_{3r}<T_A\}\Phi(
Y_{t_{3r}},K)\big]<1\] 
then \eqref{e:PhikA2} can be rearranged to obtain 
\begin{equation}\label{e:PhikA3}
\Phi(x_K,A) 
= \frac{P_{x_K}(T_{A}<t_{3r}) + \cale_2}{1-\alpha(x_K,A)} .
\end{equation}
Plugging this into \eqref{e:PhikA1} gives
\begin{align*}
\Phi(x,A) = 
\Phi(x,K)
\frac{P_{x_K}(T_A<t_{3r})+\cale_2}{1-\alpha(x_K,A)}
+\cale_1,
\end{align*} and  thus
\begin{equation}\label{e:PhikA4}
\log(A^2) \Phi(x,A) =
\frac{\Phi(x,K)}{1-\alpha(x_K,A)}\, \log(A^2)P_{x_K}(T_A<t_{3r})+ \cale_3
\end{equation}
where $|\cale_3| \le C\vep_K/(1-\alpha(x_K,A))$.

Now consider $\alpha(x_K,A)$. By recurrence,
$1\{t_{3r}<T_A\}\to 1$ $P_{x_K}$-a.s.\ as 
$A\to\infty$, which implies the limit 
\[
\alpha(x_K,\infty) := \lim_{A\to\infty}\alpha(x_K,A)
= E_{x_K}\big[\Phi(Y_{t_{3r}},K)\big]
\]
exists. By monotonicity (Lemma~\ref{l:monoup2}), $|Y_{t_{3r}}|\le
3r$, and recurrence, we have $\alpha(x_K,\infty)\le \Phi((3r,0),K) \to 0$ as
$K\to\infty$. If we let $K_0$ satisfy 
$\alpha(x_K,\infty)\le 1/2$ for all $K\ge K_0$,
then $|\cale_3|\le 2C\vep_K$.  Thus, assuming in addition that 
$K>K_0$, we have from \eqref{e:PhikA4} and the above,
\begin{equation}\label{e:PhikA5}
\limsup_{A\to\infty}\left|
\log(A^2) \Phi(x,A) -
\Phi(x,K)\frac{\log(A^2)P_{x_K}(T_A<t_{3r})}
{1-\alpha(x_K,A)}
\right| \le C\vep_K .
\end{equation}
By \eqref{e:Aalimit}, $(\log A^2)P_{x_K}(T_A<t_{3r})$ is
bounded above for large $A$ and approaches $ p(3r,x_K):=
2\large(\log|x_K|-E_{x_K}[\log|Y_{t_{3r}}|]\large)>0$
as $A\to\infty$. It follows from \eqref{e:PhikA5} that 
\begin{equation}\label{e:PhikAlim}
\limsup_{A\to\infty}\left|
\log(A^2) \Phi(x,A) -
\Phi(x,K)\frac{p(3r,x_K)}{1-\alpha(x_K,\infty)}
\right| \le C\vep_K.
\end{equation}
The fact that $\vep_K\to 0$ as $K\to\infty$ now implies
$A\mapsto \log(A^2)\Phi(x,A)$ as $A\to\infty$ is 
Cauchy, hence there exists $c_\phi(x)\in[0,\infty)$ such that
$\lim_{A\to\infty} \log(A^2) \Phi(x,A)=c_\phi (x)$.

To check that $c_\phi(x)>0$, note that (for $K$ large as above) \eqref{e:PhikAlim} implies
\[
\liminf_{A\to\infty} \log(A^2)\Phi(x,A) \ge
p(3r,x_K)\frac{\Phi(x,K)}{1-\alpha(x_K,\infty)} - C\vep_K
\ge
p(3r,x_K)\Phi(x,K) - C\vep_K .
\]
The right-hand side above is positive if, for large $K$,
 $\Phi(x,K)/\vep_{K}\to\infty$ as $K\to\infty$ (note that $p(3r,x_K)\to\infty$ as $K\to\infty$).
Using
$\Phi(x,K)\ge P_x(T_K<t_{3r})$, we have 
\[
\frac{\Phi(x,K)}{\vep_K} \ge \frac{\sqrt K}{\log K}
P_x(T_K<t_{3r}) = \frac{\sqrt K}{(\log K)^2}
\log(K)P_x(T_K<t_{3r}) 
\to\infty\text{ as }K\to\infty
\]
by Lemma~\ref{l:logprob}. Hence, $c_\phi(x)>0$. 

Finally, suppose $0<|x|\le 3r$. By the strong Markov
property, for $A>6r$,
\[
\Phi(x,A) = E_x\Big[\exp(-\int_0^{T_{4r}}\phi(Y_s)ds\Big)
\Phi(Y_{T_{4r}},A)\Big]
\]
By monotonicity (Lemma~\ref{l:monoup2}) and Lemma~\ref{l:logAbnd}, 
$\log(A^2)\Phi(Y_{T_{4r}},A)\le \log(A^2)\Phi((6r\vee2,0),A)\le 
C_{\ref{e:logAbnd}}\log(6r\vee 2)$,
and also converges to $c_\phi(Y_{T_{4r}})>0$ a.s. as $A\to\infty$, by the above and $|Y_{T_{4r}}|>3r$.  Apply bounded
convergence to obtain existence of the limit
\[
\lim_{A\to\infty}\log(A^2)\Phi(x,A)= 
E_x\Big[\exp\Big(-\int_0^{T_{4r}}\phi(Y_s)ds\Big)
c_\phi(Y_{T_{4r}})\Big]>0 .
\]
This completes the proof of \eqref{e:Philim}.
\end{proof}

\begin{proof}[Proof of Proposition~\ref{t:noncoal} ($d=2$)]
Let $\phi(x)=k(x)$, so $\Phi(x,A) = E_x\Big[
\exp\Big(-\int_0^{T_A}k(Y_s)ds\Big)\Big] $. 
By Theorem~\ref{t:Philim}, for $x\ne 0$,
the positive limit $c_k(x)=\lim_{A\to\infty}\log(A^2)\Phi(x,A)$
exists. By Lemma~\ref{l:logAbnd}, 
there is a constant $C>0$ such that for all large $A$, 
$(\log A^2)\Phi(x,A)\le C$ for all $|x|\le 2r$. Therefore,
by bounded convergence, if $Y_0=\bar U$, the limit
\begin{equation}
\lim_{A\to\infty}(\log (A^2)) E_{Y_0}\Big[\exp\Big(-\int_0^{T_A}
k(Y_s)ds\Big)\Big]=\lim_{A\to\infty} E(\log(A^2)\Phi(Y_0,A))=E(c_k(Y_0))>0
\end{equation}
exists. That is, \eqref{e:reduction1} holds for $Y_0=\bar
U$, and by Proposition~\ref{p:reduction}, the limit
\[
\lim_{t\to\infty} (\log t) P_{Y_0}(\kappa>t) = E(c_k(Y_0))>0
\]
exists, which is exactly \eqref{e:noncoal2d}. 
  \end{proof}

 \begin{remark} Theorem~\ref{t:Philim} is proved for
continuous time rate $2\rho|B_r|$ random walks $Y_t$ with
step distribution \eqref{e:barUdens}. The result is needed
in our proof of Proposition~\ref{t:noncoal} because the dual particle
difference $\txi_t$ behaves like $Y_t$ when $|\txi_t|>2r$
in our fixed radius case of the SLFV. With a view to the variable radius case 
in \eqref{varradcase}, we remark that 
Theorem~\ref{t:Philim} holds for more
general radially symmetric step distributions. Assume
$r_{\text{max}}>0$ and suppose $\nu$ is a probability measure on
$[0,r_{\text{max}}]$ with a bounded density such that
$\int_0^{r_{\text{max}}}\rho^{-2}\nu(d\rho)<\infty$, and consider the
step distribution on $\Rtwo$ given by
\begin{equation}\label{e:newrep}
\int_0^{r_{\text{max}}} \nu(d\rho) P(\rho U\in\cdot),
\end{equation}
where $U$ is uniform on $B_1(0)$. A bit of calculus shows that
in the variable radius case \eqref{varradcase}, $\tilde \xi_t$ behaves like a fixed rate random walk with step distribution satisfying \eqref{e:newrep} for appropriate $\nu$ when $|\tilde\xi_t|>2r_{\text{max}}$. Theorem~\ref{t:Philim} holds 
when  $Y_t$ is a continuous
time rate $\lambda>0$ with this step distribution
\eqref{e:newrep}.  
The  modifications needed to prove this are minor. 
Of course this will not immediately give Proposition~\ref{t:noncoal} as we no longer have an identity like \eqref{e:taukappa} which arose from our time-change representation for the difference of the dual components.
\end{remark}

\noindent{\bf Acknowledgement.} We thank Sarah Penington for answering our queries on \cite{AFPS}. Some
of this work was carried out while the first author was visiting the Mathematics Department at the University of British Columbia, or while the second author was visiting the Mathematics Department at Syracuse University. We thank both Institutions for their hospitality.

\end{document}